\documentclass[final,3p,times]{elsarticle}  
\usepackage{amssymb,amsmath,amsthm,bm,graphicx,color}
\usepackage{epstopdf}   

\newcommand{\bi}{\begin{itemize}}
\newcommand{\ei}{\end{itemize}}
\newcommand{\ben}{\begin{enumerate}}
\newcommand{\een}{\end{enumerate}}
\newcommand{\be}{\begin{equation}}
\newcommand{\ee}{\end{equation}}
\newcommand{\bea}{\begin{eqnarray}} 
\newcommand{\eea}{\end{eqnarray}}
\newcommand{\ba}{\begin{align}} 
\newcommand{\ea}{\end{align}}
\newcommand{\bse}{\begin{subequations}} 
\newcommand{\ese}{\end{subequations}}
\newcommand{\bc}{\begin{center}}
\newcommand{\ec}{\end{center}}
\newcommand{\bfi}{\begin{figure}}
\newcommand{\efi}{\end{figure}}
\newcommand{\ca}[2]{\caption{#1 \label{#2}}}
\newcommand{\ig}[2]{\includegraphics[#1]{#2}}
\newcommand{\bmp}[1]{\begin{minipage}{#1}}
\newcommand{\emp}{\end{minipage}}
\newcommand{\pig}[2]{\bmp{#1}\includegraphics[width=#1]{#2}\emp}
\newcommand{\bigO}{{\mathcal O}}
\newcommand{\tbox}[1]{{\mbox{\tiny #1}}}
\newcommand{\mbf}[1]{{\mathbf #1}}
\newtheorem{thm}{Theorem}
\newtheorem{dfn}[thm]{Definition}

\newtheorem{rmk}{Remark}
\newcommand{\fref}[1]{Fig.~\ref{#1}}          
\newcommand{\sref}[1]{Sec.~\ref{#1}}          
\newcommand{\tref}[1]{Table~\ref{#1}}

\DeclareMathOperator{\im}{Im}
\DeclareMathOperator{\erf}{erf}
\newcommand{\vt}[2]{\left[\begin{array}{r}#1\\#2\end{array}\right]} 
\newcommand{\mt}[4]{\left[\begin{array}{rr}#1&#2\\#3&#4\end{array}\right]} 

\newcommand{\xx}{\mbf{x}}
\newcommand{\yy}{\mbf{y}}
\newcommand{\kk}{\mbf{k}}
\newcommand{\NN}{{\cal N}}       
\newcommand{\km}{{k_-}}      
\newcommand{\pO}{{\partial\Omega}}
\newcommand{\RR}{\mathbb{R}}
\newcommand{\ZZ}{\mathbb{Z}}
\newcommand{\Gm}{{\Gamma_-}} 
\newcommand{\eo}{\mbf{e}_x}       
\newcommand{\et}{\mbf{e}_y}       
\newcommand{\al}{\alpha}
\newcommand{\bt}{\beta}
\newcommand{\ui}{u^\tbox{i}}       
\newcommand{\ut}{u^\tbox{t}}          
\newcommand{\kx}{{\kappa^m_x}}
\newcommand{\ky}{{\kappa^n_y}}
\newcommand{\kz}{{\kappa^{(m,n)}_z}}
\newcommand{\snp}{\sum_{n=-P/2+1}^{P/2}} 
\newcommand{\pdeg}{p}    
\newcommand{\intc}{\int_0^{2\pi}}        
\newcommand{\mintc}{\frac{1}{2\pi}\int_0^{2\pi}}        
\newcommand{\ebc}{{\epsilon_\tbox{bc}}}           
\newcommand{\eper}{{\epsilon_\tbox{per}}}
\newcommand{\eflux}{{\epsilon_\tbox{flux}}}
\newcommand{\UU}{\Omega_\tbox{box}}            
\newcommand{\Aelse}{A_\tbox{else}}

\journal{Journal of Computational Physics}

\begin{document}  
\begin{frontmatter}
\title{Efficient numerical solution of acoustic scattering from doubly-periodic arrays of axisymmetric objects}
\author[p]{Yuxiang Liu\corref{YL}}
\ead{Yuxiang.Liu.GR@dartmouth.edu}
\cortext[YL]{Corresponding author. Tel.: +1 603 277 0791.}
\author[m,p]{Alex H. Barnett}
\address[m]{Department of Mathematics, Dartmouth College, Hanover, NH 03755, USA}
\address[p]{Department of Physics \& Astronomy, Dartmouth College, Hanover, NH 03755, USA}

\begin{abstract}
We present a high-order accurate boundary-based solver for
three-dimensional (3D)
frequency-domain scattering from a doubly-periodic grating of
smooth axisymmetric sound-hard or transmission obstacles.
We build the one-obstacle solution operator
using separation into $P$ azimuthal modes via the FFT,
the method of fundamental solutions (with $N$ proxy points
lying on a curve), and dense direct least-squares solves;
the effort is $\bigO(N^3P)$ with a small constant.
Periodizing 
then combines fast multipole summation of nearest neighbors
with an auxiliary global Helmholtz basis expansion
to represent the distant contributions,
and enforcing quasiperiodicity
and radiation conditions on the unit cell walls.
%
Eliminating the auxiliary coefficients,
and preconditioning with the one-obstacle solution operator,
leaves a well-conditioned square linear system that is solved iteratively.
The solution time per incident wave is
then $\bigO(NP)$ at fixed frequency.
%
Our scheme avoids singular quadratures,
periodic Green's functions, and lattice sums,
and its convergence rate is unaffected by resonances within obstacles.
We include numerical examples such as
scattering from a
grating of period $13\lambda \,\times\, 13\lambda$ comprising highly-resonant
sound-hard ``cups'' each needing $NP=64800$ surface unknowns,
to 10-digit accuracy, in half an hour on a desktop.
%
\end{abstract}


\end{frontmatter}

\section{Introduction}
\label{sec:intro}

The control of waves using periodic structures
is crucial for modern optical, electromagnetic and
acoustic devices such as diffraction gratings, filters,
photonic crystals and meta-materials \cite{jobook},
solar cells \cite{atwater}, and absorbers \cite{noisecontrol,DCD1}.
Periodic scattering problems also arise in monitoring \cite{scatterometry}
or imaging \cite{Malcolm14} a patterned structure.
Outside of asymptotic regimes where analytic models are useful,
efficient and accurate numerical simulation is key
to assess sensitivity to changes in parameters,
and to optimize those parameters to improve device performance.

Here we present a solver for 3D acoustic scattering from a doubly-infinite
array of isolated axisymmetric objects, as shown in Fig.~\ref{f:geom}(a).
With acoustic applications in mind, we focus on the Neumann (sound-hard)
boundary condition (including resonators) and transmission problems,
and on highly accurate solutions.
Sound absorbing surfaces often involve periodic structures
such as perforated slabs \cite[Sec.~6.5.4]{noisecontrol}, resonators
\cite[Sec.~9.2.3]{noisecontrol} \cite[Ch.~12]{handbkac},
or wedges (as in anechoic chamber walls) \cite[Fig.~12-13]{handbkac}.
Recently, there has also been interest in
acoustic meta-materials \cite{acousticmeta} or ``phononics'',
including new phenomena such as anomalous transmission
\cite[Ch.~4]{acousticmeta} and acoustic cloaking
\cite{acousticcloak}. 
The new techniques we present
are relatively simple to extend to
multilayer geometries (such as perforated slabs)
and poroelastic media (such as foams) common in noise control
\cite{poroelas}.
We view this work---in particular the periodizing scheme,
which is very general---as a step toward 3D multilayer periodic
boundary-based solvers (generalizing recent work in 2D \cite{mlqp})
for acoustics, coupled acoustics-elastodynamics,
the Maxwell equations, and Stokes flow.

Let us explain where our contribution fits into the bigger picture.
Direct volume discretization methods including the finite element
\cite{bao95,scatterometry} and
finite difference time-domain \cite{taflove,FDTD3DBC} are
common for acoustic scattering problems,
and can be successful for low-to-medium frequencies and accuracies.
However, low-order finite elements suffer from
the accumulation of phase errors across the domain (known as ``pollution''
\cite{pollution}),
meaning that an increasing number of unknowns per wavelength
are needed as the wavenumber $k$ grows, 
so that greater than $\bigO(k^3)$ unknowns are needed to maintain accuracy.
High-order finite elements,
while showing recent promise for medium-frequency acoustics
at accuracies of a couple of digits \cite{beriot16},
are tricky to generate in complex geometries, and have not been used for the
periodic problem as far as we are aware.
Time-domain methods (e.g.\ FDTD) also suffer from a low convergence order
(usually at most 1st-order in the presence of surfaces;
recent work has recovered 2nd-order \cite{FDTD3DBC}),
difficulty in modeling impedance boundary conditions \cite{FDTDimp},
and very long settling times if a structure is resonant.
It should be noted that via the Fourier transform, FDTD can solve
many frequencies in a single shot.
However, accuracies from such methods are commonly 1-2 digits,
even with dozens of grid points per wavelength \cite{FDTD3DBC}.

For piecewise-uniform media,
boundary-based methods become more efficient than direct discretization
once the geometry is more than a couple of wavelengths across,
and/or if an accuracy beyond a couple of digits is needed.
Only $\NN = \bigO(k^2)$
unknowns are needed for a smooth obstacle
if a high-order surface quadrature is used.
The most common approach is the boundary integral method (BIE),
in which the scattered wave is represented via
potential theory using the Helmholtz Green's function
\be
G_k(\xx,\yy) := \frac{e^{ik|\xx-\yy|}}{4\pi |\xx-\yy|}, 
\qquad \xx,\yy \in \RR^3 
~.
\label{G}
\ee
For the mathematical foundation of this method in the non-periodic
setting see \cite{CK83,coltonkress}, and in our periodic setting \cite{arenshabil}.
Formulation as a 2nd kind Fredholm integral equation on $\pO$,
the boundary of an obstacle $\Omega\subset\RR^3$, has the advantage that
the condition number of the discretized linear system remains small
independent of $\NN$ as $k$ is held fixed.
Using the fast multipole method (FMM) \cite{fmm_Greengard,fmm1}
to apply the dense matrix
discretization of the operator inside an iterative
Krylov method solver such as GMRES \cite{gmres}
can create an $\bigO(\NN)$ solver (at moderate frequencies).
However, in practice two problems plague this otherwise attractive scheme:
1) for resonant or geometrically complex
objects the large number of eigenvalues close to the
origin causes the number of iterations to be unreasonably large,
and direct solvers can be orders of magnitude faster
\cite{qpfds}.
2) high-order surface quadratures in 3D are quite challenging
and are still an area of active research
\cite{bruno01,ying06,nicholas,bremer3d,qp3dqbx}.
We note that low-order Galerkin methods are the most commonly used
in Helmholtz problems \cite{otani08}. 
The method we propose fixes both of these problems in the axisymmetric setting.

We use a close relative of the BIE, the {\em method of fundamental solutions} (MFS,
also known as auxiliary sources \cite{Fridon_MAS} or the
charge simulation method \cite{Katsurada1}),
in which (for sound-hard scattering)
the scattered wave has a representation
\be
u(\xx) \approx \sum_{j=1}^\NN c_j G_k(\xx,\yy_j)
\label{mfs}
\ee
where $\yy_j \in \RR^3$, $j=1,\ldots,\NN$,
are $\NN$ {\em source points} covering a source surface $\Gm$
lying inside the obstacle $\Omega$ but close to its boundary $\pO$.
The coefficients $\{c_j\}$ are the solution to a linear system
set up by matching \eqref{mfs} to the boundary data
at collocation points on $\pO$.
The MFS idea (see the review \cite{mfs_review}) is old, being
first proposed by Kupradze--Aleksidze \cite{mfs_Kupradze},
and is common in the engineering community \cite{doicu}.
The MFS may be viewed as an exponentially ill-conditioned first-kind
BIE. Yet, when combined with a backward stable linear solver
it can achieve close to machine precision
when the source points are chosen correctly \cite{mfs}.
An optimal choice of source points remains one of the more ad-hoc
aspects of the MFS, although we demonstrate in \sref{s:loc} an excellent scheme
for analytic boundaries that requires only a single adjustable parameter.
The MFS has a significant advantage over BIE:
because $\Gm$ is separated from $\pO$,
{\em no singular surface quadratures} are needed
(either for the self-interaction of the object or for evaluation of $u$
close to $\pO$).
This will enable us to handle Neumann and transmission
conditions simply 
(in contrast, the robust BIE formulation in the Neumann case requires
handling the derivative of the double-layer operator
either as a hypersingular operator \cite{kress95}
or using Calder\'on regularization
\cite[Sec.~3.6]{CK83} \cite{Anand_Corner},
and in the transmission case the compact difference of such operators
\cite[Sec.~3.8]{CK83}).
The MFS has proven useful in 2D acoustic settings \cite{poroelas}.

In the axisymmetric case the MFS becomes more efficient
\cite{karagaxi,Fridon_MAS,Fridon_Optical,Chen13},
because the problem
separates into $P$ angular Fourier modes that may be solved
independently, each with a small number of unknowns $N$,
the total number of unknowns being $\NN = NP$.
As we will show, many smooth obstacles up to $10\lambda$ in diameter need
only $N<10^3$ for 10-digit accuracy, with $P$ of order one hundred.
Since dense least-squares solves of such a  size $N$ are cheap, the
good conditioning of a BIE approach confers little advantage over
using MFS, at least for smooth domains.
(Recent work also shows that several digits of accuracy is possible with the MFS
in corner domains \cite{Hcorner,larrythesis}.)
We note that recently some technical challenges of high-order BIE
on axisymmetric surfaces have been solved
\cite{Young,hao3daxi,helsing_axi,helsingIEEE15},
but not in the case of transmission
boundary conditions that we address.

We now outline how we turn a scheme for
solving the scattering from one obstacle into a scheme for a bi-infinite array of obstacles---we refer to this as ``periodizing'';
it is one of the main contributions of this paper.
The standard way to periodize in 2D
\cite{CWnystrom,brunohaslam09}
or 3D \cite{nicholas,arenshabil,brunoqp3d} 
is to replace the free-space Green's function \eqref{G}
by its {\em quasiperiodic} version which sums over all source points,
\be
G_k^\tbox{QP}(\xx,\yy) := \sum_{n,m\in \ZZ}\al^m \bt^n G_k(\xx,\yy + m\eo + n\et)
~,
\label{GQP}
\ee
where the {\em Bloch phases} are $\al$ and $\bt$ (defined below in \eqref{bloch}),
and the array lattice vectors are $\eo$ and $\et$ as in Fig.~\ref{f:geom}(a).
The above sum is notoriously slowly convergent,
hence a host of schemes such as Ewald's method
\cite{ewald,jordan86,arenshabil,arens10},
other spatial-spectral splittings \cite{jorgenson},
or lattice sums \cite{otani08,lintonrev}
have been developed for numerical evaluation.
These schemes are generally
quite complicated, both analytically and in terms of implementation,
and raise two major problems:
\ben
\item
While they are able to fill the $\NN^2$ elements of a dense matrix,
most such schemes are incompatible with the FMM or other fast algorithms.
Exceptions are the lattice-sum based correction to the FMM of
Otani et al \cite{otani08} and the rolled-off spatial sum
of Bruno et al \cite{brunoqp3d}.
\item At certain sets of parameters $(\al,\bt)$ and $k$
called {\em Wood anomalies}, the quasiperiodic
Green's function does not exist, i.e.\ the sum \eqref{GQP} diverges,
even though the solution to the diffraction problem remains well-posed and
finite.
\een
We propose a simple new approach, following
\cite{BG_BIT,Gillman_Barnett_jcp,mlqp,qp3dqbx},
which cures the first problem.
In the setting of continuous interfaces such as \cite{mlqp},
our approach would also cure the second;
however, since we focus on isolated obstacles, then
(as with the work in \cite{otani08,brunoqp3d})
we do not attempt to address Wood anomalies here.
Our representation restricts the sum \eqref{GQP} to only the $3\times 3$ block of nearest
neighbors, adding an auxiliary spherical harmonic basis
(of maximum degree $\pdeg$) for smooth
Helmholtz solutions in the neighborhood of the object,
\be
u(\xx) \approx \sum_{j=1}^\NN c_j G^\tbox{near}_k(\xx,\yy_j)
+ \sum_{l=0}^\pdeg\sum_{m=-l}^l d_{lm} j_l(kr) Y_{lm}(\theta,\phi)
~,\quad\mbox{ where}
\quad 
G_k^\tbox{near}(\xx,\yy) := \sum_{|m|,|n| \le 1}\al^m \bt^n G_k(\xx,\yy + m\eo + n\et)
~.
\label{rep}
\ee
Here $(r,\theta,\phi)$ are spherical coordinates.
The auxiliary basis represents the contribution from the remainder of the
lattice of images; since they are far from $\Omega$, this has
rapid exponential convergence with $\pdeg$.
An expanded linear system is used to solve for the coefficients.
Being rectangular and ill-conditioned (as with the plain MFS method),
this cannot be solved  iteratively in the case of large $\NN$.
In \sref{s:iter} we show that, because of axial symmetry,
preconditioning is possible using a direct
factorization of the MFS matrix pseudoinverse
for the single obstacle.
In the iterative scheme, the contributions from the near images are applied
using the FMM, to give a scheme that solves each new incident wave in
$\bigO(\NN)$.
Our work can thus be seen as a periodic generalization of the fast multibody
scattering work of Gumerov--Duraiswami for spheres \cite{multisphere},
of Gimbutas--Greengard for smooth scatterers \cite{fmm_maxwell},
and of Hao--Martinsson--Young \cite{hao3daxi} for axisymmetric scatterers.

One novel aspect is the high accuracy we achieve (around 10 digits)
compared to most other periodic integral equation work \cite{otani08,brunoqp3d}.
This is true even for resonant obstacles,
thanks to the {\em direct} solve used for the isolated obstacle.
Our scheme is very practical for periods up to a dozen wavelengths
in each direction, but cannot go much higher than this in reasonable CPU time
due to the $O(p^6)$ scaling of the dense matrix operations.
However, this covers the vast majority of diffraction
applications, where typically
the period is of order the wavelength, and also allows ``super-cell''
simulations, for instance for random media.



A similar scheme has recently been proposed
by Gumerov--Duraiswami \cite{gumerov}
for periodizing the 3D Laplace equation,
using proxy sources
instead of the auxiliary basis in \eqref{rep}; however, in Appendix~A we show that
for our application to the 3D Helmholtz equation
spherical harmonics are much more efficient.


Our paper is organized as follows.
In \sref{s:bvp} we state the two periodic scattering boundary value problems
under study, namely Neumann and transmission conditions.
In \sref{s:mfs} we explain the MFS in the axisymmetric one-obstacle setting,
including the choice of source points,
and study its convergence 
when using a dense direct solve
for each Fourier mode.
In \sref{s:per} we present the periodizing scheme and the resulting
full linear system, and then show how Schur complements can turn this into
a 
well-conditioned square system that can be solved iteratively.
Numerical results are presented in \sref{s:res}, then
we conclude in \sref{s:con}.
The appendix compares the efficiency of
our periodizing scheme with a variant using proxy points.

\begin{figure} 
  \centering
\raisebox{-1.3in}{\ig{width=2in}{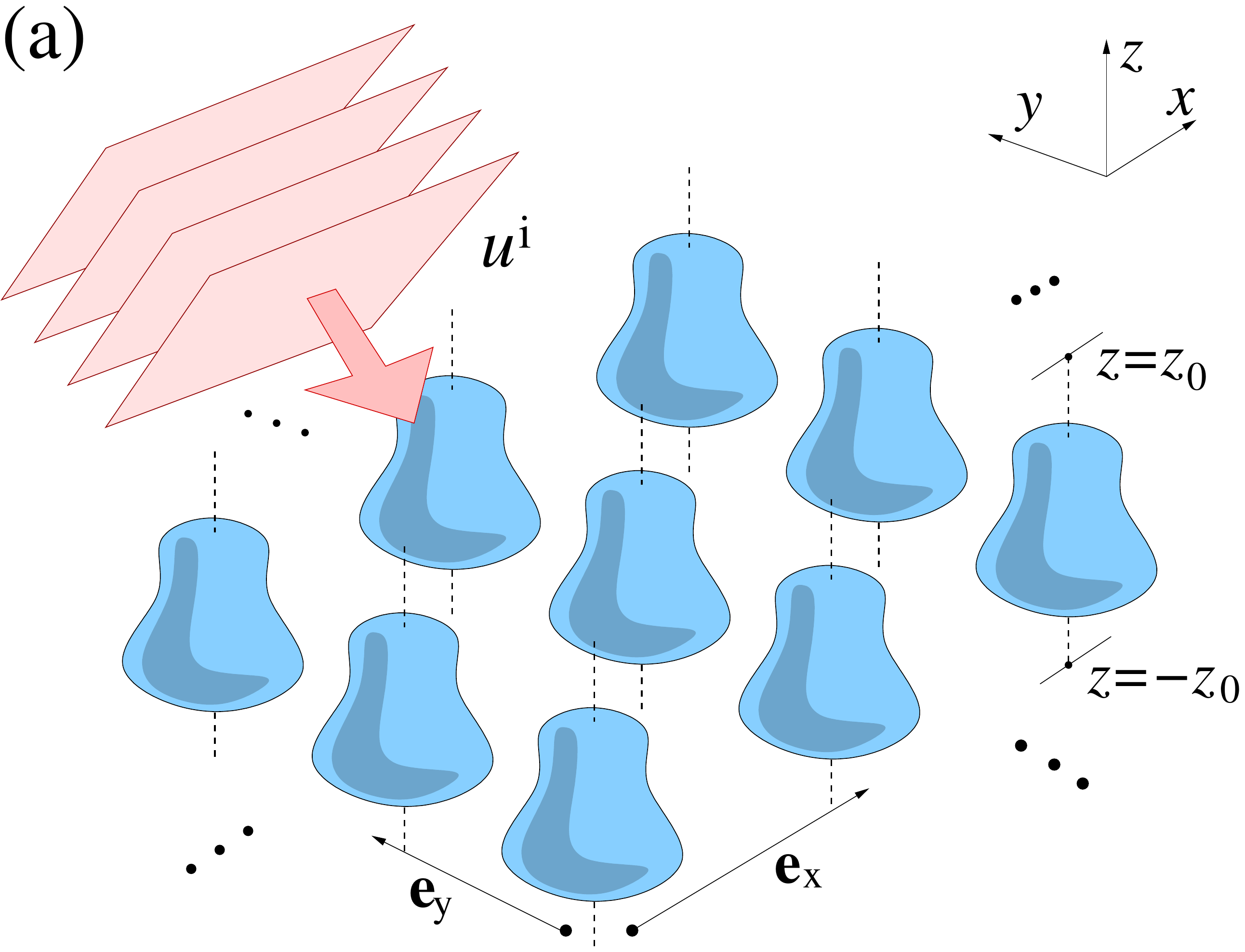}}
  (b)\raisebox{-1.3in}{\ig{width=2in}{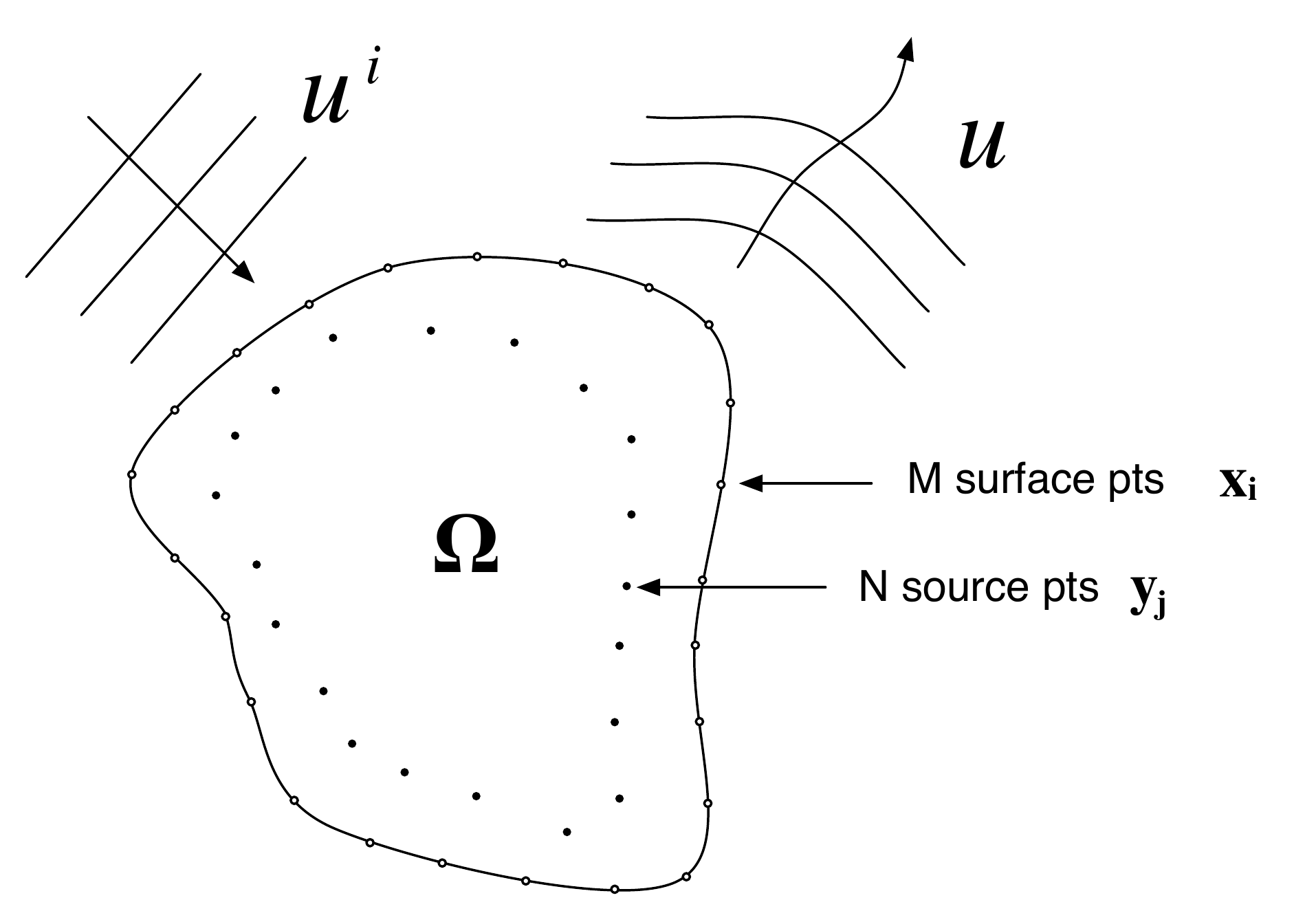}}
  (c)\raisebox{-1.3in}{\ig{width=2in}{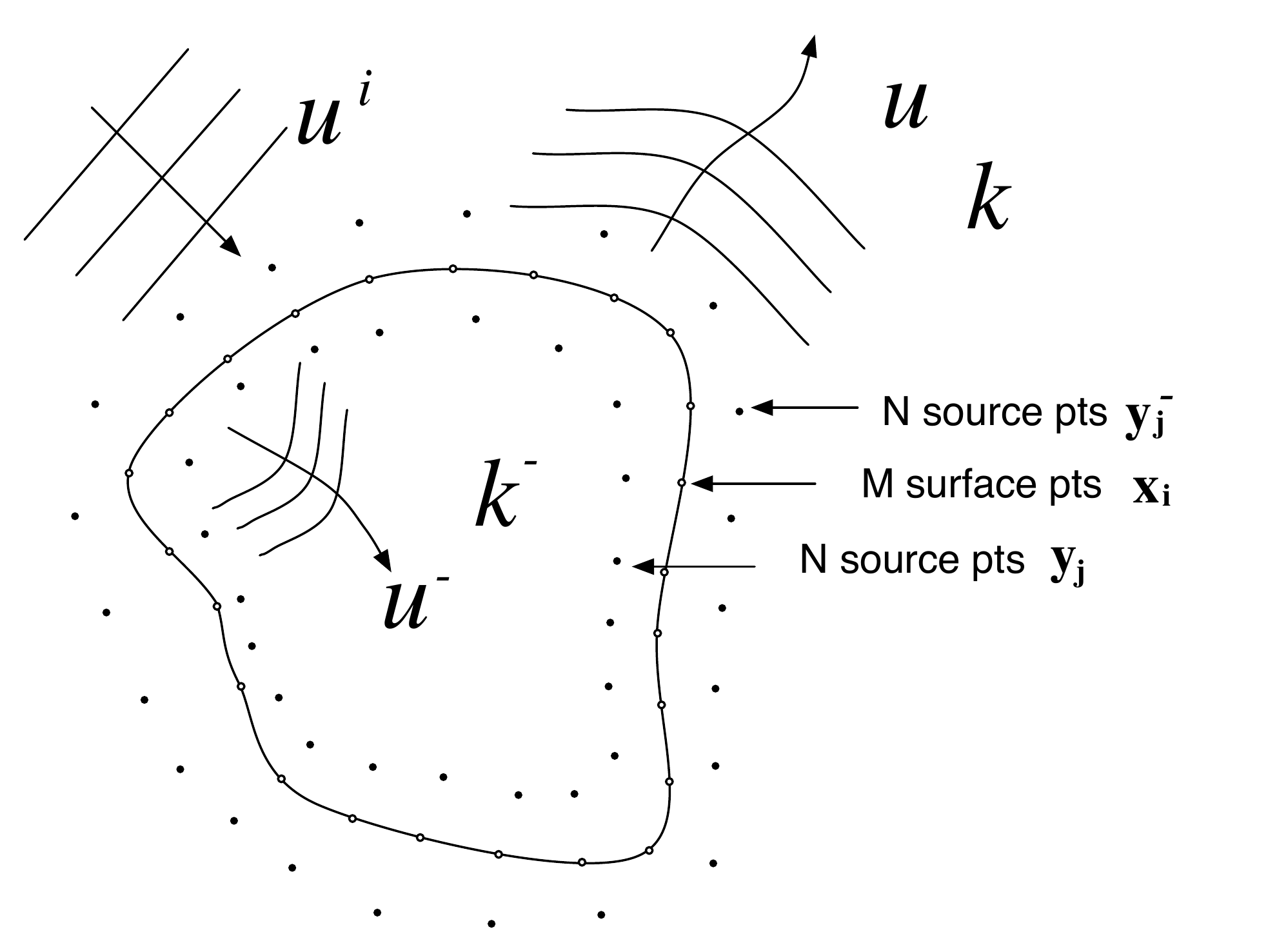}}
\caption{(a) Periodic scattering geometry in 3D.
(b) Sketch of the method of fundamental solutions (MFS) for solution 
of an exterior BVP. (For simplicity, the 2D case is shown.)
(c) MFS for a transmission problem (again, the 2D case is shown).
}
  \label{f:geom}
\end{figure}

\section{Formulation of the boundary value problems}
\label{s:bvp}

Let $\Omega\subset\RR^3$ be an axisymmetric obstacle with
boundary $\pO$ defined by the rotation about the $z$-axis of
a smooth curve $\gamma$ lying in the $\rho$-$z$ plane,
where $(\rho,\theta,z)$ define a cylindrical coordinate system.
For simplicity we consider a rectangular lattice
defined by vectors $\eo=(e_x,0,0)$ and $\et=(0,e_y,0)$,
so that the grating of objects (see \fref{f:geom}(a)) is defined by
$\Omega_\Lambda := \{(x,y,z)\in\RR^3: (x+me_x,y+ne_y,z) \in \Omega,
\; m,n\in \ZZ\}$,
and is assumed not to self-intersect.
Note that our scheme will also apply to general (possibly skew) lattices
and a general (fixed) axis of symmetry of the objects with minor changes in
bookkeepping.
An incident plane wave (representing pressure variation of
a time-harmonic acoustic wave),
\be
\ui(\xx) = e^{i\kk \cdot \xx}, \qquad \xx:=(x,y,z) \in \RR^3
\label{ui}
\ee
impinges on
this lattice, with given wavevector $\kk=(k_x,k_y,k_z)$
whose free-space wavenumber is $k:= |\kk|$.
This incident wave is quasiperiodic in the following sense.
\begin{dfn}
A function $u: \RR^3 \to \mathbb C$ is said to be quasiperiodic with
Bloch phases $\al$ and $\bt$, if 
\be
\al^{-1} u(x+e_x, y, z) = \bt^{-1} u(x,y+e_y,z)= u(x,y,z)
\qquad \forall (x,y,z) \in \RR^3~.
\label{per}
\ee
\end{dfn}
The incident wave parameters fix the Bloch phases
\be
\al=e^{i\kk\cdot \eo}~, \qquad \bt=e^{i\kk\cdot \eo}
\label{bloch}
\ee
with $|\al|=|\bt|=1$, and $\ui$ is quasiperiodic with these phases.
The resulting scattered wave $u$ will share this quasiperiodic symmetry.
As usual in scattering theory \cite{coltonkress},
the physical wave outside the lattice of objects is the total
$\ut = \ui + u$.
The scattered wave obeys the exterior Helmholtz equation
\be
\Delta u + k^2 u=0,  \quad \text{in  } \RR^3 \backslash \overline{\Omega_\Lambda}
\label{helm}
\ee
and the upwards and downwards Rayleigh--Bloch radiation conditions
\cite{bonnetBDS,shipmanreview}
\begin{subequations}
\begin{align}
u(x,y,z) &= \sum _{m,n\in \mathbf Z} a_{mn} \exp{i[\kx x + \ky y + \kz (z - z_0)]}~,\qquad z \ge z_0, \; (x,y) \in \RR^2  \label{rb1} \\
u(x,y,z) &= \sum _{m,n\in \mathbf Z} b_{mn} \exp{i[\kx x + \ky y + \kz ( - z - z_0)]}~, \qquad z \le -z_0, \; (x,y) \in \RR^2 \label{rb2}
\end{align}
\end{subequations}
where $z_0$ is such that $\overline\Omega$ lies between the planes $z=\pm z_0$,
and 
where $\kx:= k_x + 2\pi m/e_x, \ky:= k_y + 2\pi n/e_y$ and
$\kz:= + \sqrt {k^2 - (\kx)^2 - (\ky)^2}$.
define the plane wave wavevectors $(\kx,\ky,\pm\kz)$.
Note that the sign of the square-root is taken as positive real or
positive imaginary.
These conditions state that $u$ can be written as a
uniformly convergent expansion of 
quasiperiodic plane waves of {\em outgoing or decaying} type away from
the lattice.
In applications the far field amplitudes $a_{mn}$ and $b_{mn}$ are
the desired quantities, giving the radiated
strengths in the various Bragg orders.

We will solve two cases of obstacle scattering:
1) Neumann (sound hard) boundary conditions scattering from an
impenetrable obstacle,
\be
\frac {\partial u} {\partial n} = -\frac {\partial \ui} {\partial n}, \quad  \text{on $\pO$ }
\label{neu}
\ee
and 2) Transmission conditions, with a new wavenumber $\km$
and quasiperiodic scattered wave $u^-$ inside the object (the incident
wave $\ui$ being defined as zero inside $\Omega$), i.e.,
\begin{subequations}
  \begin{align}
\Delta u^- + \km^2 u^- &=0,  \quad \text{in  } \Omega \label{helmi}\\
u - u^- &= -\ui  \quad  \text{on $\partial \Omega $ } \label{eq:neup1.3}\\  
\frac {\partial u} {\partial n} - \frac {\partial u^-} {\partial n} &= -\frac {\partial \ui} {\partial n} \quad  \text{ on } \partial \Omega~.
\label{eq:neup1.4}
  \end{align}
\end{subequations}
For simplicity, \eqref{eq:neup1.4} models the case
where the interior and exterior densities are equal
(a density difference would result in prefactors here
\cite[Sec.~2.1]{coltonkress}).
Note that, due to quasiperiodicity, the 
above PDE and boundary conditions need only be
defined on a single copy of the object in the lattice.

Given an incident wave $\ui$ with wavevector $\kk$,
the full boundary value problem (BVP) is defined by
\eqref{per}--\eqref{helm}, radiation conditions \eqref{rb1}--\eqref{rb2},
and either Neumann condition \eqref{neu} or transmission conditions
\eqref{helmi}--\eqref{eq:neup1.4}.
The following rigorous results are known about the BVPs.
With a general smooth obstacle shape $\Omega$ and fixed wavespeed ratio
(material property) $\km/k$, the transmission BVP
has a solution for all incident wave directions and
frequencies $k>0$ \cite[Thm.~9]{shipmanreview}.
The solution is unique for all but possibly a discrete
set of frequencies, for each incident wave direction
\cite[Thm.~8]{shipmanreview}.
Such frequencies are referred to as {\em bound states}
of the grating \cite{shipmanreview}, and correspond to physical
resonances of the BVP.
Similar statements hold in the Neumann case.
In the transmission case, if $\km<k$ and
all lines parallel to the $z$-axis intersect $\pO$ at only two points,
then no such bound states can exist \cite[Thm.~13]{shipmanreview}.
A similar analysis for a connected interface (or multiple such interfaces)
is carried out by Arens \cite{arenshabil}.

If $\kz=0$ for any pair of integers $(m,n)$, this defines a
{\em Wood anomaly} \cite{linton07,shipmanreview},
where one (or more) of the Rayleigh--Bloch waves is constant (non-decaying)
in the $z$ direction. Although \eqref{GQP} does not exist at Wood anomalies
(it diverges like an inverse square-root with respect to the incident
wave parameters), 
the BVP remains well-behaved, i.e.\ well conditioned with respect to varying the
amplitude of $\ui$.

\begin{rmk}  
Because a Bragg scattering mode is ``on the cusp of existence'' at a Wood
anomaly,
the rate of change of scattering coefficients $a_{mn}$ and $b_{mn}$
with respect to incident angle or frequency $k$ diverges there
(as an inverse square root singularity).
Thus, for example, at a parameter distance $10^{-10}$ from a Wood anomaly,
we expect to lose around 5 digits of accuracy purely due to
round-off in the representation of the input parameters.
\end{rmk}

In this work, we assume that we are not at a Wood anomaly.
However, 
we will find that the method we present in \sref{s:per} loses
of order the same number of digits as discussed in the above remark.

\bfi 
\pig{1.75in}{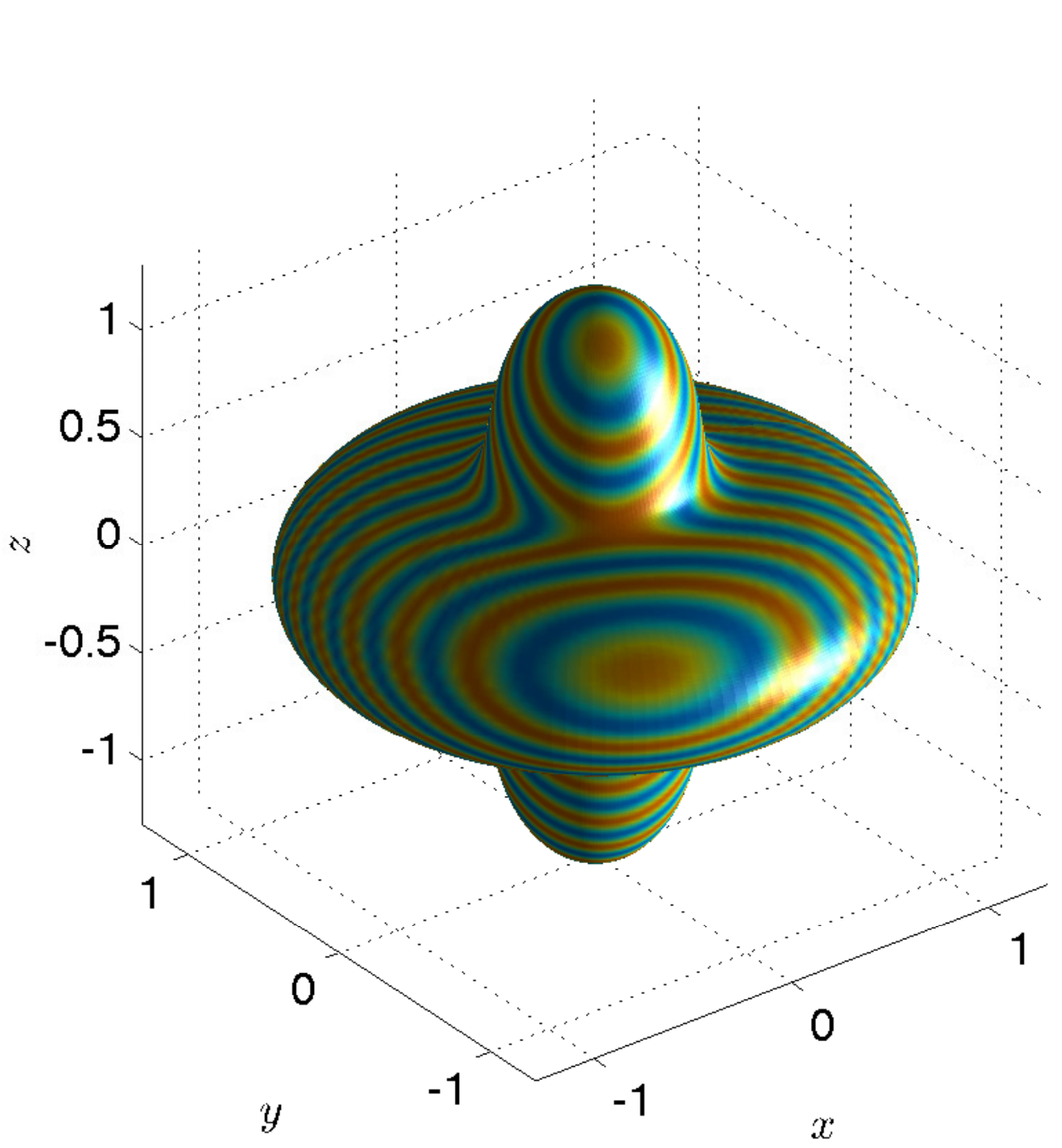} \qquad
\pig{1.75in}{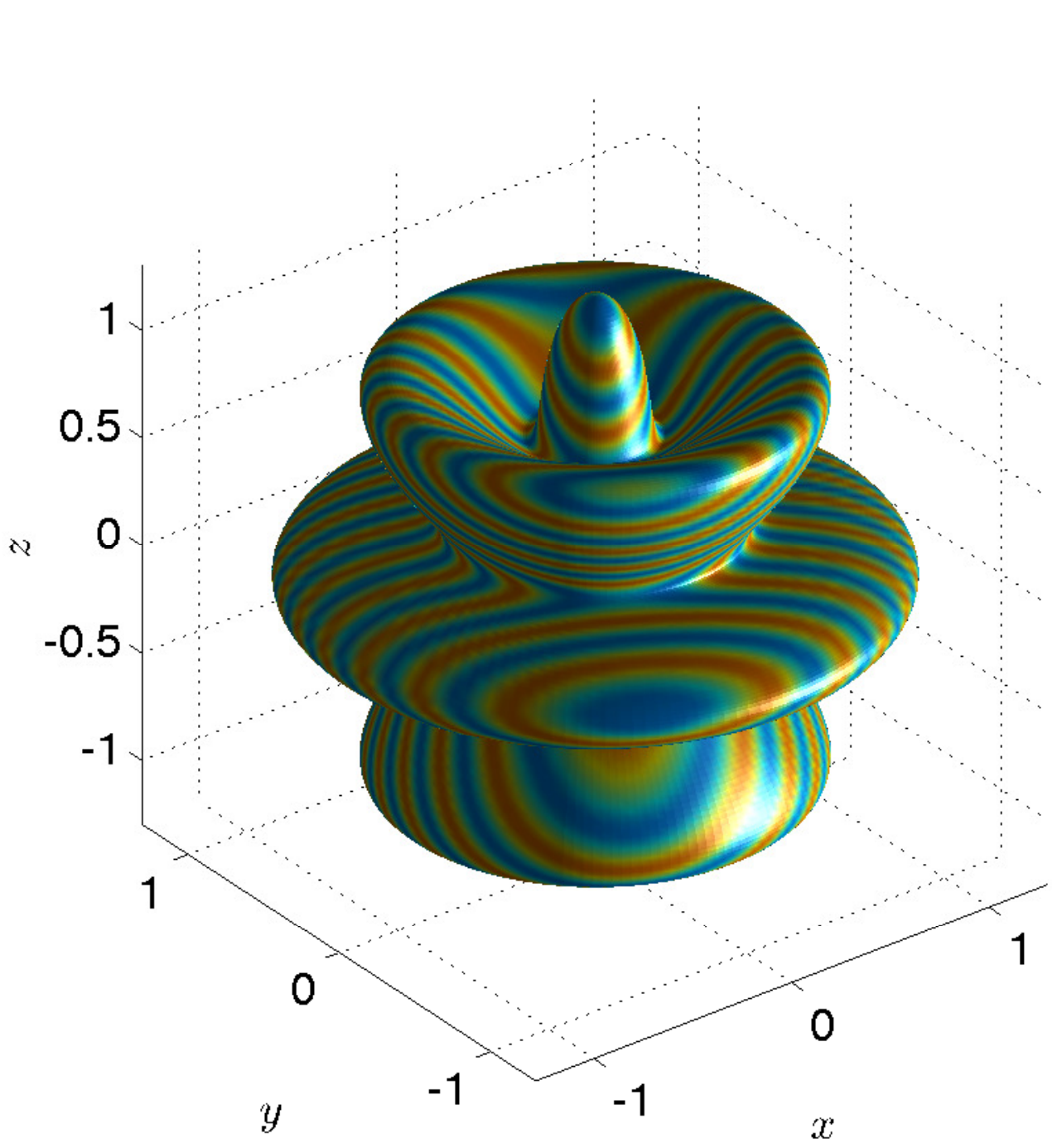} \qquad
\pig{1.75in}{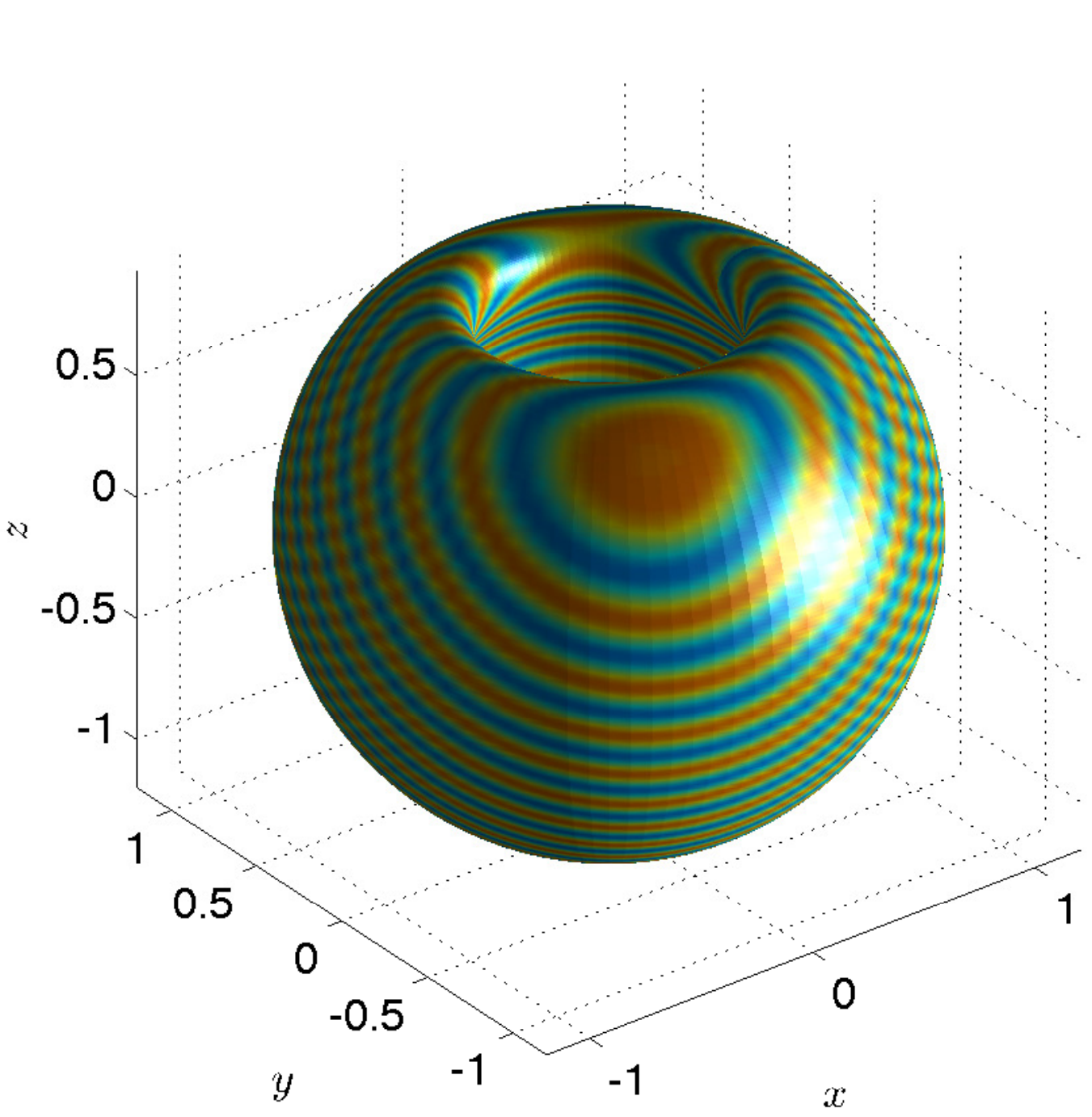}
\\
\pig{1.5in}{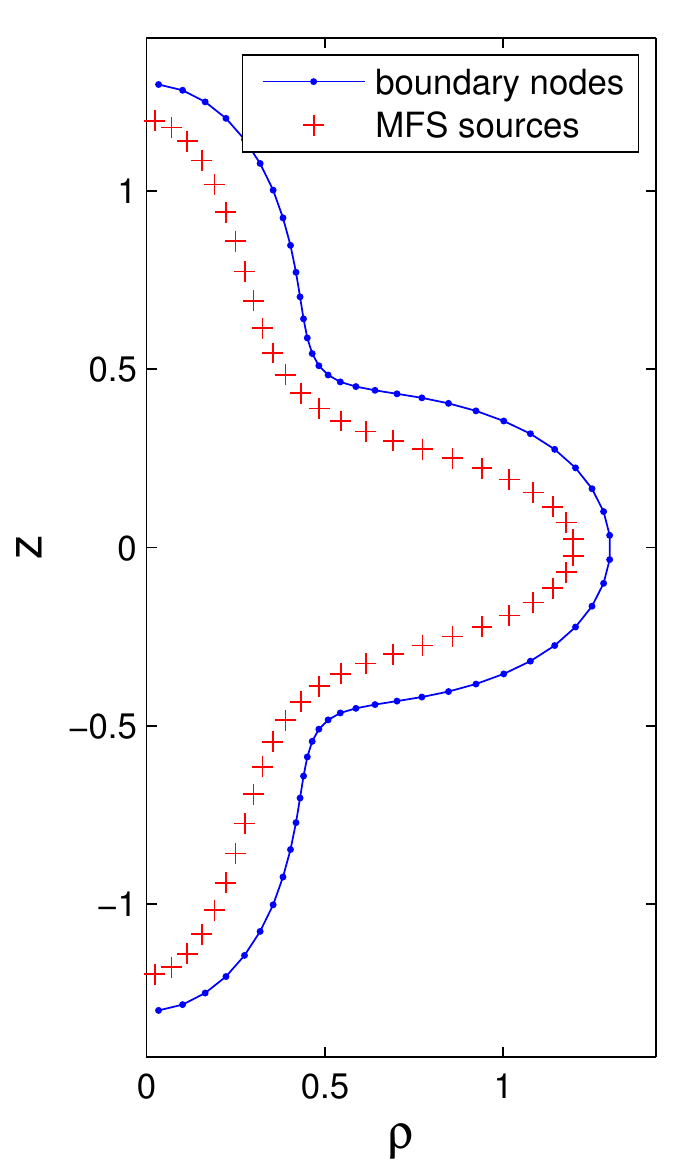}\qquad\qquad
\pig{1.5in}{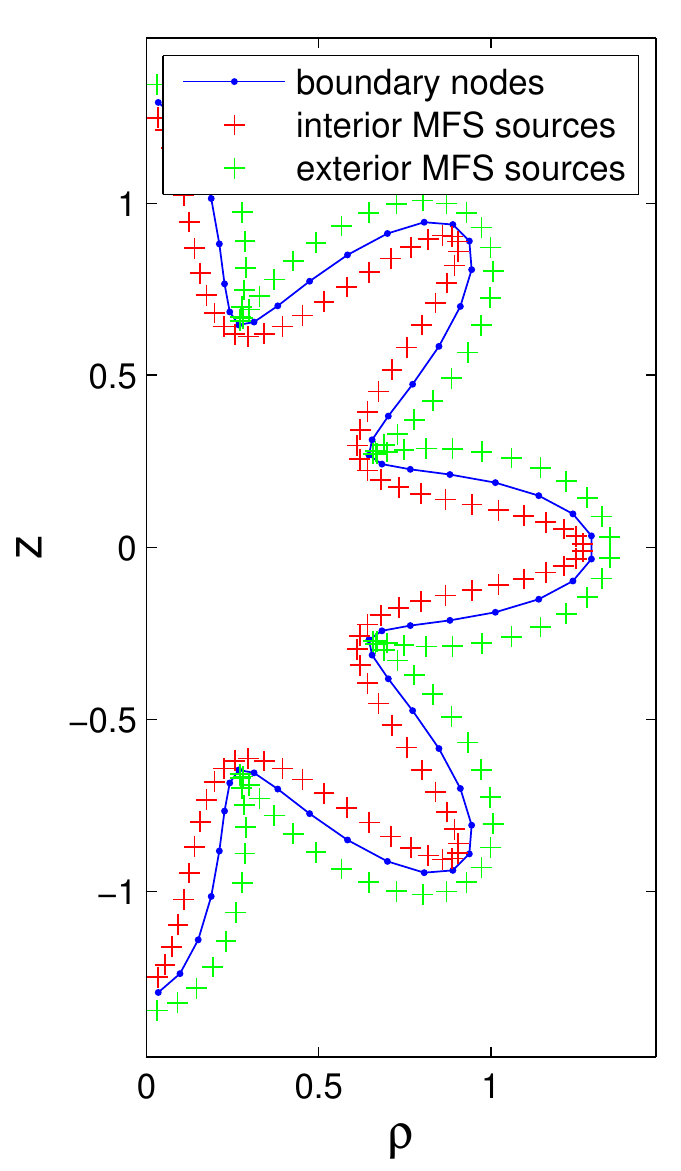}\qquad\qquad
\pig{1.6in}{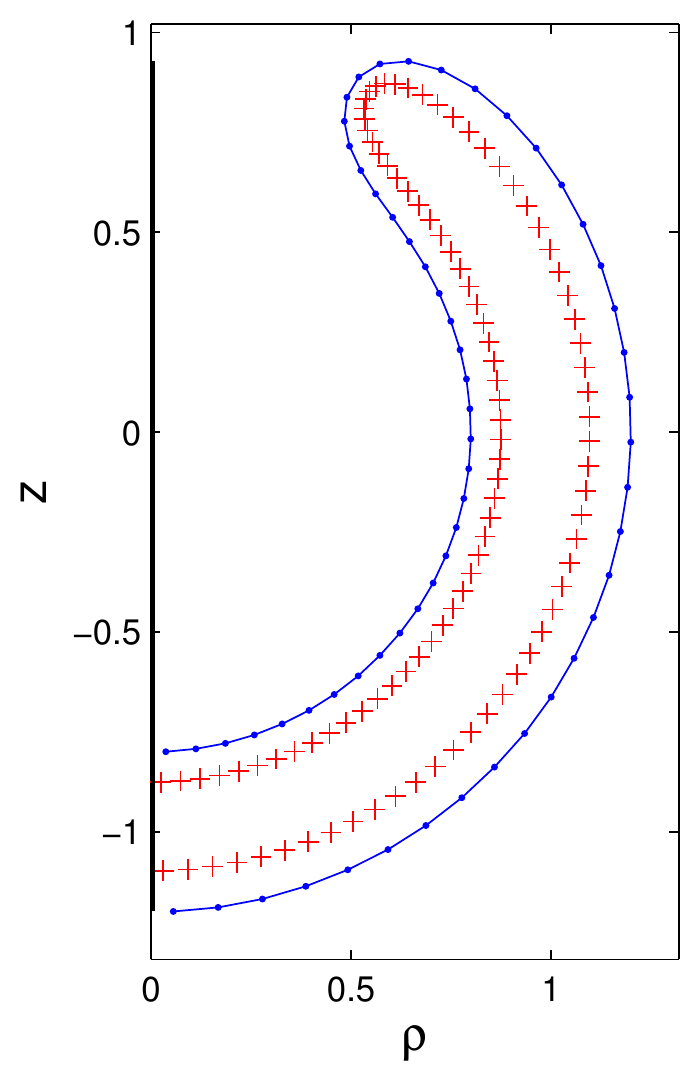}
\ca{
Obstacle bodies of revolution tested in this work.
Above is the 3D surface, and below the corresponding
2D generating curve and $N$ MFS points (with $M\approx 1.2 N$ boundary points).
Color above shows the real part of the incident wave $\ui$
at the highest frequency tested $k=40$
in \sref{s:borconv}, restricted to the surface.
From left to right:
(a) ``Smooth'' shape given in polar coordinates
in the $\rho$-$z$ plane by $r(\theta) = 1 + 0.3 \cos 4 \theta$.
The $N=50$ MFS sources are shown with distance parameter is $\tau=0.1$.
(b) ``Wiggly'' shape $r(\theta) = 1 + 0.3 \cos 8 \theta$.
$\tau=\pm 0.03$. $N=100$.
%
(c) ``Cup'' shape given parametrically
in $0\le t\le \pi$
by $r(t) = 1 - a\erf[(t-\pi/2)/a]$,
$\theta(t) = b-a + 2(1-\frac{b-a}{\pi})s_a(t-\pi/2)$
where the half-thickness is $a=0.2$ and opening half-angle $b=\pi/6$,
and the $a$-rounded abs-val function is
defined by $s_a(x):= (c/\sqrt{\pi})e^{-x^2/a^2} + \erf(x/a)$.
$\tau=0.05$. $N=100$.
}{f:shapes}
\efi

\bfi 
\bc
(a)\raisebox{-2in}{\ig{height = 2.0 in}{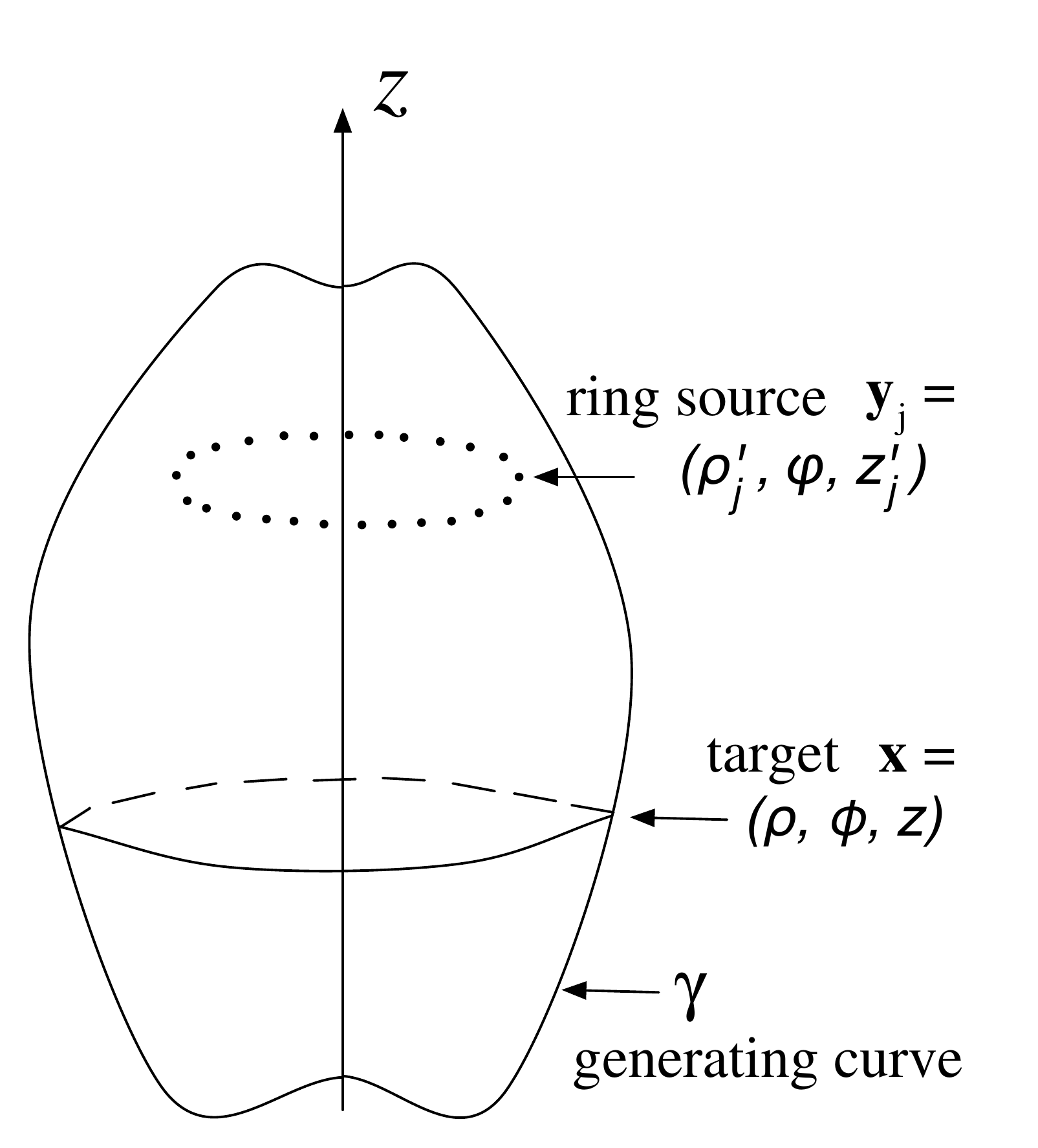}}
\quad
(b)\raisebox{-2in}{\ig{height = 2.0 in}{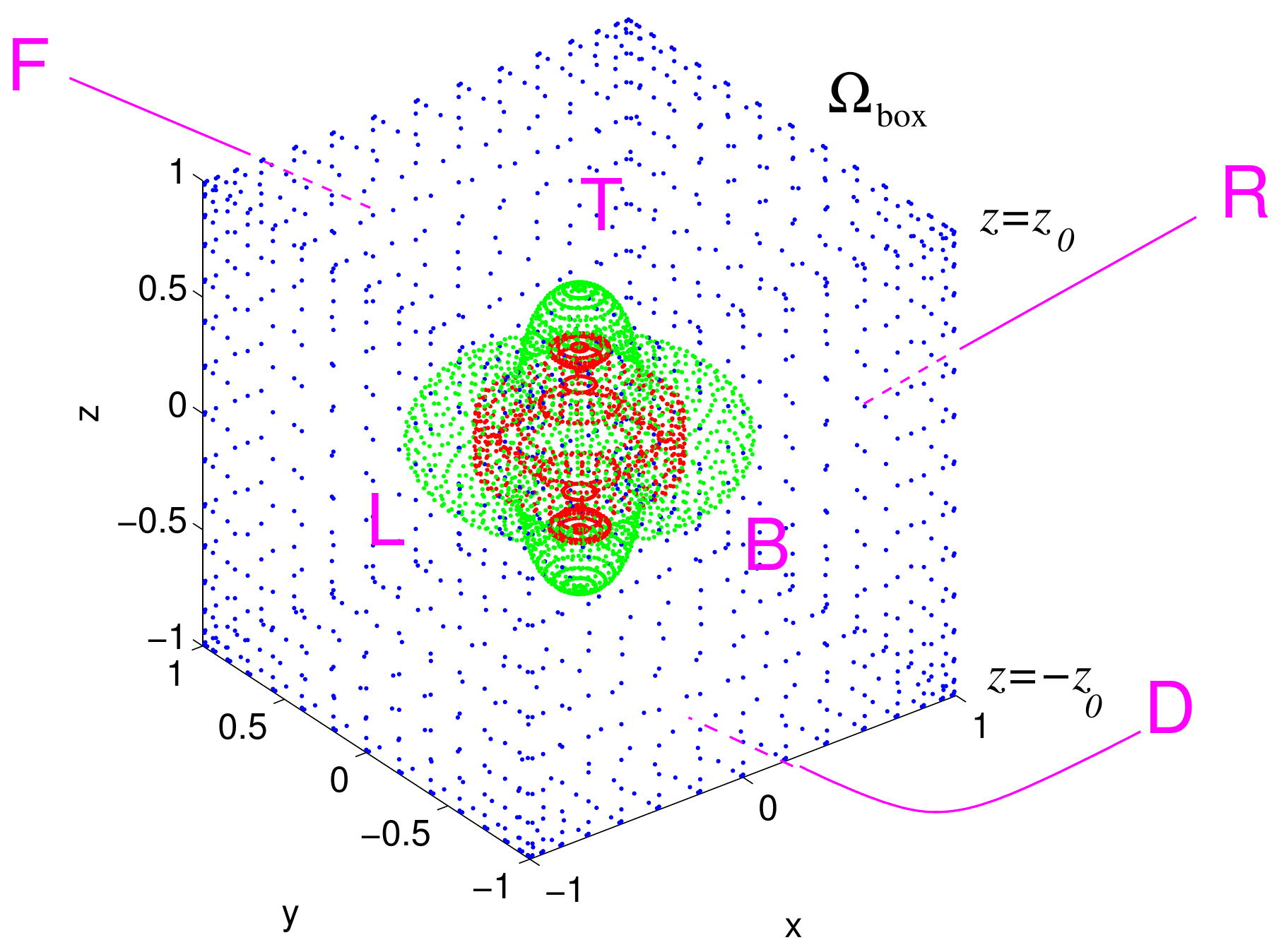}}
\ec
\ca{(a) Illustration of MFS
for the exterior axisymmetric BVP showing a ``ring charge'' (dotted circle,
where dots show approximation by individual point sources),
and target ring (solid circle).
(b) Geometry for periodization scheme:
Surface points (green), MFS source points (red),
and box wall discretization points (blue).
}{f:bor}
\efi

\section{Method of fundamental solutions in the axisymmetric setting}
\label{s:mfs}


In this section we present the MFS for scattering from an {\em isolated}
axisymmetric obstacle, which will form a key part of the periodic solver.
We present the Neumann case first, and then
explain how the transmission case differs.
Thus we care to solve the BVP given by \eqref{helm}
in $\RR^3 \backslash \overline{\Omega}$,
boundary conditions \eqref{neu},
and the usual 3D Sommerfeld radiation condition,
\be
\frac{\partial u}{\partial r} - iku = o(r^{-1}),
\qquad r := |\xx| \to \infty
\label{somm}
\ee
where the convergence implied by the little ``$o$'' is uniform in angle.
We exploit standard separation of variables
to represent the incident $\ui$ and scattered wave solution
$u$ as a sum of azimuthal Fourier modes.
Each mode results in an independent linear system.

Let the generating curve $\gamma$
for the body of revolution about the $z$-axis
be parametrized by the smooth functions $(\rho(t),z(t))$ in the $\rho$-$z$ plane, with $t\in[0,\pi]$.
The ``speed function'' $s(t) = \sqrt{(\rho'(t))^2 +(z'(t))^2}$ is
non-vanishing.
Note that the surface must also be smooth at $t=0$ and $t=\pi$, ie
$z'(0)=z'(\pi) = 0$.
For example, the generating curves we test are shown in the lower
part of \fref{f:shapes};
note that (c) is not formally smooth at $t=0$ and $\pi$, but is
smooth to within double-precision rounding error because of the
exponentially small deviation of the erf (error function) from $\pm 1$
at large arguments.

The right-hand side surface data is \eqref{neu},
$f = -\partial \ui/\partial n$,
which we approximate by the $P$-term truncated Fourier series
\be
f(\rho,z,\phi) \approx \snp \hat{f}_n(\rho,z) e^{in\phi}~, \qquad (\rho,z)\in\gamma,
\ee
with coefficients
\be
\hat{f}_n(\rho,z) = \mintc f(\rho,z,\phi) e^{-in\phi} d\phi~,
\qquad n\in\mathbb{Z}~,
\label{fhat}
\ee
whose accurate numerical evaluation we present shortly.
$P$ is even, and the asymmetry of the $\pm P/2$ terms will have
no significant effect.
The MFS source points $\{ (\rho'_j, z'_j)\}_{j=1}^N$ live in the $\rho$-$z$ plane,
inside the curve $\gamma$; their location choice is discussed in the
next section.
The rotation of the $j$th source point an angle $\varphi$ about the $z$-axis
is denoted by
$$
\yy_j(\varphi) := (\rho'_j,\varphi,z'_j)
$$
in cylindrical 
coordinates.
Our representation for the scattered potential will be in terms of
Fourier mode
Helmholtz ``ring sources'',
defining for the $n$th mode evaluated at target $\xx$,
\be
\Phi_{nj}(\xx) \; := \;
\mintc G_k(\xx,\yy_j(\varphi)) e^{-in \varphi} d\varphi~,
\ee
recalling that $G_k$ is the free-space fundamental solution \eqref{G}.

Our ansatz for the scattered potential is then
\be
u(\xx) \;\approx\;
\sum_{j=1}^N \snp c_{nj} \Phi_{nj}(\xx)
\label{ubormfs}
\ee
where the complex unknowns $c_{nj}$ are stacked into vectors
$\mbf{c}_n := \{c_{nj}\}_{j=1}^N$ grouped by Fourier mode.
We write $\eta := [\mbf{c}_{-P/2+1}; \ldots;\mbf{c}_{P/2}]$ for the
$NP$-component column vector of all unknowns.

Imposing the boundary condition \eqref{neu} means, for all surface points
$\xx = (\rho,\phi,z)$ in cylindrical coordinates,
$$
\snp \sum_{j=1}^N c_{nj} \mintc \frac{\partial G_k}{\partial n_\xx}((\rho,\phi,z),\yy_j(\varphi)) e^{-in\varphi}
d\varphi
\; \approx \;
\snp \hat{f}_n(\rho,z) e^{in\phi}
~,\qquad (\rho,z)\in\gamma, \quad 0\le \phi< 2\pi
$$
By changing variable $\varphi$ to $\varphi-\phi$, and using orthogonality
of the modes $e^{in\phi}$ it is
easy to derive that, for each mode $n$ separately,
\be
\sum_{j=1}^N c_{nj} A'_n(\rho,z;\rho'_j,z'_j) \; \approx \; \hat{f}_n(\rho,z)
~,\qquad (\rho,z)\in\gamma
\label{ringbc}
\ee
should hold, where the target normal-derivative of the $n$th ``ring kernel'' from source point
$(\rho',z')$ to target $(\rho,z)$ (both lying in the $\rho$-$z$ plane) is
\be
A'_n(\rho,z;\rho',z') := \mintc \frac{\partial G_k}{\partial n_\xx}((\rho,0,z), (\rho',\varphi,z'))
e^{-in\varphi} d\varphi
~,
\label{ringkernelderiv}
\ee
where $n_\xx$ is the normal to $\gamma$ at $(\rho,z)$.
Finally, we enforce \eqref{ringbc} at a set of $M$ collocation
points $\{(\rho_m,z_m)\}_{m=1}^M$ on $\gamma$
to give the set of $P$ independent rectangular linear systems
\be
\sum_{j=1}^N A'_n(\rho_m,z_m;\rho'_j,z'_j) c_{nj}\;=\; \hat{f}_n(\rho_m,z_m)
~,\qquad m=1,\ldots,M~,
\quad n = -P/2+1,\ldots,P/2~.
\label{linsys}
\ee
Typically $M$ must be slightly larger than $N$ to ensure accurate
enforcement of the boundary condition.
As is standard with the MFS, these systems are
highly ill-conditioned, but if the sources are well chosen, are
numerically consistent and possess a small coefficient norm
$\|\mbf{c}_n\|$.
Assuming this norm is $\bigO(1)$, if a backward-stable least-squares
solve is used, machine precision accuracy can be reached
despite the ill-conditioning \cite{mfs}.
For each $n$ independently the system is now solved in this way.

The set of linear systems may be written as
\be
	\begin {bmatrix}
	A'_{-\frac{P}{2}+1}& &  \\
	 & \ddots & \\
 	&&A'_{\frac{P}{2}}
	\end {bmatrix} 
 \begin {bmatrix} 
        \mbf{c}_{-\frac{P}{2}+1} \\
         \vdots \\
        \mbf{c}_{\frac{P}{2}}
\end{bmatrix}  =    \begin {bmatrix} 
        \hat{f}_{-\frac{P}{2}+1} \\
         \vdots \\
         \hat{f}_{\frac{P}{2}}
\end{bmatrix}
~, \qquad\mbox{ summarized by }\quad
A'\eta = \hat{f}~,
\label{linsysb}
\ee
where $A'$ has diagonal block structure, each of the $P$ diagonal blocks
being a dense ill-conditioned
rectangular matrix of size $M$ by $N$.
Each of the $P$ systems will be solved independently in the least-squares sense
using standard dense direct methods based on the QR decomposition
(we use MATLAB's {\tt mldivide}).
The only remaining task is to fill their matrix elements and right-hand
side vectors.

Numerical evaluation of the RHS $\hat{f}_n(\rho_m,z_m)$ for each curve point $m$
is done by applying the $q$-node
periodic trapezoid rule quadrature \cite[Sec.~12.1]{LIE}
to \eqref{fhat}, i.e.,
\be
\hat{f}_n(\rho,z) \approx \frac{1}{q} \sum_{l=1}^q f(\rho,2\pi l/q,z) e^{-2\pi inl/q}
= 
({\cal F}^{-1}\mbf{f})_n~,
\qquad \mbox{ where }
\;\mbf{f}:= \{f(\rho,2\pi l/q,z)\}_{l=1}^q~,
\qquad n\in\mathbb{Z}~,
\label{fhatquad}
\ee
where ${\cal F}$ is the $q$-point discrete Fourier transform matrix.
The error in this quadrature formula (``aliasing'' error)
can be bounded by the sum of the magnitudes of the
exact Fourier coefficients $\hat{f}_n$ for which $|n|\ge q/2$
\cite[(3.10)]{PTRtref}.
Since, for an accurate BVP solution, $P$ must already be large enough that
coefficients are small beyond $P/2$, we may choose $q=P$
without losing accuracy due to quadrature in \eqref{fhatquad}.
Note that \eqref{fhatquad}
can be evaluated for all $-P/2<n\le P/2$ at once using a single
FFT in $\bigO(P \log P)$ time.
Since $\ui$ is analytic and the surface smooth, the Fourier data
$\hat{f}_n$ decays super-algebraically with $|n|$, and
thus we expect similar convergence with respect to $P$.

We use a similar idea to evaluate matrix elements of $A'_n$,
which we now describe.

\bfi 
\bc
\vspace{-2ex}
(a)\raisebox{-2in}{\ig{height = 2.0 in}{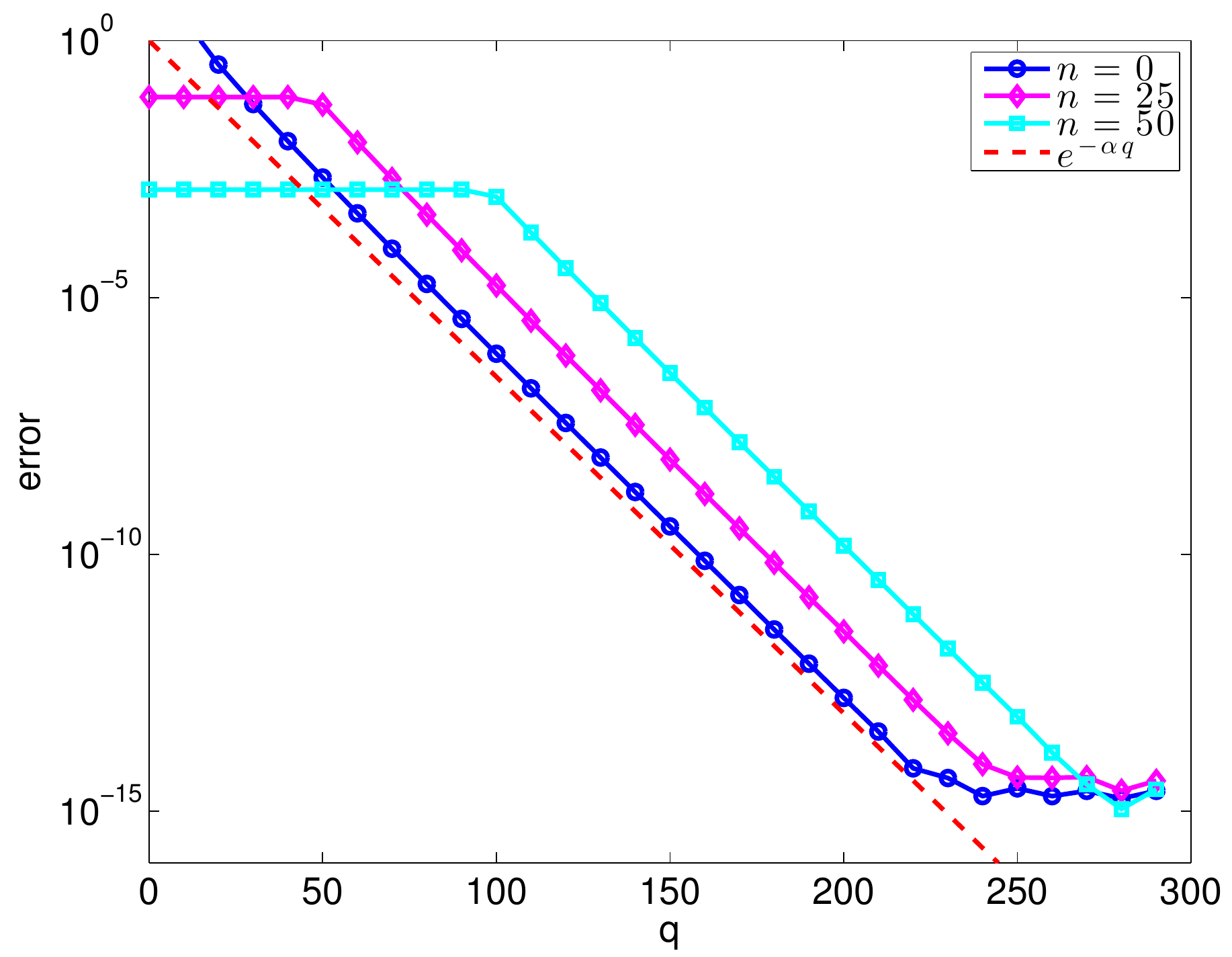}}
\qquad
(b)\raisebox{-2in}{\ig{height = 2.0 in}{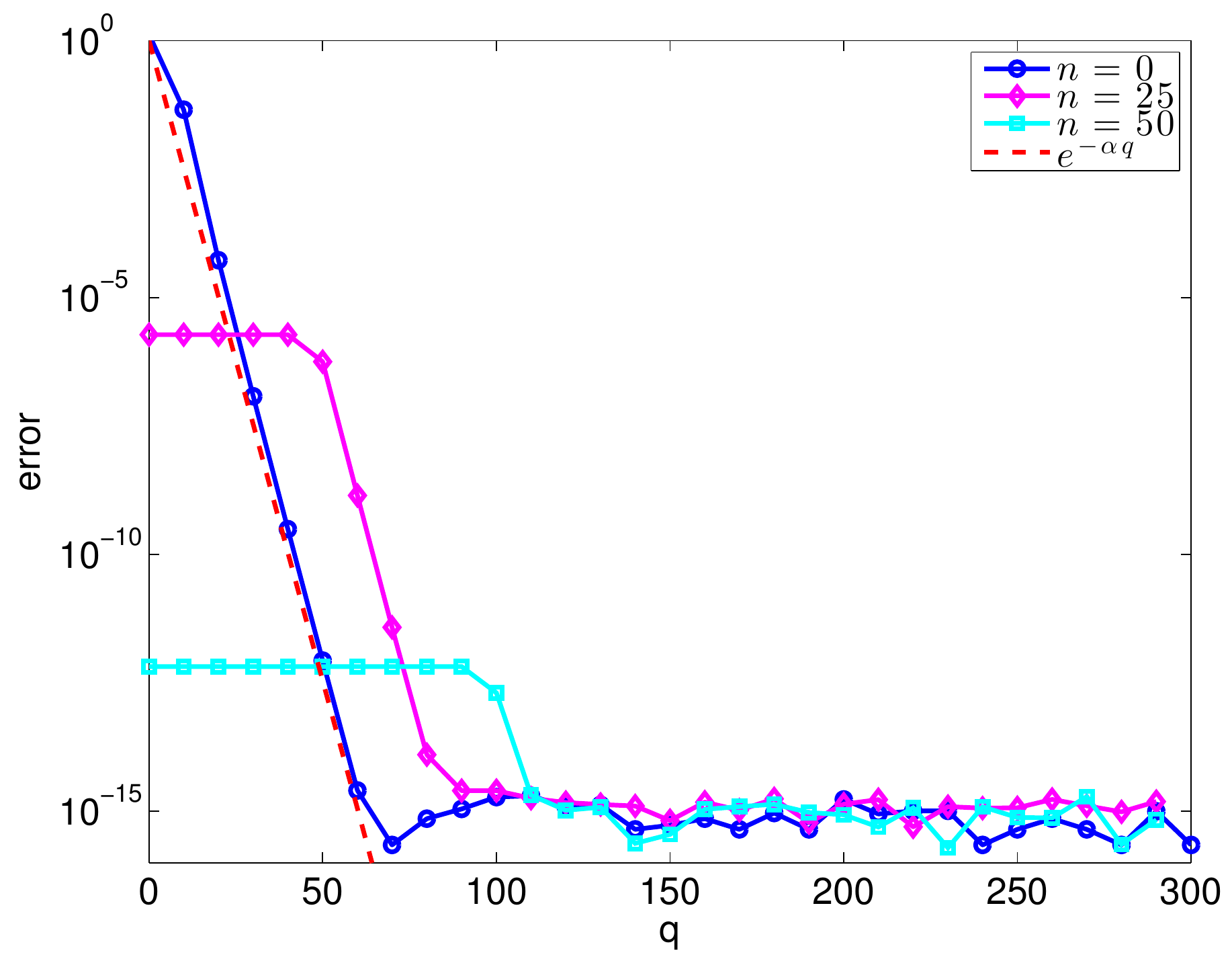}}
\ec
\ca{Convergence of the absolute error in 
evaluating the Helmholtz ring kernel \eqref{ringkernel},
for the various $n$ shown in the caption with $k = 10$.  
The shape we used is the one shown
in \fref{f:shapes}(a) with $N = 200, \tau = 0.1$. 
The $q$-node periodic trapezoid rule is used.
(a) Small distance: target $(\rho,z)=(0.5, 0.49)$ and source $(\rho',z')=(0.45, 0.44)$.
(b) Large distance: target $(\rho,z)=(0.6, 0.2)$ and source $(\rho',z')=(0.45, 0.44)$.
In each case the red line shows exponential convergence with rate equal to the
supremum of the $\alpha$ predicted by Thm.~\ref{t:qconv}.
}{f:qconv}
\efi

\subsection{Evaluation of the axisymmetric Helmholtz ring kernel}
\label{s:eval}

The $m,j$ matrix element of $A'_n$ in the linear system \eqref{linsys}
is given by the target normal-derivative of the Helmholtz ring kernel
$$
A'_n(\rho_m,z_m;\rho'_j,z'_j)
 =  \intc \frac{\partial G_k}{\partial n_{m}}((\rho_m,0,z_m), (\rho_j',\varphi,z_j'))
e^{-in\varphi} d\varphi
~,
$$
where $n_m$ indicates the normal to $\gamma $ at the point $(\rho_m,z_m)$.
For the transmission BVP (\sref{s:trans}) we will also need the
ring kernel value matrix element
\be
A_n(\rho_m,z_m;\rho'_j,z'_j)
 :=  \intc G_k((\rho_m,0,z_m), (\rho_j',\varphi,z_j'))
e^{-in\varphi} d\varphi
~.
\label{ringkernel}
\ee
Analytic series for this kernel are presented by Conway--Cohl \cite{cohl},
and evaluation methods for them in the BIE setting,
where the target may come very close to the source, are given in
\cite{Young,hao3daxi,helsing_axi,helsingIEEE15}.
However, in the MFS setting,
the target $(\rho,z)$ is separated from the source $(\rho',z')$
by a distance of at least a few times the quadrature node spacing
(see \sref{s:loc}),
and hence the periodic trapezoid rule quadrature
\be
A_n(\rho_m,z_m;\rho'_j,z'_j)
\; \approx\;
\frac{1}{q} \sum_{l=1}^q
G_k((\rho_m,0,z_m), (\rho_j',2\pi l/q,z_j'))
e^{-2\pi in l / q}
\label{ptr}
\ee
will be efficient for evaluating \eqref{ringkernel}.
In fact, we will have exponential convergence with a rate controlled
by the source-target separation in the $\rho$-$z$ plane,
as the following theorem shows.

\begin{thm} 
The $q$-node
periodic trapezoid rule applied to the Helmholtz ring kernel
or it's target normal derivative
is exponentially convergent with known rate $\alpha$.
Namely, for any $\alpha < \im \cos^{-1} \left (\frac { (\rho - \rho') ^2 + (z - z')^2} {2\rho \rho'} + 1\right) $,
there is a constant $C$ independent of $q$ such that
\be
\left|
\mintc G_k((\rho,0,z), (\rho',\varphi,z'))
e^{-in\varphi} d\varphi
-
\frac{1}{q} \sum_{l=1}^q
G_k((\rho,0,z), (\rho',2\pi l/q,z'))
e^{-2\pi in l / q}
\right|
\;\leq\;
Ce^{-\alpha q}
\ee
for all sufficiently large $q$,
and the same holds when $G_k$ is replaced by $\frac{\partial G_k}{\partial n_{\xx}}$,
$n_\xx$ being the normal to $\gamma$ at $(\rho,z)$.
\label{t:qconv}
\end{thm} 
\begin{proof}
The theorem of Davis \cite{davis59} states that
the periodic trapezoid rule is exponentially convergent
for a periodic analytic integrand.
Specifically, for the periodic interval $\varphi \in [0,2\pi)$,
if an integrand $f(\varphi)$ can be analytically continued off the real axis
to a bounded $2\pi$-periodic analytic function in the strip
$|\im \varphi|\le \alpha$, then
$$
\left| \mintc f(\varphi) d\varphi - \frac{1}{q} \sum_{l=1}^q
f(2\pi l /q) \right|
\; \le \;
C e^{-\alpha q}
$$
holds for some constant $C$ and all sufficiently large $q$.
It only remains to understand the region of analyticity of
the integrands
$$
f(\varphi) := G_k(\xx,\yy) e^{-in\varphi}
\qquad \mbox{ and } \quad
f(\varphi) := \frac{\partial G_k}{\partial n_\xx}(\xx,\yy) e^{-in\varphi}~,
\qquad
\mbox{ where } \xx=(\rho,0,z), \quad \yy = (\rho',\varphi,z'),
$$
with respect to the variable $\varphi$.
The Green's function \eqref{G} is analytic unless the distance between
source and target vanishes, ie
$$
0 = |\xx-\yy| = \sqrt { (\rho - \rho' \text{ cos} \varphi)^2 + (\rho' \text{ sin} \varphi)^2+ (z-z')^2 }
~.
$$
Solving for $\varphi$ gives locations of the
singularities in the complex $\varphi$ plane,
$$\varphi = \cos^{-1}
\left (\frac { (\rho - \rho') ^2 + (z - z')^2} {2\rho \rho'} + 1\right)
~.
$$
These occur with $2\pi$ periodicity in the real direction, in pairs
symmetric about the real axis.
Thus the two integrands $f(\varphi)$ are analytic and bounded in any strip
$|\im \varphi|\le \al$ for $\al$ given in the statement of the theorem.
Applying the Davis theorem completes the proof.
\end{proof}

Note that in the limit of small separation, by expanding the cosine, the
maximum convergence rate guaranteed by Thm.~\ref{t:qconv} is
$$\alpha \approx \frac{\sqrt{(\rho-\rho')^2 + (z-z')^2}}{\rho}
~,
$$
that is, equal to ratio of the source-target separation in the $\rho$-$z$ plane
to $\rho$.

In \fref{f:qconv} we demonstrate Thm.~\ref{t:qconv} numerically
in a typical axisymmetric MFS setting.
Panel (a) shows the error convergence in evaluating \eqref{ringkernel},
i.e.\ one matrix element of $A_n$,
for a small source-target separation of 0.071.
The convergence rate is slow, independent of $n$, and
observed to match very closely the upper bound on the rate $\alpha$
predicted by the theorem.
Panel (b) shows a much larger distance, and
much faster convergence, again with rate independent of $n$
and as predicted by the theorem.
The results for the normal derivative ring kernel \eqref{ringkernelderiv}
are similar.

Notice that the value of $q$ after which converge starts depends on
the Fourier mode $n$,
and that, as $n$ increases this adds roughly $|n|$
to the $q$ value at which convergence to machine precision is reached.
We explain this as follows.
In the Davis theorem the constant $C$ asymptotic to
$4\pi\sup_{\im\varphi = \pm \alpha} |f(\varphi)|$, i.e.\ a bound
on the size of $f$ in the strip \cite[Thm.~12.6]{LIE}.
The factor $e^{-in\varphi}$ in the integrand contributes a maximum
value within the strip of $e^{\al |n|}$, which is the dominates $C$.
Thus around $|n|$ extra quadrature nodes are needed to reach the error
level for $n=0$.

\begin{rmk}
Since the individual terms in the quadrature sum \eqref{ptr} are
of size $O(1)$ for a typical source-target separations in
a unit-sized obstacle, we can get convergence to machine precision in
the absolute error,
but not the relative error, of each matrix element of $A'_n$ or $A_n$.
However, absolute error is the relevant quantity controlling
final accuracy since
the solution field is summed over Fourier modes \eqref{ubormfs}.
\end{rmk}

In the rest of this paper, we fix $q$ to the same value for filling
all matrix elements in $A'_n$ or $A_n$, determined by a convergence
test as shown in \sref{s:borconv}.
It would be possible to use that above theorem to choose $q$
differently to be more optimal for each matrix element (source-target
distance); however, since the matrix filling is not a large part
of the total solution time there would not be much benefit.
For evaluation of the solution potential \eqref{ubormfs} on other
periodic images of the obstacle, we find that a fixed value $q=P$ is adequate.

\bfi 
(a)\raisebox{-1.4in}{\ig{width=2in}{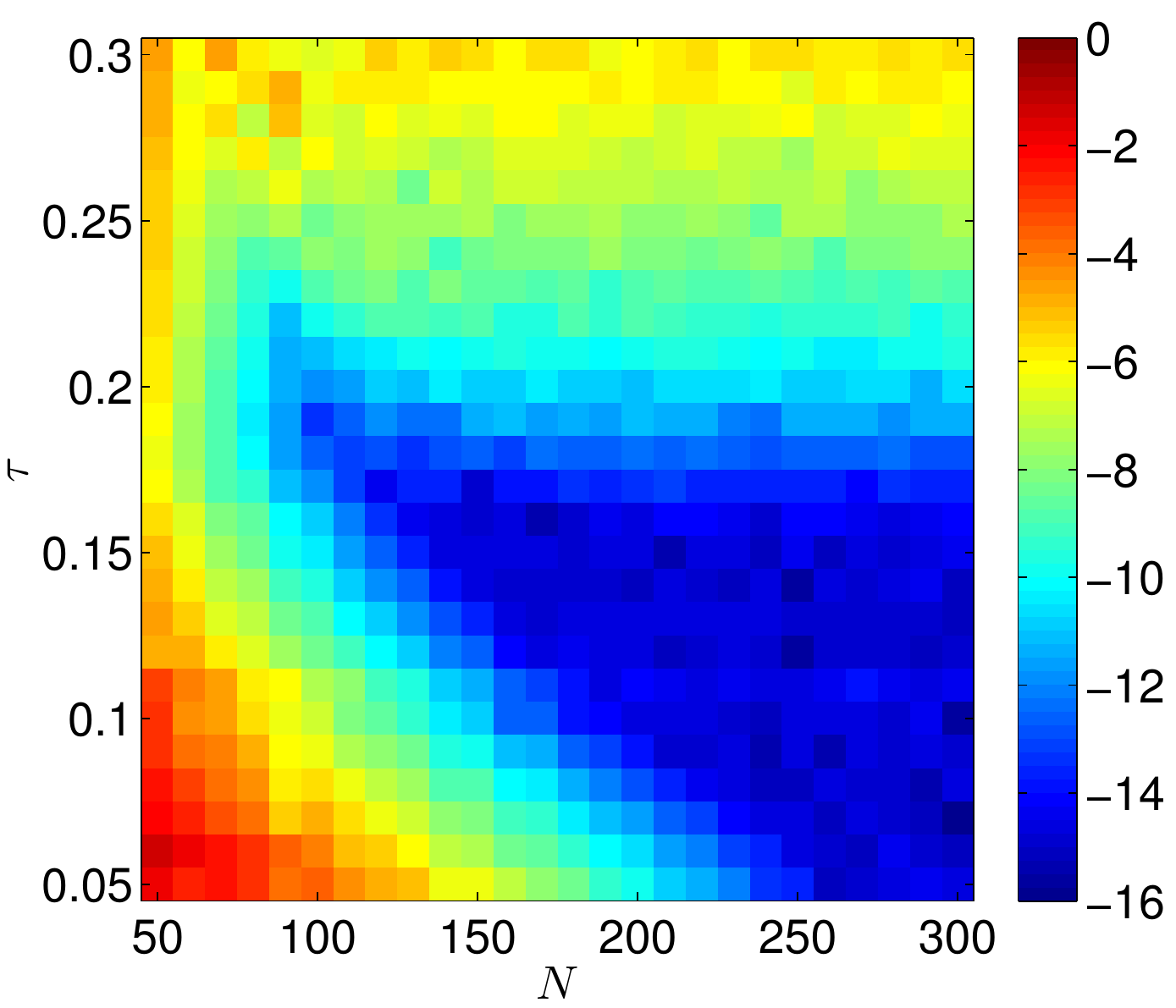}}
(c)\raisebox{-1.4in}{\ig{width=2in}{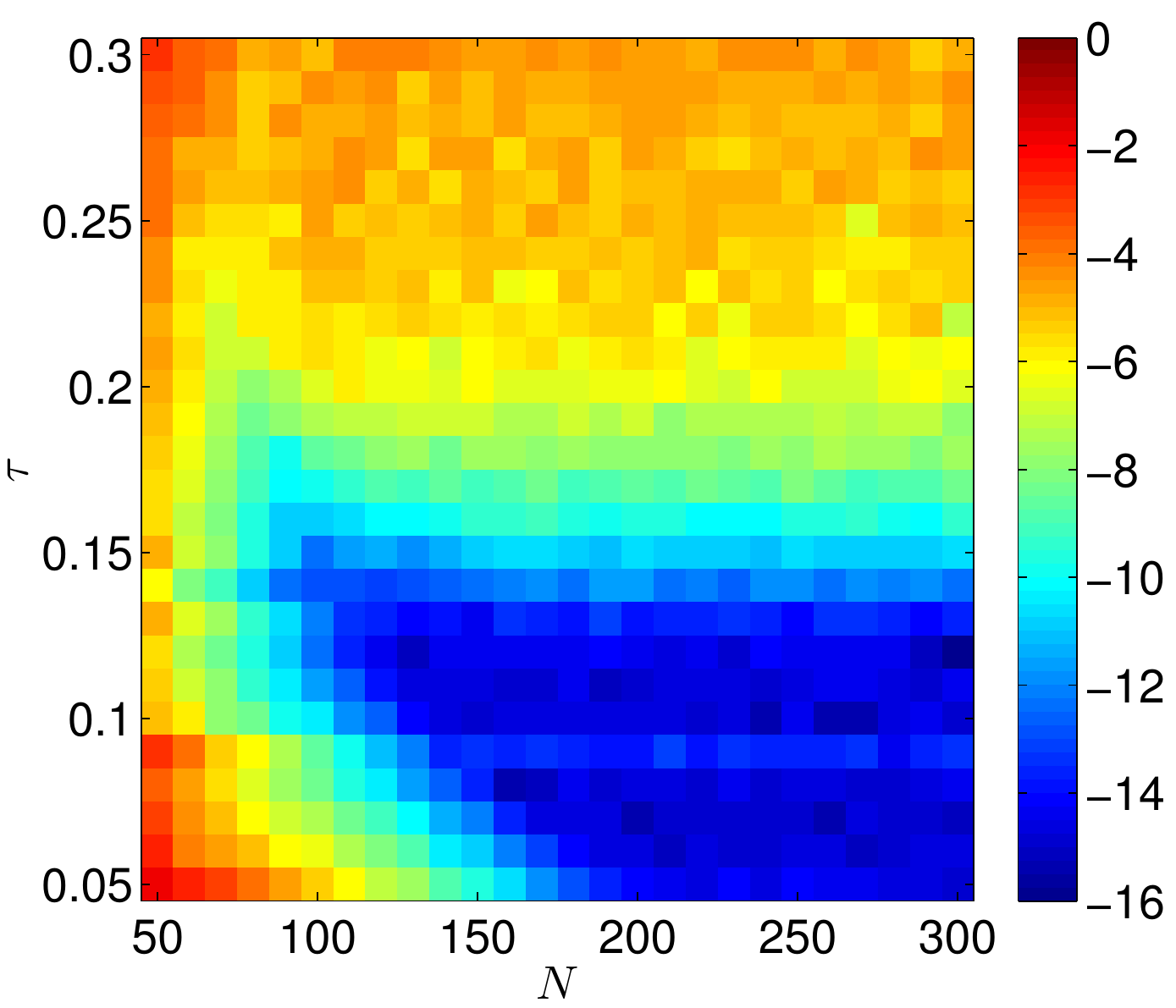}}
(e)\raisebox{-1.4in}{\ig{width=2in}{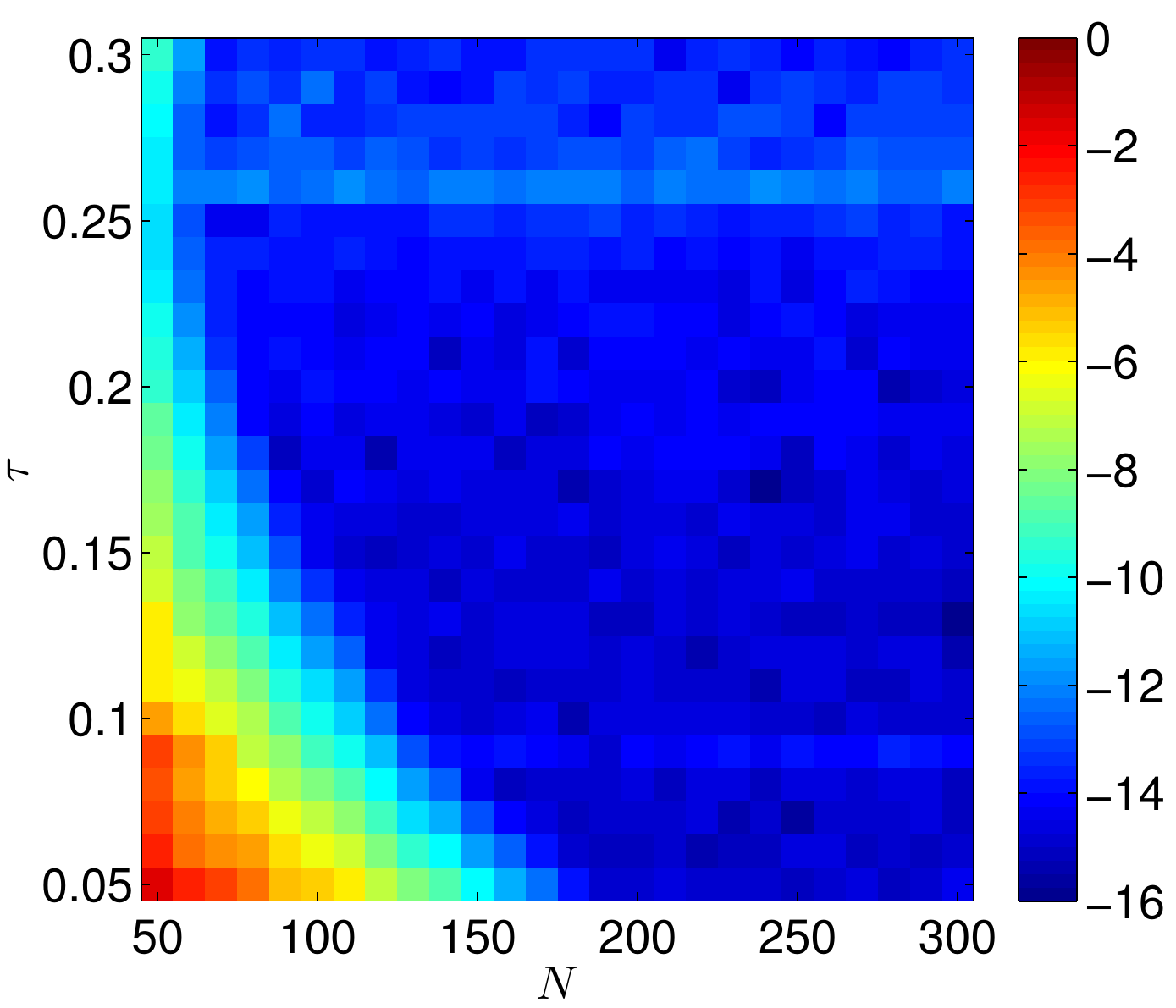}}
\\
(b)\raisebox{-1.4in}{\ig{width=2in}{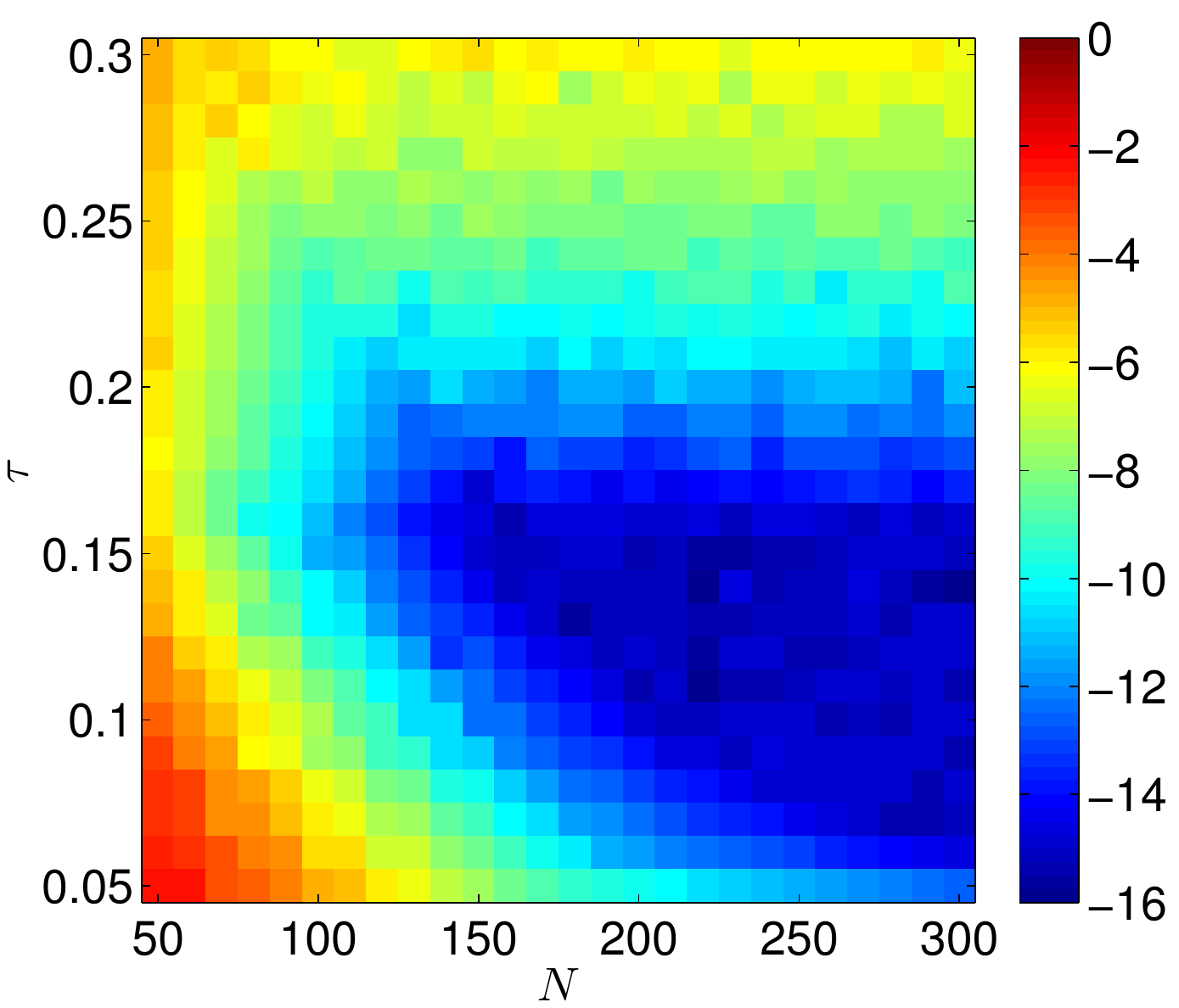}}
(d)\raisebox{-1.4in}{\ig{width=2in}{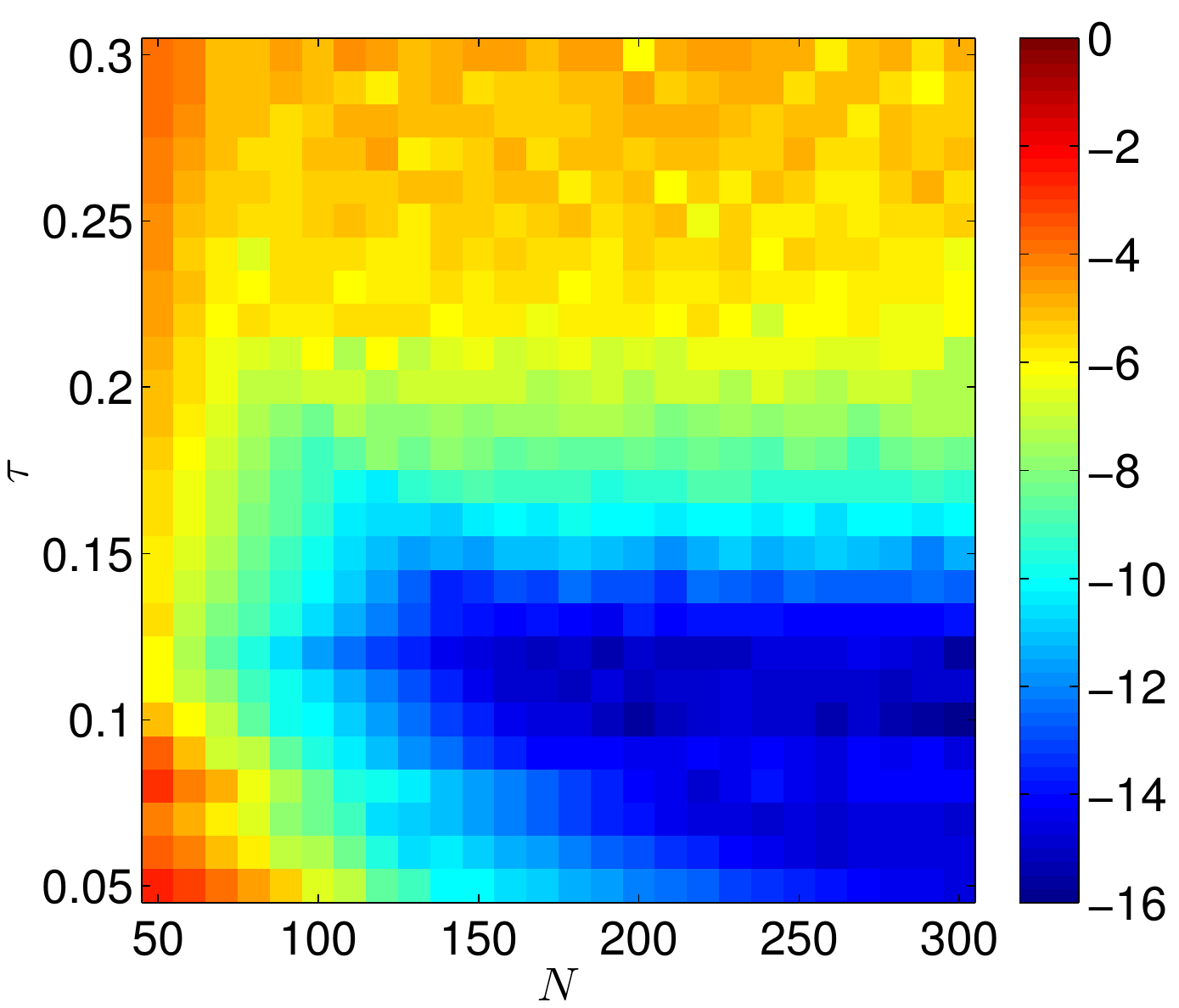}}
(f)\raisebox{-1.4in}{\ig{width=2in}{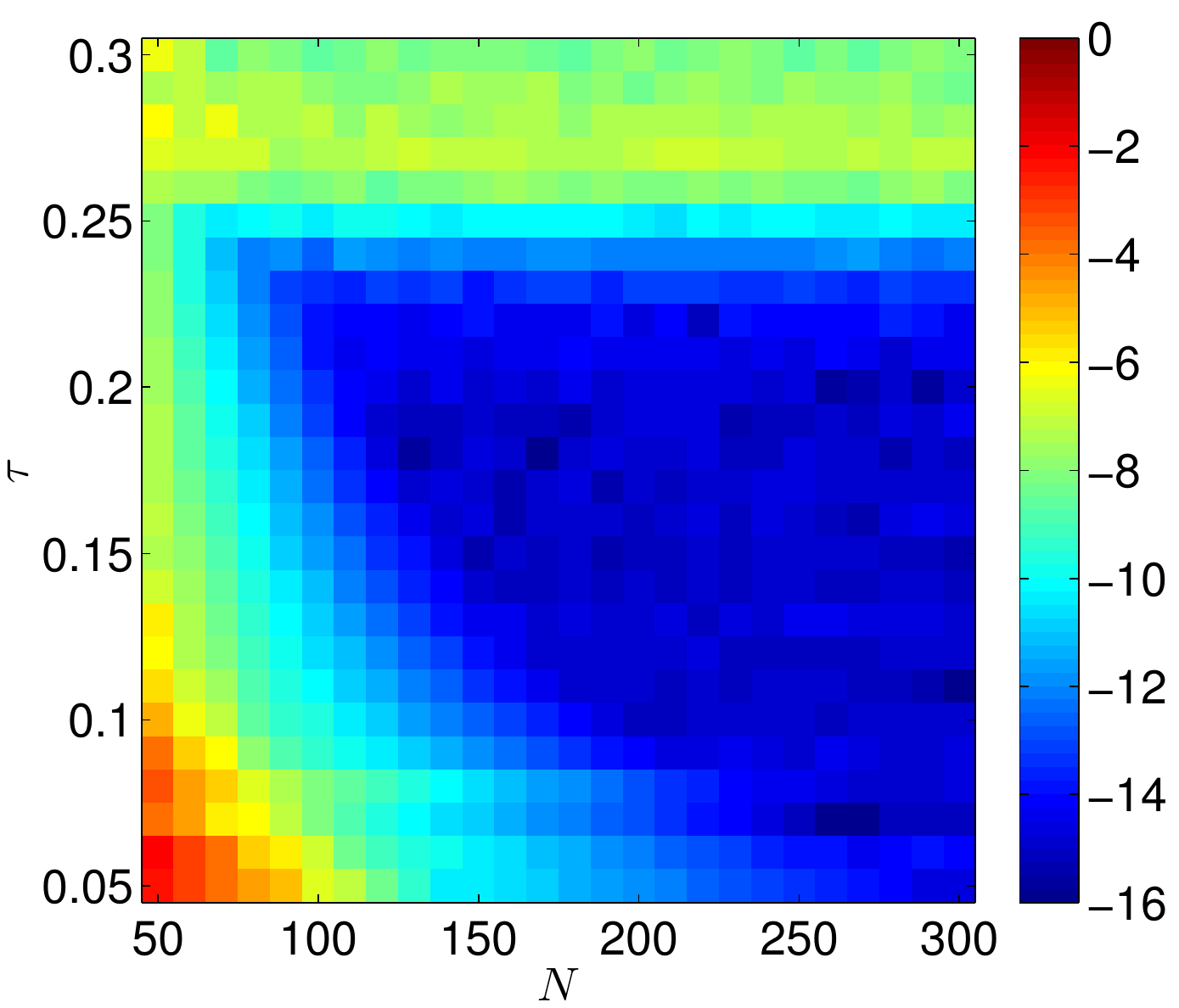}}
\ca{Convergence of error at a single distant point in the solution of Neumann Helmholtz scattering
problems at $k=10$ via the MFS for various source point displacement schemes,
for the boundary curve the ``smooth'' shape of \fref{f:shapes}(a),
$f(\theta) = 1 + 0.3 \cos 4 \theta$ for $\theta \in [0,\pi]$.
The top row of plots shows the error for solution of a 2D BVP using a boundary
curve reflected to form a closed curve $\theta \in [0,2\pi]$;
the bottom row shows error for the corresponding 3D axisymmetric BVP.
$N$ is the number of source points in $\theta \in [0,\pi]$.
(a) and (b): constant displacement from surface points along the normal vector.
(c) and (d): scaling the displacement in proportion to the ``speed'' $|s(t)|$.
(e) and (f): constant displacement in the imaginary parameter direction.
In each plot ``aliasing'' error due to sources too close to the boundary
dominates in the bottom left, and growth of error in the vertical direction
is caused by round-off due to growing coefficient norms $\|\mbf{c}\|$.
}{f:srcloc}
\efi
%
%

\subsection{Choice of source point locations}
\label{s:loc}

The performance of the MFS depends critically on the choice
of the curve $\Gm$ on which the
$N$ sources lie \cite{mfs_review,doicu}.
For analytic boundaries
there are strong theoretical results.
In this case,
Katsurada \cite{Ka90}
proved that, in exact arithmetic,
the MFS has exponential convergence for the 2D Dirichlet Laplace
($k=0$) interior BVP, with $M=N$
and points chosen equally spaced in the parametrization of the
curves $\Gamma$ and $\Gm$.
This was generalized to all analytic kernels
(thus including the Helmholtz kernel) by Kangro \cite{kangro2d},
and recently to the 3D Maxwell case when the source points $\{\yy_j\}$
are the nodes of an exponentially convergent quadrature on $\Gm$
\cite{kangro3d}
(presumably a similar proof would apply for Helmholtz).

However, in floating-point arithmetic another constraint arises,
associated with the exponentially large condition number of
the MFS system matrix.
Namely, the coefficient norm $\|\mbf{c}\|$ should remain small, i.e.\ $O(1)$.
It is conjectured that this happens if and only if
the exterior solution $u$ can be continued as a regular solution
to the PDE inside $\Omega$ up to and including $\Gm$,
in other words if $\Gm$ encloses all of the singularities in
the continuation of $u$ \cite[Conj.~12]{mfs}.
This is suggested by theory on the continuous
first-kind integral equation \cite[p.~1238]{Ky96} \cite[Thm.~2.4]{doicu},
and numerical results
where the singularities are known (via the Schwartz function) \cite{mfs}.
If $\Gm$ is placed too far from the surface, it will not enclose
these singularities, and exponential blowup of $\|\mbf{c}\|$ results,
causing rounding error which limits the solution error for $u$.

Now we compare the $N$-convergence of the error in our setting, for
various different methods for choosing source locations.
We use the notation that the sources
$\yy_j:=(\rho_j,z_j)$, $j=1,\ldots, N$ lying on a curve $\Gm$
are displaced from $N$ surface points $\xx_j:= (\rho(t_j),z(t_j))$
with parameters $t_j = \pi(j-1/2)/N$, $j=1\ldots,N$,
uniformly spaced on the generating curve $\gamma$.
The methods we compare are:
\ben
\item[(a)] displacement by a constant distance $\tau$ in the normal direction,
\item[(b)] displacement by a distance $\tau s(t_j)$ in the normal direction,
where $s(t)$ is the speed function, and
\item[(c)] displacement in the ``imaginary direction'' by complexification of
the boundary parametrization, i.e.\ $\yy_j:= (\rho(t_j+i\tau),z(t_j+i\tau))$.
This is a simplification of methods from \cite{mfs}.
\een
Note that all methods depend upon the parametrization,
and we assume that one is used for which $\{\xx_j\}$
provides a good quadrature scheme on $\gamma$ (combined with its reflection).


In \fref{f:srcloc} we compare these three methods for solving
the Neumann Helmholtz scattering problem in the smooth shape shown
in \fref{f:shapes}(a), using a dense direct linear 
In fact we compare the 2D case of plane-wave scattering from the
curve (closed by combining with its reflection about the $z$ axis),
against the results for the 3D axisymmetric case, 
using a dense direct solve for each Fourier mode.
Two conclusions stand out:
1) the 2D BVP is an excellent indicator of error performance in the
3D axisymmetric case,
and
2) the complexification method is at least as good as the
other methods at small $\tau$, and is more forgiving at large $\tau$.
Thus, in the rest of this work we choose the
complexification method.

The tested wavenumber $k=10$ is not high; we find that at $k=30$
the blow-up at large $\tau$ is more severe
even for the complexification method.
Thus we choose $\tau$ empirically so that it gives rapid $N$ convergence
but does not blow up at high $k$. This we assess via experimentation with
the 2D BVP, since is very rapid to run (taking a fraction of a second
per solve), and is demonstrated to reflect the 3D performance.
In future work we will present an automated method for choosing
$\tau$ and $N$.

\begin{rmk}
The monopole (``single-layer'') MFS representation that we use fails to
be a complete representation of exterior radiative potentials
if the surface on which the sources lie has an interior Neumann eigenvalue
equal to $k^2$ \cite[Sec.~2.1]{doicu}.
In practice, this can be detected by a poor boundary error ($\epsilon_1$
in \sref{s:borconv} below), and $\tau$ changed slightly. In our
experiments we have never had to make such a change.
\end{rmk}

\bfi 
(a)\raisebox{-1.5in}{\ig{width=2in}{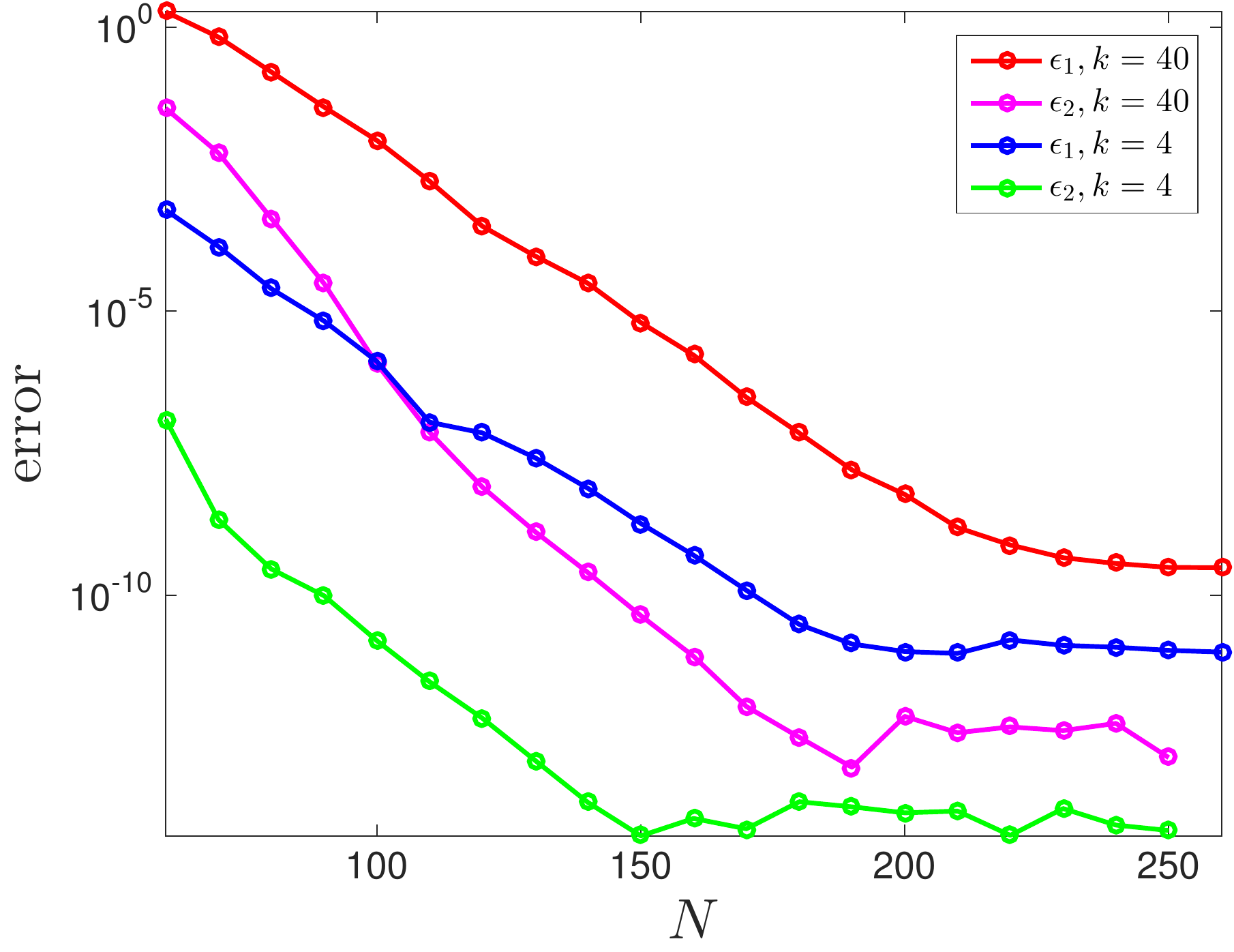}}
(d)\raisebox{-1.5in}{\ig{width=2in}{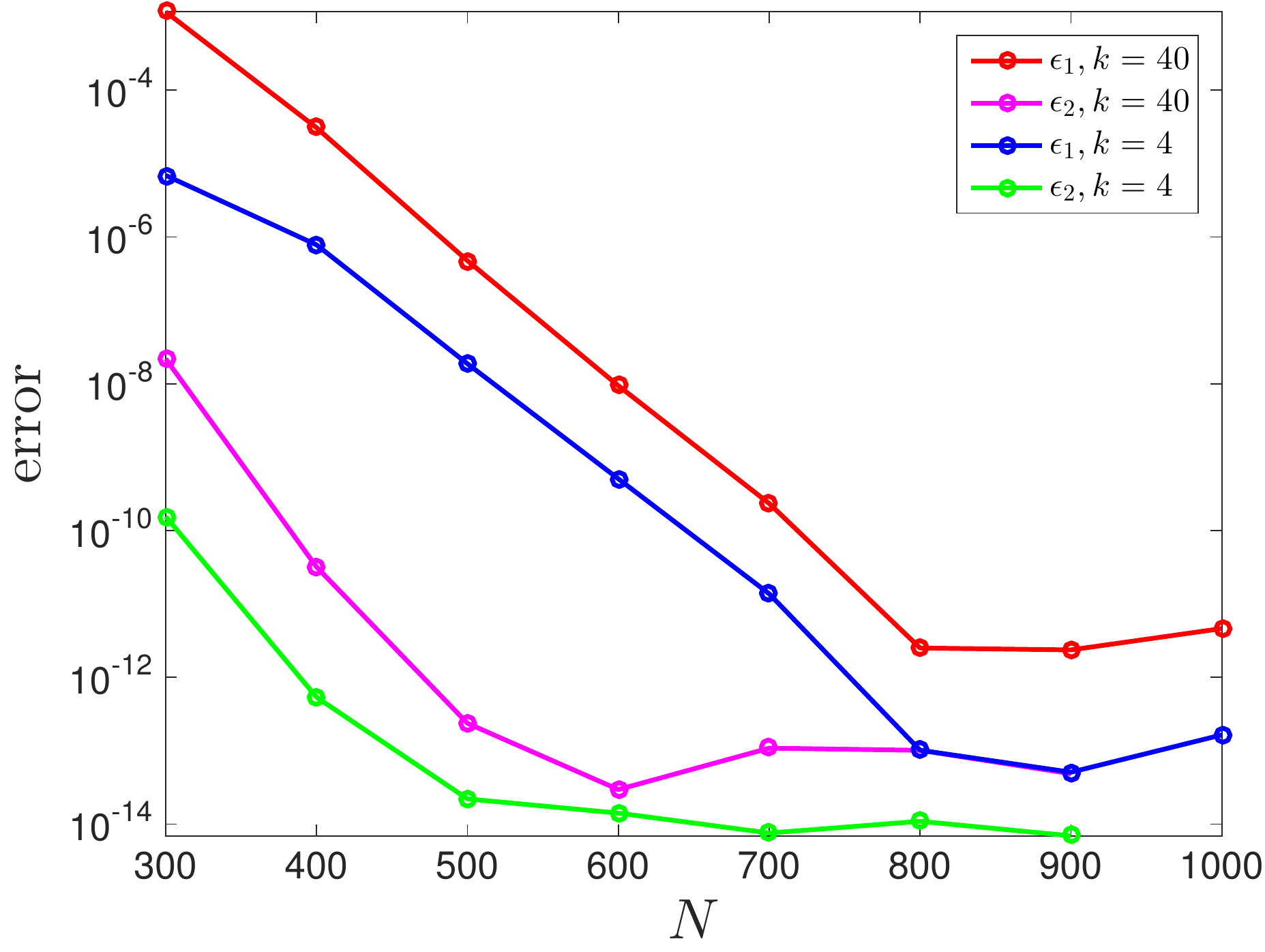}}
(g)\raisebox{-1.5in}{\ig{width=2in}{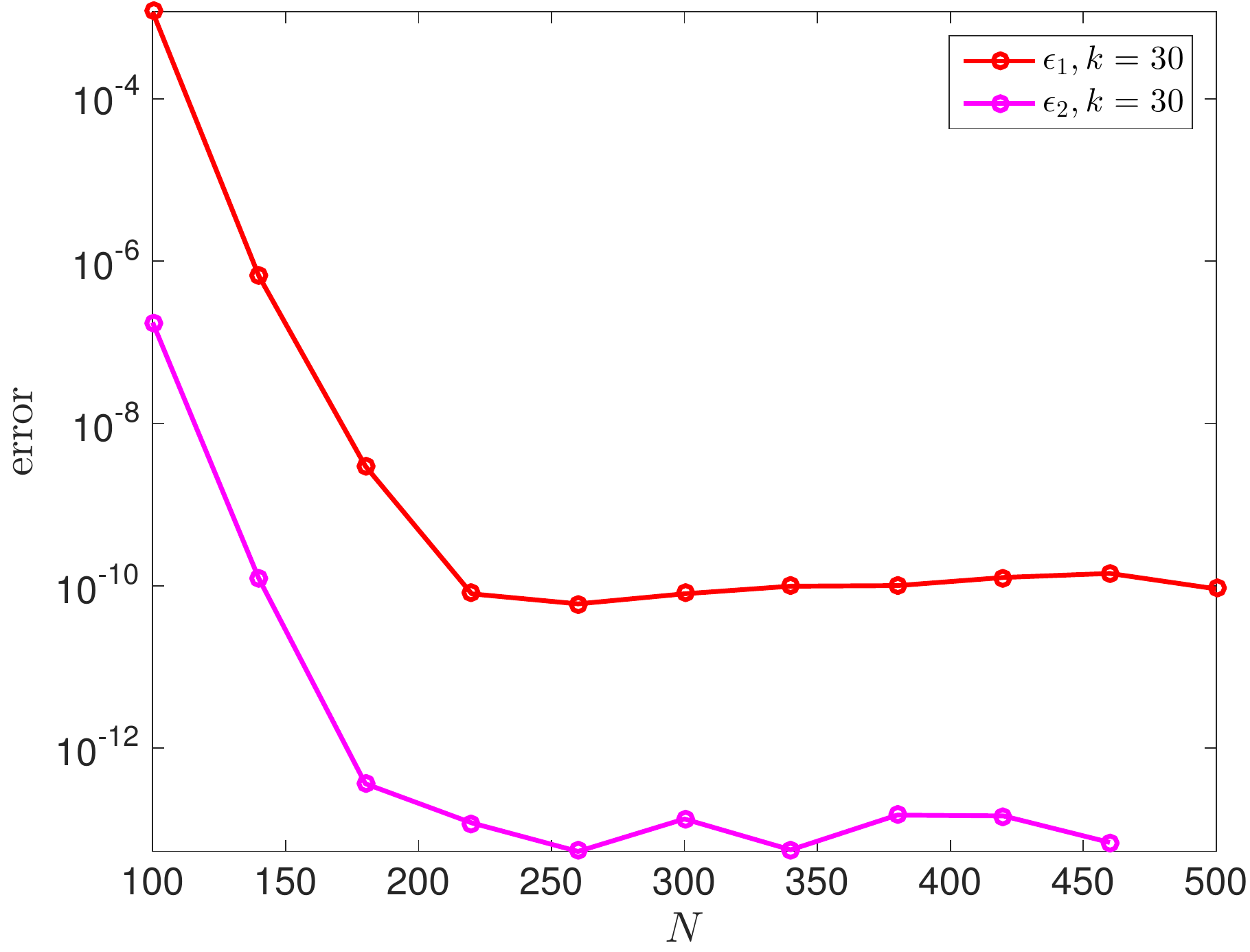}}
\\
(b)\raisebox{-1.5in}{\ig{width=2in}{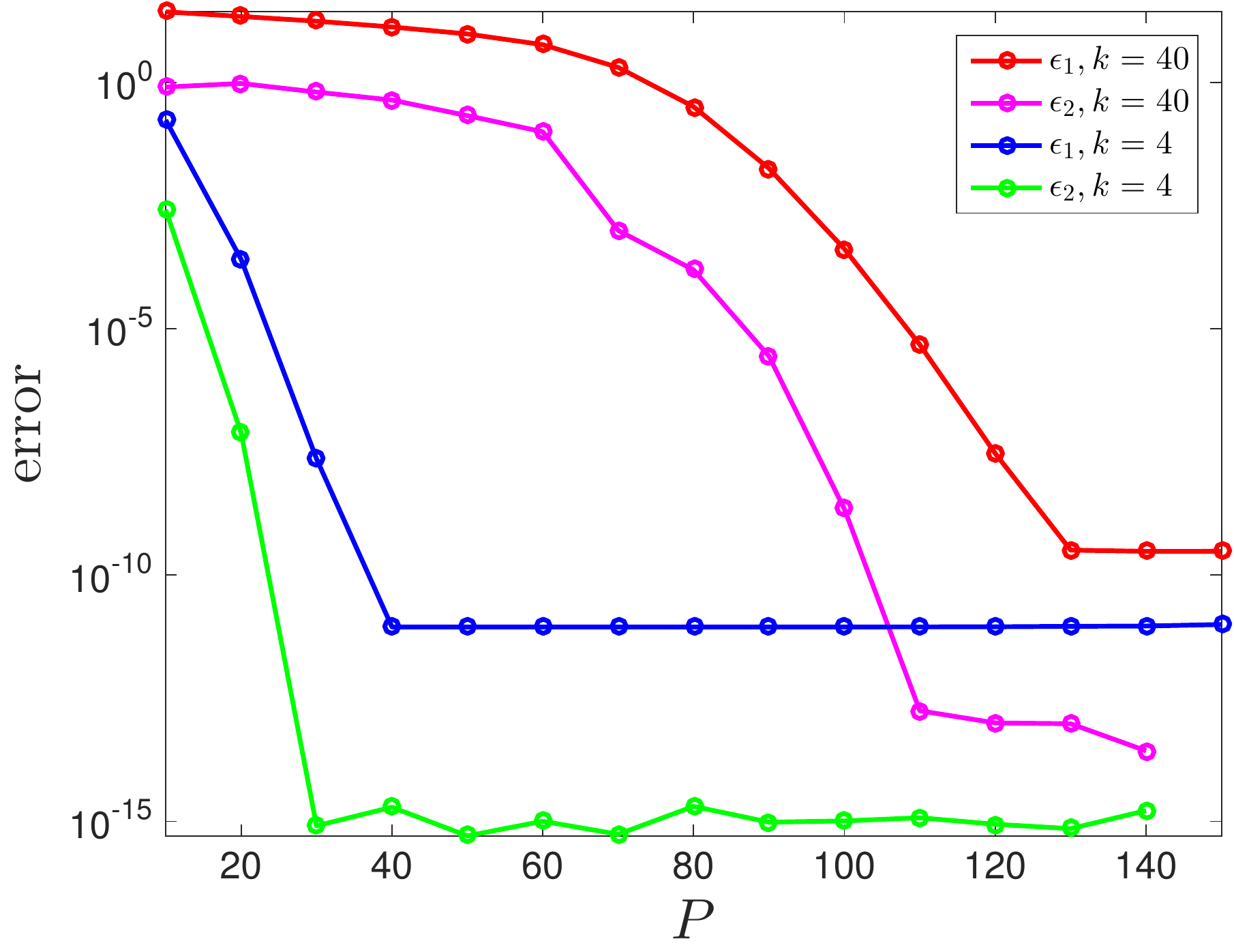}}
(e)\raisebox{-1.5in}{\ig{width=2in}{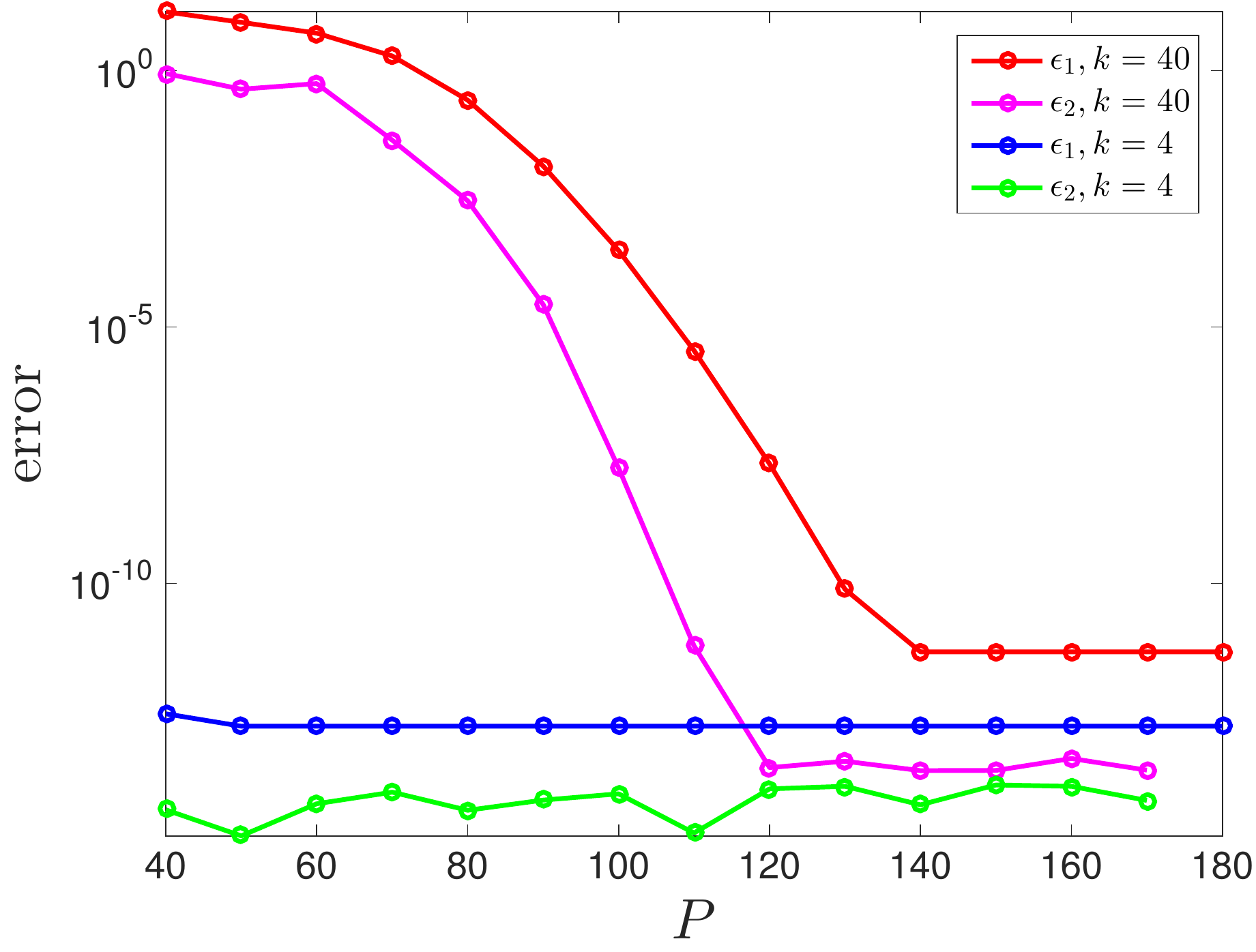}}
(h)\raisebox{-1.5in}{\ig{width=2in}{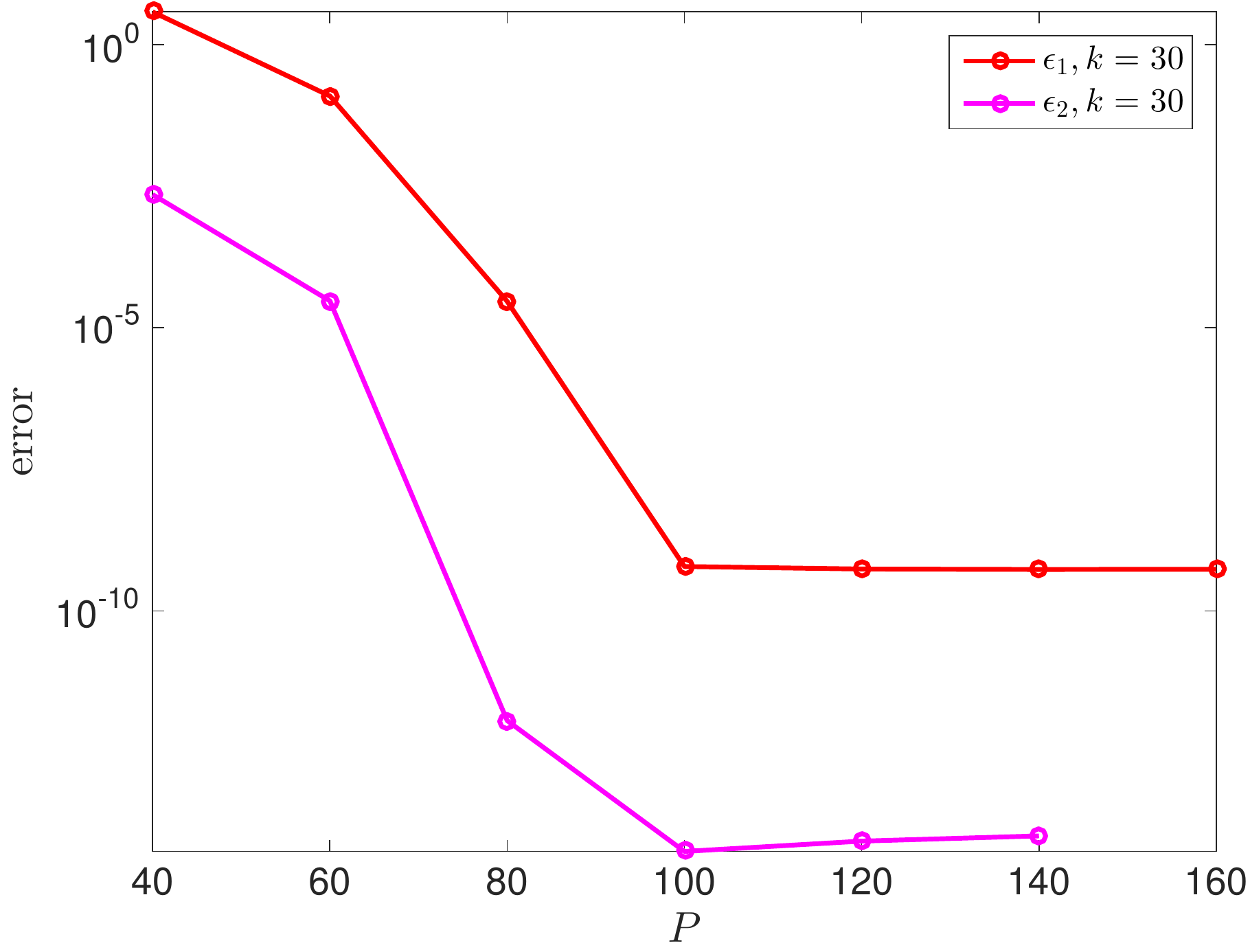}}
\\
(c)\raisebox{-1.5in}{\ig{width=2in}{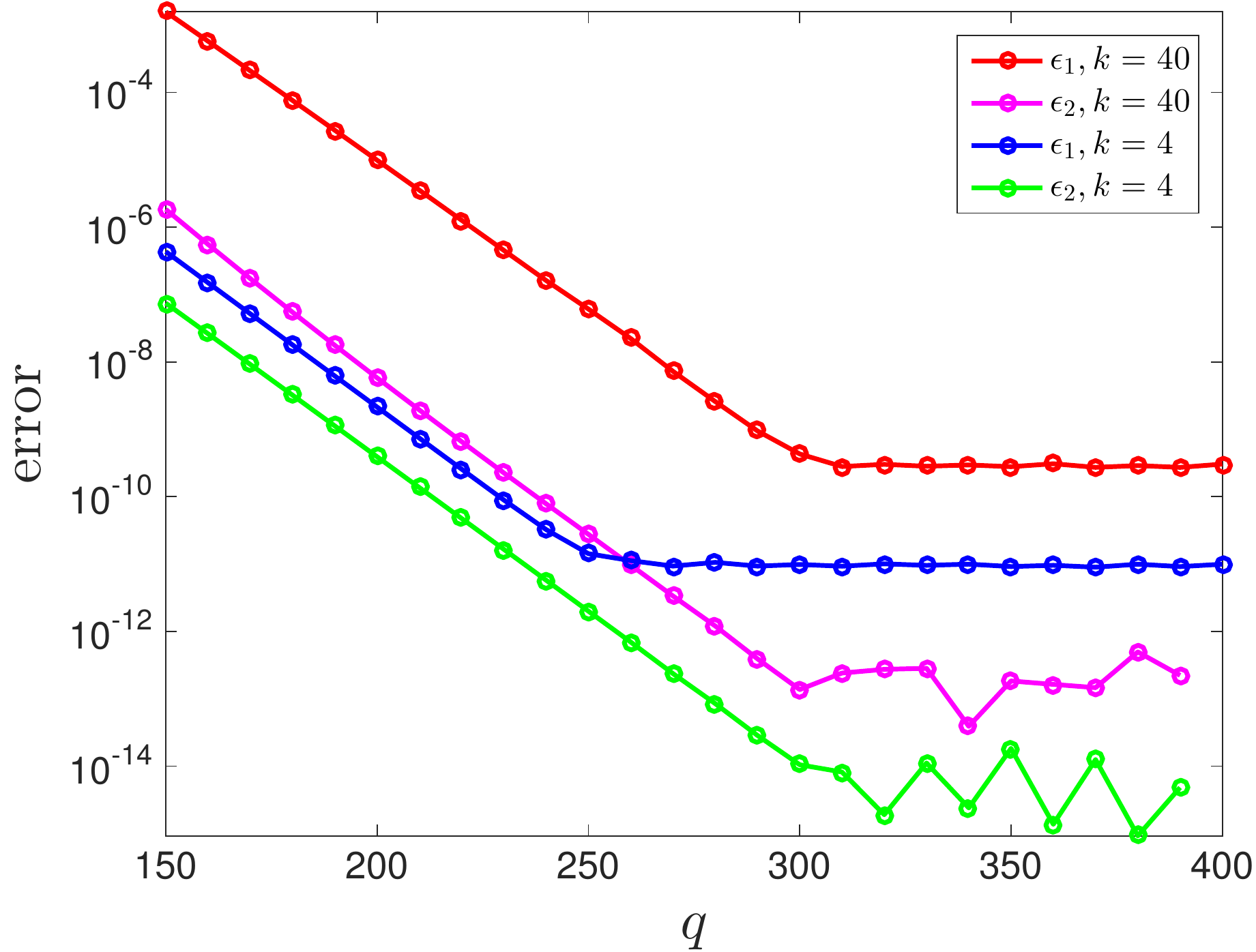}}
(f)\raisebox{-1.5in}{\ig{width=2in}{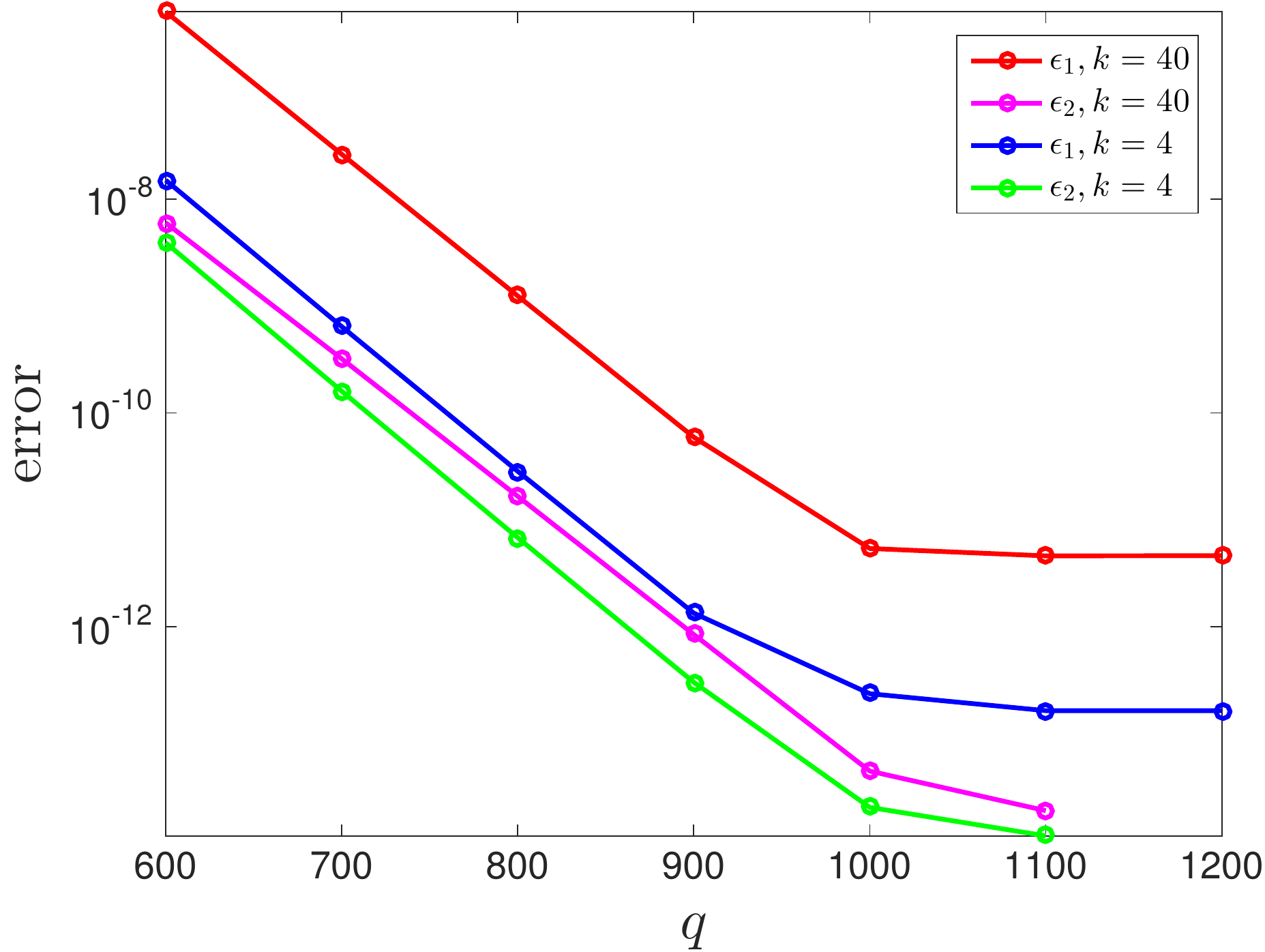}}
(i)\raisebox{-1.5in}{\ig{width=2in}{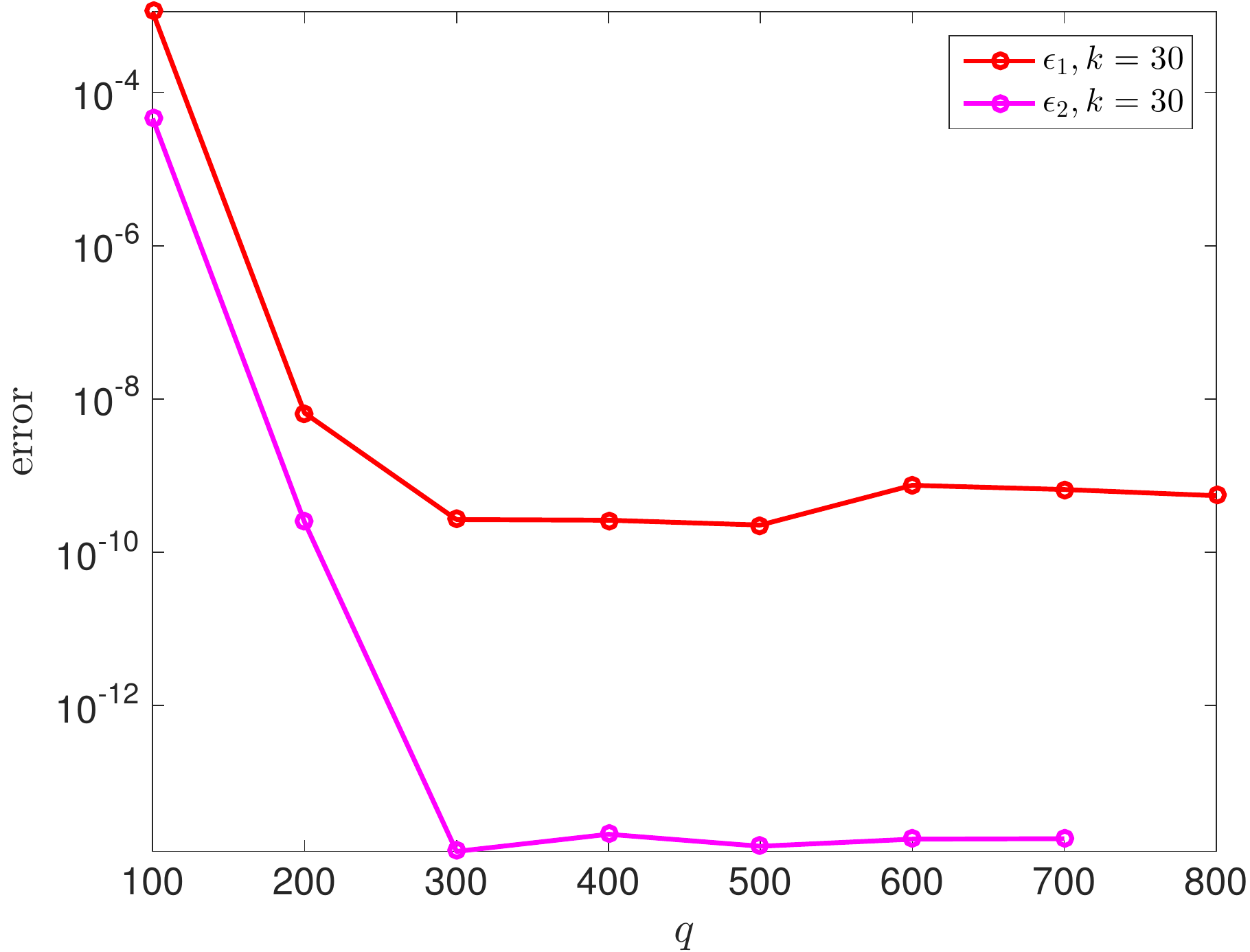}}
\ca{%
Error convergence for the 3D axisymmetric MFS solution of the Helmholtz
Neumann scattering BVP for low and high frequencies (wavenumber $k$
is shown in legends).
$\epsilon_1$ estimates the boundary condition error,
$\epsilon_2$ the solution error at a distant point.
Left column: ``smooth'' shape from \fref{f:shapes}(a),
with $\tau=0.1$.
(a) varying $N$, (b) varying $P$, (c) varying $q$.
The fixed (converges) parameter values are $N=260$, $P=150$, $q=400$.
Middle column: ``wiggly'' shape from \fref{f:shapes}(b),
with $\tau=0.03$.
Fixed values are $N=1000$, $P = 180$, $q=1200$.
Right column: cup shape from \fref{f:shapes}(c),
with $\tau=0.1$.
Fixed values are $N=500$, $P = 160$, $q=800$.
In all cases $M \approx 1.2 N$.
}{f:borconv}
\efi

\subsection{Isolated obstacle Neumann scattering convergence tests}
\label{s:borconv}


Thus far we have made three approximations, each of which involves a
parameter:
1) the MFS approximation involving $N$ source points,
2) the Fourier series truncation to $P$ terms,
and 3) the quadrature evaluation of the ring kernel via $q$ periodic
trapezoid nodes.
We now show convergence results for these parameters
in the context of solving the one-obstacle scattering problem
\eqref{helm}, \eqref{neu} and \eqref{somm}.
The first row of \fref{f:borconv} shows the convergence with respect
to $N$, the second row with respect to $P$, and the last row
with respect to $q$.
In each case the parameters not under convergence study are fixed
at their converged values.
Two types of error are shown:
\bi
\item $\epsilon_1$: 
Absolute $L^2$ error in the boundary
condition \eqref{neu}, estimated at a grid of $128\times 128$
points on $\pO$ very few of which coincide with collocation points
$\xx_i$,
and
\item $\epsilon_2$: 
relative error in the value of the scattered
potential $u$ at a distant point
near to $(10,10,10)$,
compared to its converged value.
\ei 
\fref{f:borconv} is consistent with exponential
convergence in all three parameters,
with the values where convergence starts being larger
for large wavenumber $k$.
The convergence of the two error measures is similar,
usually differing by a fixed constant,
and saturating at different values.
Even in the case of a resonant object at high frequency (the
plots in the right column), 9 digits are reached in the boundary condition error, and 12 digits in the solution $u$ at a distant point.

\subsection{The transmission case}
\label{s:trans}

Now we briefly describe how the above axisymmetric MFS scheme
is adjusted for the transmission case
\eqref{helmi}--\eqref{eq:neup1.4}.
As before, the scattered wave $u$ is represented by
the ring kernel MFS source sum \eqref{ubormfs}
with source locations $(\rho'_j,z'_j)$ inside the obstacle.
In addition the scattered wave $u^-$ inside the obstacle is
represented by a similar sum
$$
u^- \;\approx\;
\sum_{j=1}^N \snp c^-_{nj} \Phi_{nj}^-,
\qquad
\mbox{ where }
\quad
\Phi^-_{nj}(\xx) \; := \;
\mintc G_{\km}(\xx,\yy^-_j(\varphi)) e^{-in \varphi} d\varphi~,
\qquad \yy^-_j(\varphi) := (\rho''_j,\varphi,z''_j)~,
$$
with sources $(\rho''_j,z''_j)$ lying outside the obstacle.
The set-up is as in \fref{f:geom}(c) (which shows a 2D sketch),
an example being the green points in \fref{f:shapes}(b).
These exterior source locations are chosen by negating $\tau$ in the
source point location algorithms in \sref{s:loc}.
Imposing the value and normal derivative matching conditions
\eqref{eq:neup1.3} and \eqref{eq:neup1.4} gives the linear system
\be
\mt{A}{-A^-}{A'}{-A'^-}
\vt{\eta}{\eta^-}
= \vt{\hat{f}}{\hat{f}'}
\label{trans}
\ee
where each of the four blocks is block diagonal, with rectangular
diagonal blocks
$A_n$, $A^-_n$, $A'_n$, $A'^-_n$ respectively, $-P/2<n\le P/2$.
The new blocks $A^-_n$ and $A'^-_n$ are identical to \eqref{ringkernel}
and \eqref{ringkernelderiv} respectively, but with $k$ replaced by $\km$,
$\rho'_j$ by $\rho''_j$, and $z'_j$ by $z''_j$.
The RHS data blocks $\hat{f}_n$ and $\hat{f}_n'$ are constructed
as before via \eqref{fhatquad} using
$f = -\ui$ and $f' = -\frac{\partial\ui}{\partial n}$
respectively.
Abusing notation somewhat, we will summarize both \eqref{linsysb}
and \eqref{trans} by
the linear system
\be
A_0 \eta = \hat{f}~.
\label{A0}
\ee
where $A_0$ is of size $MP$-by-$NP$ in the Neumann case,
and size $2MP$-by-$2NP$ in the transmission case.

\section{Periodizing scheme}
\label{s:per}

We choose a cuboid unit cell $\UU$
of sizes $e_x$ and $e_y$ in the $x$ and $y$ axes
respectively, and truncated in the vertical direction to 
$z\in[-z_0,z_0]$ such as to contain $\Omega$.
The rectangular walls enclosing $\UU$
are left $L$, right $R=L+(e_x,0,0)$, back $B$, front $F = B+(0,e_y,0)$,
and downwards $D$ at $z=-z_0$ and top $T$ at $z=z_0$;
see \fref{f:bor}(b).
All normals point in the positive coordinate directions.
We use the abbreviation $u_L$ to mean $u$ restricted to $L$,
and $u_{nL}$ to mean $\partial u/\partial n$ restricted to $L$.
We now reformulate the periodic BVP on $\UU$ alone, as in \cite{bonnetBDS,mlqp}.
Taking for simplicity the Neumann case, $u$ satisfies \eqref{helm}
in $\UU$, \eqref{neu},
and we match Cauchy data on the four side walls giving
\bea
u_R - \al u_L &=& 0
\label{qp1}
\\
u_{nR} - \al u_{nL} &=& 0\\
u_F - \bt u_B &=& 0 \\
u_{nF} - \bt u_{nB} &=& 0~,
\label{qp4}
\eea
which (because of the unique continuation property of an elliptic PDE)
is equivalent to quasiperiodicity \eqref{per},
and finally match Cauchy data to
Rayleigh--Bloch expansions \eqref{rb1}--\eqref{rb2}
on the top and bottom walls,
\bea
u(x,y,z_0) &=& \sum _{m,n\in \mathbf Z} a_{mn} \exp{i[\kx x + \ky y]}~,
\qquad (x,y,z_0)\in T \label{T} \\
u_z(x,y,z_0) &=& \sum _{m,n\in \mathbf Z} i\kz a_{mn} \exp{i[\kx x + \ky y]}~,
\qquad (x,y,z_0)\in T \label{nT} \\
u(x,y,-z_0) &=& \sum _{m,n\in \mathbf Z} b_{mn} \exp{i[\kx x + \ky y]}~, \qquad (x,y,-z_0)\in D  \label{D}\\
u_z(x,y,-z_0) &=& -\sum _{m,n\in \mathbf Z} i\kz b_{mn} \exp{i[\kx x + \ky y]}~, \qquad (x,y,-z_0)\in D~.
\label{nD}
\eea
The solution $u$ in $\UU$ is represented as
\be
u(\xx) \approx \sum_{j=1}^N \sum_{n=-P/2+1}^{P/2} c_{nj} \Phi^\tbox{near}_{nj}(\xx)
+ \sum_{l=0}^\pdeg\sum_{m=-l}^l d_{lm} j_l(kr) Y_{lm}(\theta,\phi)
\label{urep}
\ee
where the sum of the ring kernel over the nearest neighbors in the lattice is
\be
\Phi^\tbox{near}_{nj}(\xx) := \sum_{|m|,|n| \le 1}\al^m \bt^n
\mintc
G_k(\xx,\yy_j(\varphi) + m\eo + n\et) e^{-in\varphi} d\varphi
~.
\label{Phinear}
\ee
The unknowns are $\eta := \{c_{nj}\}_{j=1,\dots,N, -P/2<n\le P/2}$,
$\mbf{d} := \{d_{lm}\}_{l=0,\dots,\pdeg, |m|\le l}$,
and we restrict the Rayleigh--Bloch expansions
\eqref{T}--\eqref{nD}
to the $xy$-plane wavevectors of magnitude at most $\pi N_0$,
where $N_0$ is a convergence parameter,
so $\mbf{a} := \{a_{mn}\}_{(\kx)^2+(\ky)^2 \le \pi^2N_0^2}$
and $\mbf{b} := \{b_{mn}\}_{(\kx)^2+(\ky)^2 \le \pi^2N_0^2}$. 

\begin{rmk}[Choice of $z_0$]
If the vertical extent of $\Omega$ is not larger than the period,
then choosing $\UU$ roughly cubical
is efficient. Generally, $z_0$ must be at least
some distance (say, $1/4$ period)
above the extent of $\Omega$ in order that the Rayleigh--Bloch
expansions converge rapidly;
on the other hand if $z_0$ is too large, the
spherical harmonic convergence rate on the faces is reduced.
For high aspect ratio obstacles (not tested in this work),
either more local terms are needed in \eqref{Phinear},
or an elongated proxy surface could be used (in the style
of App.~A; also see \cite{BG_BIT}).
\end{rmk}

\subsection{Full linear system}
\label{s:sys}

We now build the full linear system that the stacked
column vector of all unknowns $[\eta;\mbf{d};\mbf{a};\mbf{b}]$ must satisfy,
by enforcing the boundary conditions, but also the quasiperiodicity
and upward and downward radiation conditions.

The first block row arises from enforcing \eqref{neu} on the boundary
nodes on $\pO$ as in \sref{s:mfs}, giving
\be
A\eta + \tilde{B} \mbf{d} = \hat{f}
\label{row1}
\ee
where
\be
A = A_0 + \Aelse,
\label{A}
\ee
$A_0$ being the direct self-interaction of the obstacle
as in \eqref{A0} (i.e.\ $A_0$ is zero apart from diagonal blocks $A^-_n$).
Rather than writing a long formula for the matrix elements
of $\Aelse$, it is more useful to describe its action:
$\Aelse\eta$ is the set of Fourier series coefficients of
the normal-derivatives of $u$ on the target rings
$\{\rho_m,z_m\}_{m=1}^M$ due to the eight phased ring sources
from the first term in \eqref{urep}, omitting the central
copy $m=n=0$.
In practice we may compute $\Aelse\eta$ efficiently as follows:
the quadrature \eqref{ptr} is equivalent to replacing
each ring kernel by $q$ point sources with strengths
given by the FFT of the coefficients $c_{nj}$ in the $n$ direction.
The FMM evaluates the potential on the  $8qN$ sources at a set of
$qM$ trapezoidal-node targets on rings on $\pO$,
and the FFT is finally used to convert back to Fourier coefficients
as in \eqref{fhatquad}.
A quadrature parameter $q=P$ is sufficient here, since interactions
in $\Aelse$ are distant.
The cost of applying $\Aelse$ (assuming $M=\bigO(N)$) is thus
$\bigO(NP\log P)$, although it is typically dominated
by the $\bigO(NP)$ of the FMM.

The matrix $\tilde{B}$ in \eqref{row1} has elements that
can approximated by the periodic trapezoid rule,
$$
\tilde{B}_{ni,lm} = \mintc j_l(kr_i)Y_{lm}(\theta_i,\phi) e^{-in\phi}d\phi
\approx \frac{1}{q} \sum_{s=1}^q
j_l(kr_i)Y_{lm}(\theta_i,2\pi s/q) e^{-2\pi in s/q}
~,
$$
where $r_i = \sqrt{\rho_i^2+z_i^2}$ and $\theta_i = \tan^{-1} z_i/\rho_i$
are the spherical coordinates of the $i$th boundary point on $\gamma$.
Note that if the axis of the obstacle is aligned with the $z$-axis of
the spherical harmonic expansion, simplifications apply
making $\tilde{B}$ sparse;
for more generality we leave it in the above form.
We fill the matrix $\tilde{B}$ once and for all at a given $k$, at
a cost $\bigO(NP\pdeg^2)$.

The remaining block rows of the full
linear system result from substituting 
\eqref{urep} into \eqref{qp1}--\eqref{nD},
or, more specifically, \eqref{qp1}--\eqref{nD} evaluated at collocation
nodes lying on the faces.
For these nodes we use the nodes from a $M_1$-by-$M_1$ product Gaussian quadrature on each rectangular face,
with $M_1$ chosen large enough that further changes have no effect on
the solution.
We pick $M_1$ to be $M_1 \approx 4 \frac{k}{\pi} $ so that we can guarantee
$4$ points per wavelength in each direction of the rectangular faces.

The resulting full system has the form
\be
\begin {bmatrix}
  A& \vline & \tilde{B} & 0 & 0  \\
	\hline 
	C_{L,R}  &\vline& S_{L,R} & 0 & 0 \\
	C_{nL,nR} & \vline& S_{nL,nR} & 0 & 0 \\
	C_{B,F}  &\vline &S_{B,F} & 0 & 0 \\
	C_{nB,nF}& \vline & S_{nB,nF} & 0 & 0 \\
	C_{T} &\vline & S_{T}  & -W_T & 0 \\
	C_{nT} & \vline & S_{nT}  & -W_{nT} & 0 \\
	C_{D}  &\vline  & S_{D}  & 0 & -W_D \\
	C_{nD} &\vline &  S_{nD}  & 0 & -W_{nD} \\
\end {bmatrix} 
 \begin {bmatrix} \eta \\ \hline \mathbf d \\ \mathbf a \\ \mathbf b \end{bmatrix}
=   \begin {bmatrix}
        \hat{f}\\
        0\\
        0\\
        0\\
        0\\
        0\\
        0\\
        0\\
        0
\end {bmatrix}
~,
\qquad\mbox{ summarized as }
\quad
\mt{A}{B}{C}{Q}
\vt{\eta}{\xi}
=\vt{\hat{f}}{0}~,
\label{full}
\ee
where $\xi := [\mbf{d};\mbf{a};\mbf{b}]$ groups the periodizing unknowns,
and where the division of the matrix blocks defining the $2\times2$
block notation is shown by lines.

The recipe for filling the above blocks is implicit in the above description,
and rather than give their full formulae (see \cite{mlqp} for full formulae
in a related 2D problem), we explain their meaning.
The $A$ block is already given in \eqref{A}.
The $B$ block is merely $\tilde{B}$ padded with zeros to its right.
The $C$ block describes the effect of the
ring kernel coefficients $c_{nj}$ in $\eta$
on the {\em discrepancies}, namely the left-hand sides of
\eqref{qp1}--\eqref{qp4}, and on the Cauchy data on $T$ and $D$.
Significant cancellation occurs in the upper four $C$ blocks,
identical to that in \cite{BG_jcp}.
Consider the $3\times3$ grid of nearest neighbors in phased image sums such as
\eqref{Phinear}, in the block $C_{L,R}$:
the effect of the left-most six on the $L$ wall cancels the effect of the right-most six on the $R$ wall, leaving only 6 of the original 18 terms.
The remaining terms correspond only to {\em distant interactions}:
the effect of the right-most three on $L$ and the left-most three on $R$.
Identical cancellations occur in the next three blocks of $C$.
As discussed in \cite{BG_jcp}, this also allows the scheme
to work well even if $\Omega$ is not confined within $\UU$
and wall-obstacle intersections occur.
The $C$ matrix is filled densely once and for all at each $k$,
at a cost $\bigO(M_1^2NP)$.

$Q$ has $8M_1^2$ rows and $N_Q := (\pdeg+1)^2 + \bigO(N_0^2)$ columns;
note that this is independent of $\NN$, the number of obstacle unknowns.
However, both dimensions of $Q$ grow with wavenumber as $\bigO(k^2)$.
We fill its non-zero blocks densely by evaluation of spherical harmonics
(for the $S$ blocks) and plane-wave expansions (for the $W$ blocks).
The $S$ blocks give the effect of the spherical harmonic basis on
the discrepancies and Cauchy data on $T$ and $D$, while
the $W$ blocks give the effect of the Rayleigh--Bloch expansions on
and Cauchy data on $T$ and $D$. The negative signs in the latter
account for the fact that the jumps in Cauchy data should vanish.

The transmission case is similar to the above, with
$A$ twice the size in each dimension (with $A_0$ as given in \eqref{trans}), and $[\hat{f};\hat{f}']$ replacing
$\hat{f}$ in the RHS.
Note that the representation of the interior potential $u^-$ only involves $A_0$, i.e.\ a non-periodized ring kernel.

In the next section we will demonstrate
convergence with respect to the parameters $\pdeg$ and $N_0$.

\begin{rmk}
Proxy source points have recently been proposed in a similar periodizing
scheme for the 3D Laplace equation \cite{gumerov}.
However, in App.~A we show that for Helmholtz problems
spherical harmonics are a much more efficient choice than
proxy points.
\end{rmk}

\subsection{Rapid solution of the linear system}
\label{s:iter}

Since we expect $\NN = NP$,
the number of columns of $A$, to be $10^4$ or greater,
the linear system \eqref{full} is too large to solve directly
(in contrast to related 2D work \cite{BG_jcp,BG_BIT}).
Hence we wish to apply an iterative method for the obstacle
ring source unknowns $\eta$.
We eliminate the smaller number of unknowns in $\xi$,
taking a Schur complement of \eqref{full}, to give
$$
(A - B Q^+ C) \eta = \hat{f}
$$
where $Q^+$ is the Moore--Penrose pseudoinverse of $Q$.
This is a rectangular system with poor conditioning similar to that
of the axisymmetric MFS system $A_0\eta = \hat{f}$ in \eqref{A0}.
Note that $A - B Q^+ C$ computes the one-obstacle interaction $A_0$
but with the quasiperiodic Green's function \eqref{GQP};
this explains why the scheme breaks down at Wood anomalies.
Using \eqref{A} and
right-preconditioning by $A_0^+$, we would get the square system
$$
(A_0 A_0^+ + \Aelse A_0^+ - B Q^+ C A_0^+) \tilde{\eta} = \hat{f}
~,
$$
from which we can recover the solution $\eta = A_0^+ \tilde{\eta}$.
However, since $M > N$ and $A_0$ is often
rank-deficient in MFS applications, $A_0$ has less than full range.
$A_0 A_0^+$ is the orthogonal projector onto
the range of $A_0$.
To create a well-conditioned system which can be solved iteratively,
we replace $A_0 A_0^+$ by the identity, since this has no effect
on the resulting desired $\eta$, solving
\be
(I + \Aelse A_0^+ - B Q^+ C A_0^+) \tilde{\eta} = \hat{f}
\label{iter}
\ee
via a non-symmetric Krylov method such as GMRES.
In effect we are working in a space of surface unknowns $\tilde{\eta}$
rather than ring charge source unknowns $\eta$.
Once $\eta$ is known, $\xi = [\mbf{d};\mbf{a};\mbf{b}]$ is reconstructed
via
$$
\xi = -Q^+ C \eta~,
$$
and the solution potential $u$ can then be evaluated anywhere in $\UU$
using $\eta$ and $\mbf{d}$.
The solution in $|z|>z_0$ can be evaluated via \eqref{rb1}
using $\mbf{a}$ or \eqref{rb2} using $\mbf{b}$.

A practical word is needed about handling the pseudoinverses.
Because of axisymmetry
$A_0$ is block diagonal, so each block $A_n$ can be inverted independently
by taking its SVD to give $A_n = U_n \Sigma_n V_n^\ast$.
Care must be taken to apply each $A_n^+$ correctly, otherwise a large loss of
accuracy results: to compute $A_n^+x$ for some $x\in \mathbb{C}^M$,
one uses $V_n \Sigma_n^+ (U_n^\ast x)$.
The truncation parameter used in $\Sigma_n^+$ is taken as $10^{-10}$.
Since there are $P$ blocks, 
the precomputation required for $A_0^+$ is $\bigO(PN^3)$, then each
application takes $\bigO(PN^2)$ with a very small constant.
Note that filling $A_n^+$ then using it for matrix-vector multiplication
would be dangerous since it is not backward stable; $A_n^+$ may
have exponentially large elements which induce catastrophic round-off error
(see comments in \cite[Sec.~5]{junlai} and \cite[Sec.~3.2]{mlqp}).

Similar care is needed for $Q$:
a dense SVD gives $Q = U\Sigma V^\ast$.
Assuming that $y = C A_0^+ \eta$ has already been computed as above,
$Q^+ y$ computed as $V \Sigma^+ (U^\ast y)$.
The SVD of $Q$ takes $\bigO(M_1^2N_Q^2)$ time.
Fixing the obstacle, the numerical parameters
$N$, $P$, $\pdeg$, $N_0$, and $M_1$ all grow as $\bigO(k)$.
Thus the time for the SVD of $Q$ grows rapidly with
wavenumber as $\bigO(k^6)$.
This limits the largest $k$ in practice on a standard
workstation to one in which the unit
cell is around a dozen wavelengths in period.

\begin{rmk}  
There are other fast solution methods for the linear system \eqref{full}.
For instance, one could instead eliminate $\eta$
if a fast application of $A^{-1}$ were available
(this is done in the analogous 2D periodic scattering
problem in \cite{Gillman_Barnett_jcp}).
\end{rmk}

\bfi 
(a)\raisebox{-1.8in}{\ig{width=2.5in}{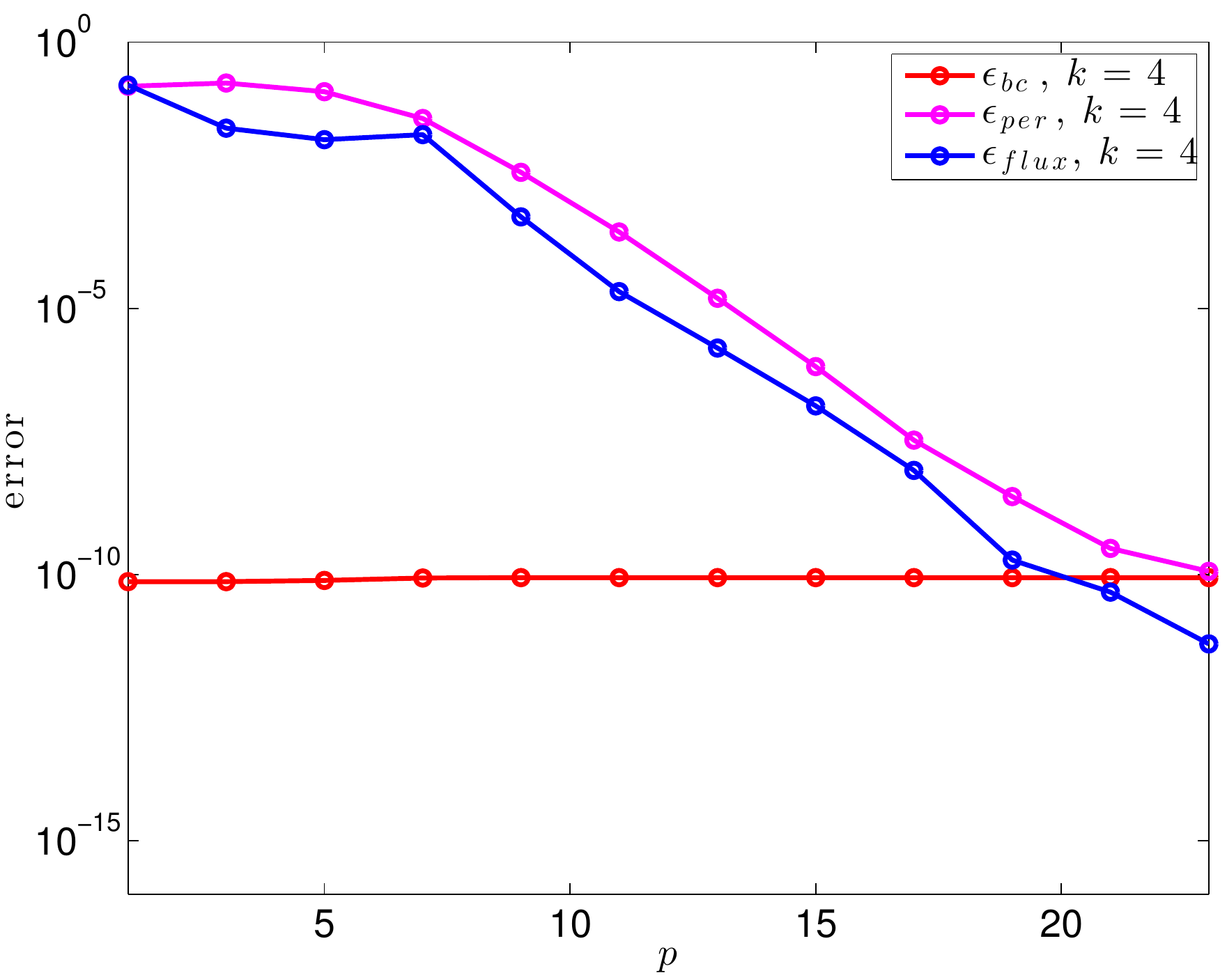}}
\qquad\qquad
(b)
\raisebox{-1.8in}{\ig{width=2.5in}{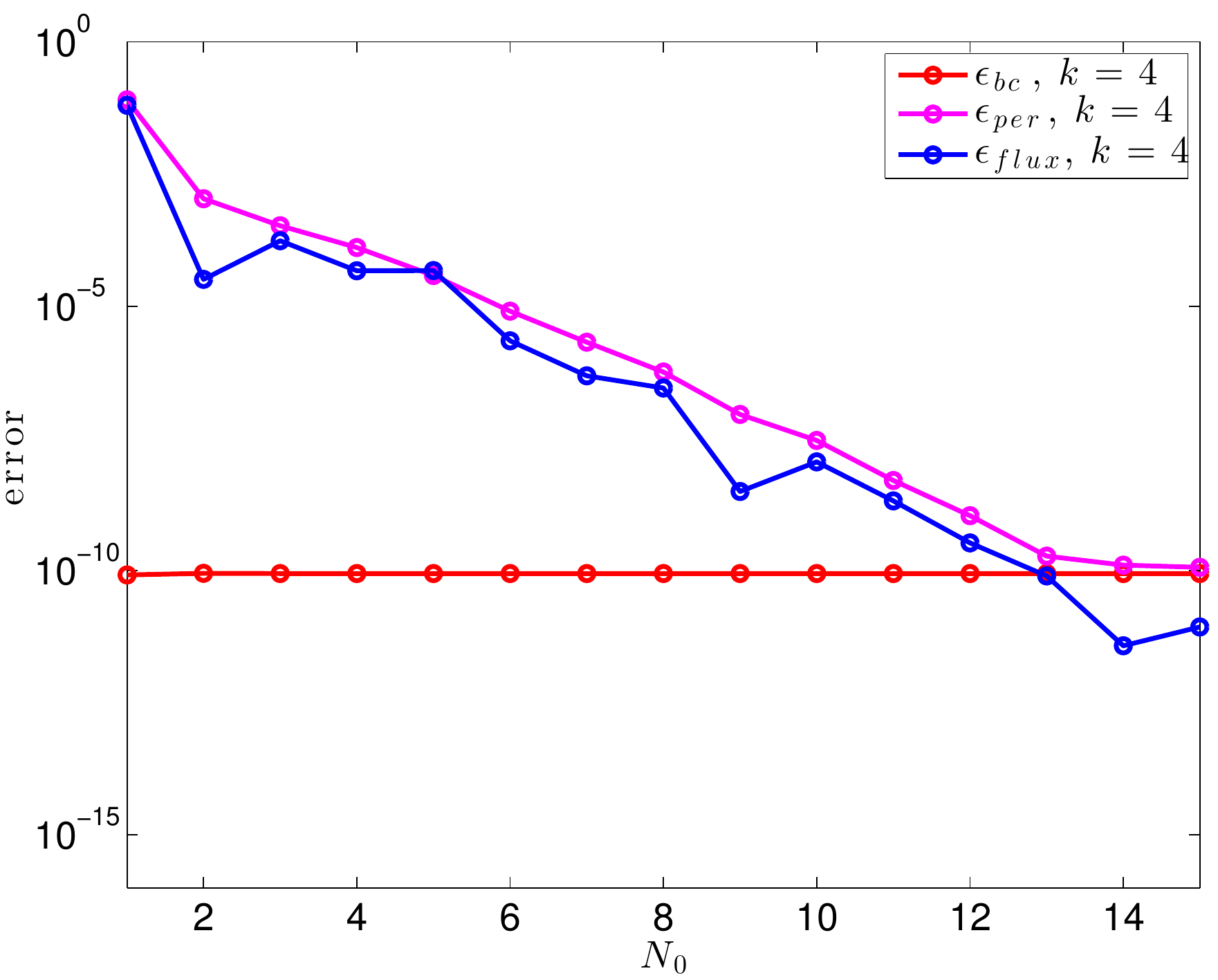}}
\ca{Convergence of errors for the periodizing scheme
for the Neumann scattering from a grating of ``smooth'' objects
as in \fref{f:shapes}(a), at low frequency.
Three types of error are shown: $L_2$ norm for the boundary condition, $L_2$ norm for the periodicity matching
flux conservation, and 
(a) error vs $\pdeg$, fixing $N_0 = 15$; (b) error vs $N_0$ when $\pdeg = 24$.%
}{f:perconv}
\efi

\begin{table}
\begin{tabular}{l*{13}{c}r}
BVP/shape & $k$ & $N$ & $P$ &  $p$  & $N_0$& $M_1$&fill&factor& solve& \# iters& $\ebc$ & $\eper$&$\eflux$&RAM\\ \hline\hline
Neum/smooth & 4 & 150 & 60 &  24 &13  & 24& 22 s &4.8 s & 34 s  & 12  & 9e-11 &  2e-10 &  8e-11  & 2GB\\
Trans/smooth & 4,6 & 150 & 60  & 24& 13& 24 &26 s & 6.6 s& 36 s & 12  &1e-11 & 1e-10&   2e-11  & 2.5GB \\
\end{tabular}
\caption{Low-frequency periodic scattering example, for the
``smooth'' shape of \fref{f:shapes}(a) in a unit cell $1.3 \lambda$ in period.
The columns show wavenumbers ($k$, and $\km$ when appropriate),
numerical parameters, timings, three error metrics, and total RAM usage.
The two rows are for Neumann boundary condition, and transmission condition.
$N$ is the number of MFS source points,
$P$ is the number of Fourier modes, $p$ the maximum degree of the
auxiliary spherical harmonic basis, $N_0$ the maximum order of the
Rayleigh--Bloch expansion.
The column ``fill'' reports the time to fill the matrices, i.e.\ $A_0, B, C$ and $Q$;  ``factor'' reports the factorization time, i.e.\ doing the SVD on the matrix blocks of $A_0$ and on $Q$; while ``solve'' reports the iterative
solver time.
\label{t:reslo}
}
\end{table}

\begin{table}
\hspace{-2ex}
\begin{tabular}{l*{13}{c}r}
BVP/shape & $k$ & $N$ & $P$ &  $p$  & $N_0$ & $M_1$&fill&factor& solve & \# iters& $\ebc$ & $\eper$&$\eflux$&RAM\\ \hline\hline
Neum/cup & 30 & 240 & 150 &  70 &21 &38 & 167 s & 297 s & 346 s  & 57  & 2e-10 &  5e-11 &  4e-11  & 17GB\\
Neum/cup & 40 & 360 & 180 &  86 &25 &45 &420 s  &871 s & 736 s &  65 &2e-10 &2e-10  &  4e-11  & 41GB\\
Trans/wiggly & 30,40 & 360 & 200  & 70& 21& 38& 440 s& 322 s& 756 s & 62  &7e-10 & 4e-11&   9e-11  & 58GB \\
Trans/wiggly & 40,60 & 400 & 200  & 86&25 &45 &700 s &934 s& 1090 s &  65 &1e-10 & 2e-10 & 1e-11    & 93GB \\
\end{tabular}
\caption{Higher-frequency periodic scattering examples ($10 \lambda$ and
$13 \lambda$ in period).
The first two rows use Neumann boundary conditions on the resonant cup shape,
and $\tau=0.03$;
the last two rows are for transmission conditions on the ``wiggly''
shape of \fref{f:shapes}(b) and $\tau=0.1$.
Other notation is as in Table~\ref{t:reslo}.
\label{t:reshi}
}
\end{table}

\begin{figure} 
\hspace{-5ex}\mbox{
(a)
\raisebox{-2.5in}{\ig{width=3.2in}{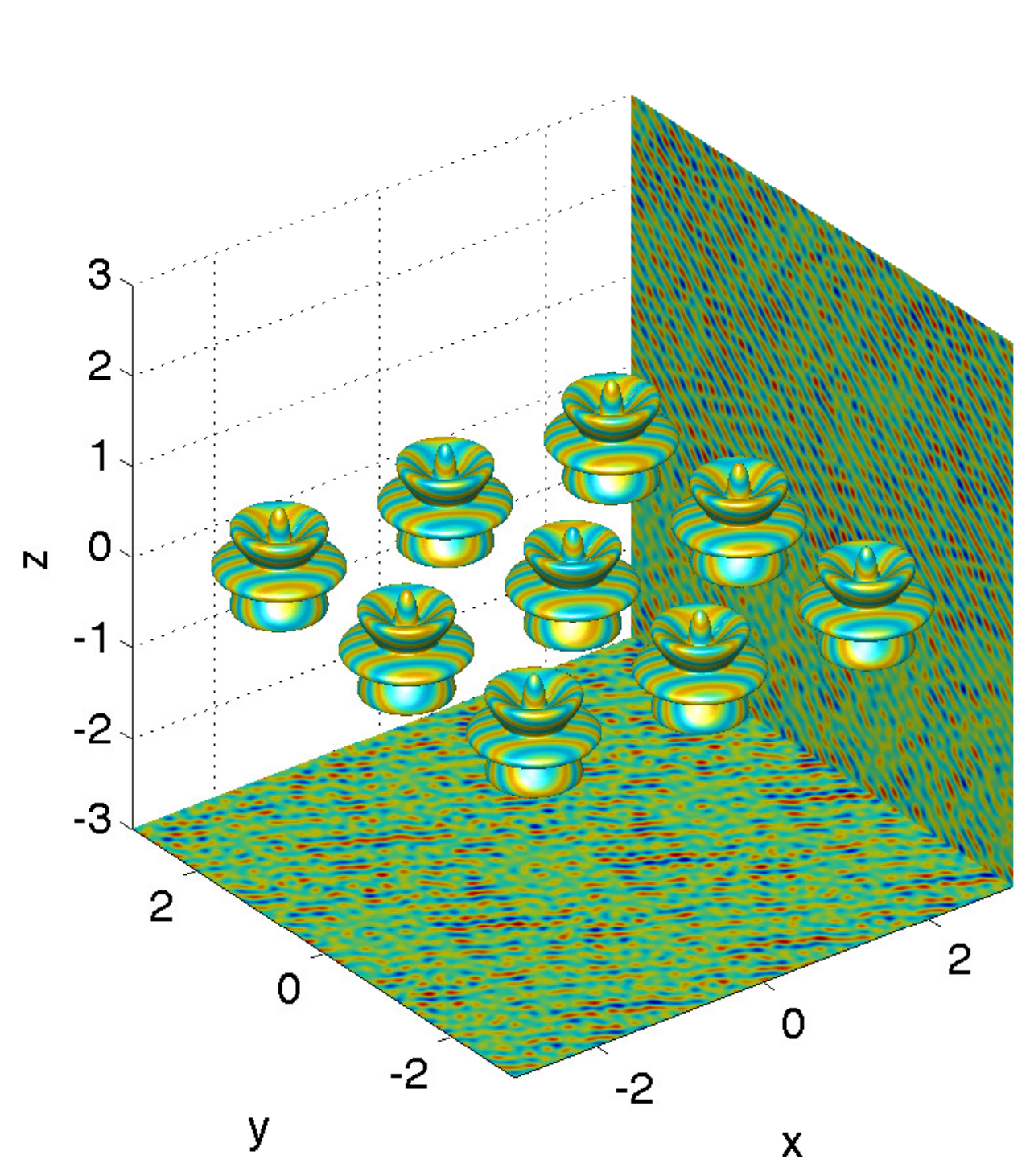}}
(b)
\raisebox{-2.5in}{\ig{width=3.2in}{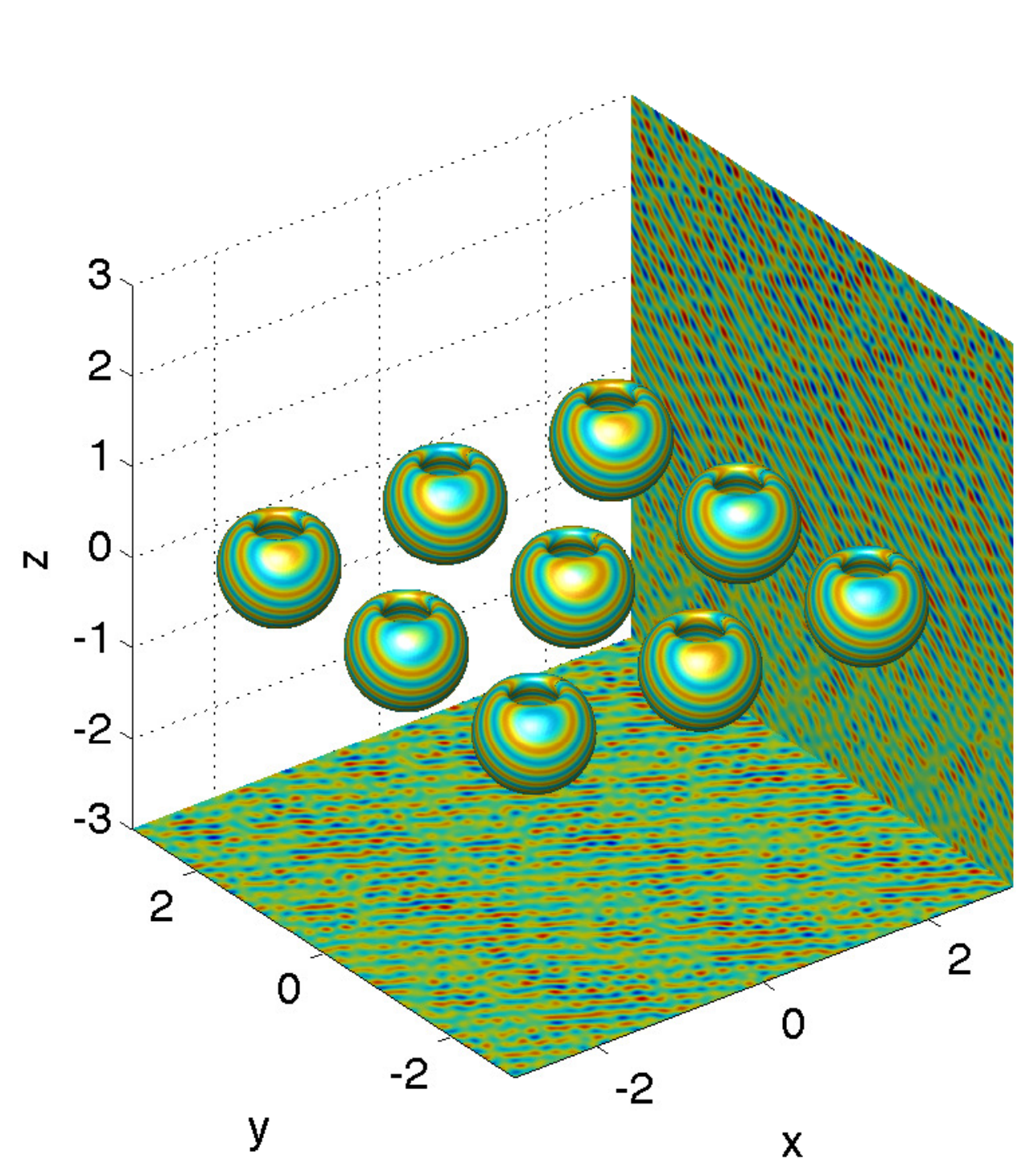}}
}
\caption{(a) High-frequency periodic transmission scattering solution with $k = 40,\km = 60$, for the ``wiggly'' shape of \fref{f:shapes}(b).
(b) High-frequency periodic Neumann scattering solution with $k=40$,
for the cup shape.
In both cases the full wave $\ut$ is shown on two slices,
the incident wave $\ui$ shown on the obstacle surfaces,
and the periodicity is $13 \lambda$.
The direction of the incident wave is $\theta = -\frac{\pi}{4}, \phi = \frac{\pi}{3}$, in spherical coordinates.
\label{f:plots}
}
\end{figure}

\section{Results} 
\label{s:res}

We first present some details of our implementation.
We use MATLAB R2013b on a desktop workstation
with two quad-core E5-2643 CPUs and 128 GB of RAM.
The tolerance for GMRES (MATLAB's implementation) is set to $10^{-12}$.
We apply $\Aelse$ via the FMM,
using MEX interfaces to
the Fortran implementation by Gimbutas--Greengard \cite{HFMM3D},
with the following modification:
we set the internal parameter {\tt maxlevel=3} in the routine
{\tt d3tstrcr},
which has the effect of limiting the depth of building the oct-tree
to 2.
Since all sources are well-separated from all targets in $\Aelse$,
this bypasses time spent propagating the source multipoles
up the tree via M2M, and local expansions L2L down the tree to the targets.
The result is a factor 1--3 speed-up over the vanilla FMM call
from this library.
We set {\tt iprec=3} which requests 9 digits of accuracy in the FMM.
Other spherical harmonic evaluations required for $B$, $C$ and $Q$ are
done using a MEX interface to the recurrence relation based fortran
libraries in \cite{HFMM3D}, available at\\
{\tt http://math.dartmouth.edu/$\sim$ahb/software/localexp3d.tgz}

The set of MATLAB codes we developed for the tests in this paper
will be available at\\
{\tt https://math.dartmouth.edu/$\sim$yliu/software/acper.tgz}

\begin{rmk}
For the $\NN$ we test in this work, the dense matrix blocks $B$ and $C$
fit in RAM, and thus we fill them once then apply then via standard BLAS2
matrix-vector multiplies.
If they cannot fit in RAM, then they can be applied on the fly using the FMM
with only a constant factor change in effort.
The (smaller) $Q$ matrix must be stored and inverted densely in our scheme.
\end{rmk}

Errors for the periodic scattering problem are measured in three
different ways:
\bi
\item $\ebc$: an estimate in the $L^2$ error in satisfying the boundary
condition on $\pO$ 
(defined as $\epsilon_1$ was in Sec.~\ref{s:mfs}),
\item $\eper$: an estimate of the $L^2$ error in satisfying the
periodic boundary conditions on the unit cell walls $L$, $R$, $F$ and $B$,
and
\item $\eflux$: the flux error giving the absolute value of the
difference between incoming and outgoing fluxes,
$$
\eflux : = \biggl|
\sum_{\kz>0} \kz \bigl(|a_{mn}|^2 + \bigl|b_{mn} + e^{-i\kappa_z^{(0,0)} z_0}\delta_{m0}\delta_{n0}\bigr|^2\bigr)
- \kappa_z^{(0,0)}
\biggr|
$$
Note that the incident wave corresponds to Bragg orders $m=n=0$, and the
phase shift is needed because the reference for $b_{mn}$ is at
$z=-z_0$.
For the Neumann or transmission BVPs with real-valued $k$ and $\km$
(non-absorbing materials), this error should be zero.
\ei

There are several numerical parameters, but they fall into two categories.
The one-body solution method is controlled by
$N$, $M$, $P$, $\tau$, and $q$ (convergence with respect to these
being shown in \sref{s:mfs}),
whereas the periodizing scheme is controlled by
$\pdeg$, $N_0$, and $M_1$ (the latter being fixed at $4k/\pi$ as explained
in \sref{s:sys}).
We thus first test convergence with respect to the periodizing parameters
$\pdeg$ and $N_0$.
\fref{f:perconv} shows error convergence consistent with
exponential, in the scattering from
a grating of ``smooth'' obstacles from \fref{f:shapes}(a)
at low frequency.
Note that $\ebc$ is small throughout the parameter range,
and thus cannot alone be used to verify that the full periodic solution
has converged.
Choosing the converged parameters $\pdeg = 24$ and $N_0=13$, the
timing and error results are then given in 
\tref{t:reslo}, for the two types of BVP.
There are $\NN = NP = 9000$ obstacle unknowns,
at the borderline where the FMM starts to become useful.
Here the size of $Q$ is 3200 by 1387, making it rapid to invert (via SVD).
Thus at low frequency, the computation is dominated by
the FMM application of $\Aelse$ needed in each GMRES iteration.

We show error and timing results at higher frequencies in \tref{t:reshi}.
The wavenumber $k=40$, i.e.\ a unit cell period of $13\lambda \times 13\lambda$,
is around the largest that is practical on a single workstation,
needing many minutes to take the SVD of the $Q$ matrix of size 
16200 by 11493.
The SVD (factorization) and GMRES (solve) stages now take comparable times.
The solution (total wave $\ut = u + \ui$) for the two highest frequency
cases are shown on planes in \fref{f:plots}.
For evaluation of $u$ we used the FMM for the first term
in \eqref{urep} and direct spherical harmonic summation for the second.
Note that the number of GMRES iterations is similar (around 60)
for the transmission case as for the highly-resonant Neumann cup shape.
The latter models a grating of acoustic Helmholtz resonators.
This illustrates that, because the one-obstacle problem is
factorized, the iterative part of the scheme is immune to
obstacle resonances, in contrast to the case where
GMRES is used to solve the entire scattering problem.

\section{Conclusions}
\label{s:con}
We have presented an acoustic
solver for doubly-periodic gratings of smooth axisymmetric obstacles,
that is spectrally accurate with respect to all of its convergence
parameters, enabling high (10-digit) accuracies
to be reached efficiently even at medium to high frequencies.
It combines an existing axisymmetric
one-body solution using the method of fundamental solutions (MFS) with a
new periodizing scheme based on auxiliary spherical harmonics,
upward and downward radiation expansions,
and collocation on the walls of one unit cell.
The result avoids any singular surface quadratures,
is efficient for grating periods up to a dozen wavelengths,
handles highly resonant obstacles without extra cost, and is
relatively simple to code.
The one-body factorization cost is $\bigO(PN^3)$
with a small prefactor; in the high-frequency
limit this would scale as $\bigO(k^4)$.
However, the solve cost per new incident wave is then
dominated by only the $\bigO(NP)$ linear cost of an FMM call, i.e.\
$\bigO(k^2)$.
A key feature is that the periodization scheme is independent of
the one-body solver,
so that the latter could be replaced by
an existing boundary integral based solver \cite{gillman3d,bremer3d,bruno01}
without modification,
allowing the handling of more general shapes, corners, and edges.

Although our main contribution is the new periodizing scheme,
we also have contributed improvements to the one-body MFS scheme.
Firstly, we use boundary complexification to
place the MFS sources; the choice of distance parameter $\tau$
may be made cheaply using simulations via the analogous 2D BVP.
Secondly we give a rigorous analysis of the use of the periodic trapezoid
rule for evaluating Helmholtz ring kernels.
Unlike in boundary integral equation methods (which require arbitrarily close
source-target evaluations), all of our kernels can be evaluated in this
way because of the source separation in the MFS.
Since applying $(A-B Q^+ C)$ is equivalent to applying
the quasiperiodic Green's function, our method as presented cannot work
at Wood anomalies.
However, because the Wood singularity is
of inverse square-root form, the neighborhood around which a Wood anomaly
causes a loss of accuracy is small,
and this accuracy loss is consistent with backwards stability
given the accuracy with which the incident wavevector is specified.
We note that, were the solver to be
extended to handle connected interfaces, robust handling Wood anomalies
would then come for free (as in \cite{mlqp}).

Future extensions that would increase the range of application of
the solver include:
automating the choice of all convergence parameters given a shape and
wavenumber $k$;
including MFS source point choices that handle axisymmetric edges \cite{larrythesis};
generalizing to multilayer media, as in \cite{mlqp};
replacing the eight-neighbor FMM call
by a low-rank compression scheme for applying the $\Aelse$ interactions, which currently dominate the solution time \cite{larrythesis};
extension to the Maxwell equations
\cite{larrythesis};
and replacing the MFS scheme by a boundary integral
solver, such as the recent
$O(\NN^{3/2})$ fast direct solver
\cite{gillman3d}, to create a 3D version of \cite{qpfds}.

\section*{Acknowledgments} 
The authors thank Fridon Shubitidze for advice about the MFS,
Arvind Saibaba for his
suggestion for changing from left- to right-preconditioning with the GMRES,
and Leslie Greengard and Zydrunas Gimbutas for helpful discussions about tuning
the FMM. We also were helped by the comments of the anonymous referees.
We are grateful for funding by the National Science Foundation
under grant DMS-1216656, and by the Neukom Institute at Dartmouth College.

\begin{figure} 
\centering
\ig{width=3in}{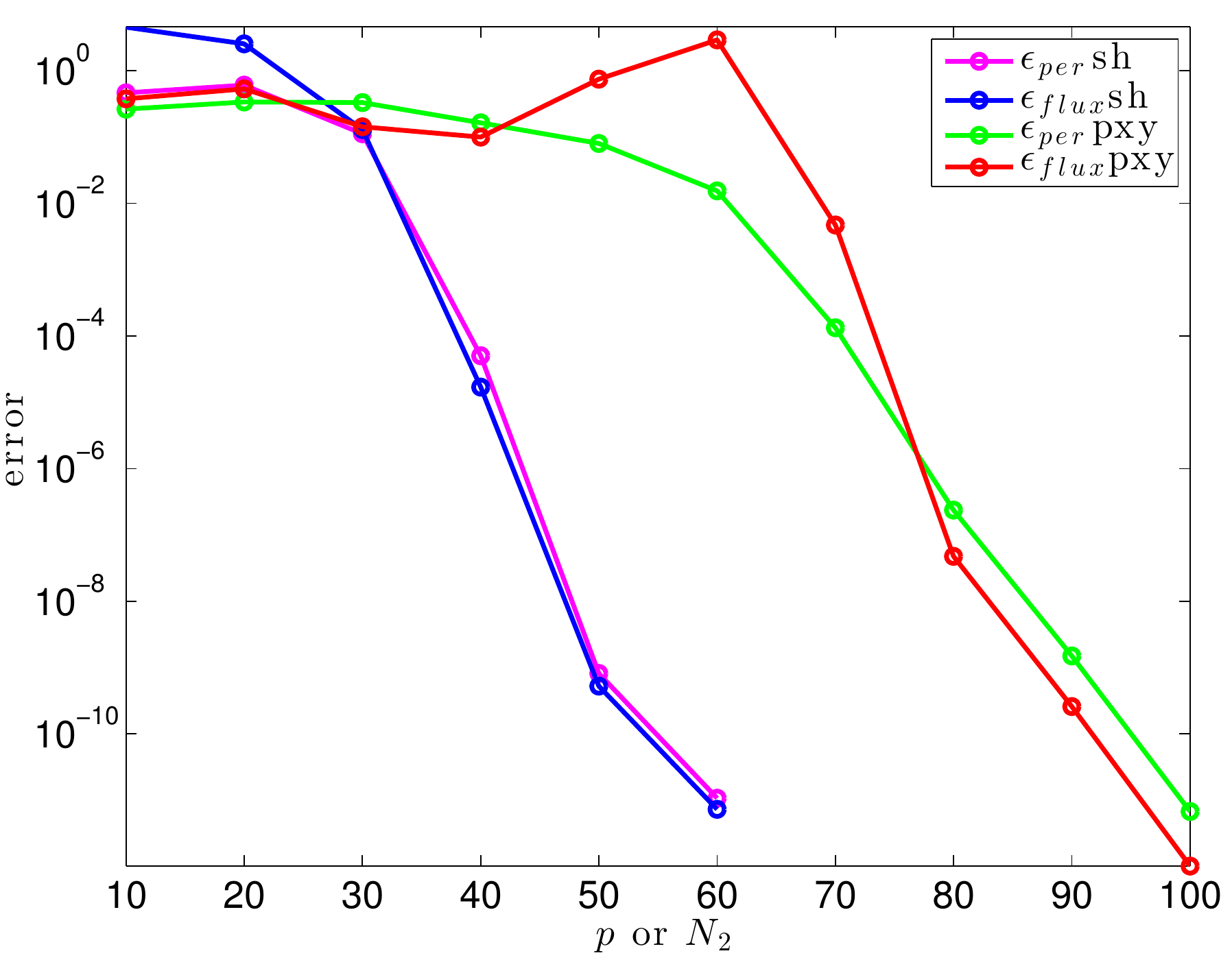}
\caption{%
Convergence comparison of spherical harmonics (``sh'') vs proxy points
(``pxy'')
as the auxiliary basis for the periodizing scheme,
for the Neumann scattering problem from a lattice of ``smooth'' obstacles
as in \fref{f:shapes}(a), at wavenumber $k=20$.
Note that the two schemes have nearly identical numbers of unknowns
when $p=N_2$.
We have fixed the Rayleigh--Bloch degree as $N_0 = 17$.}
\label{f:comp}
\end{figure}

\appendix
\section{Comparing spherical harmonics vs proxy points as a periodizing basis}

Our periodizing scheme represents the contribution
of the lattice of distant copies of the obstacle to the potential
with a basis expansion for regular Helmholtz solutions in the unit cell.
Here we compare two choices of this basis:
the $(\pdeg+1)^2$ spherical harmonics up to degree $\pdeg$ (as used in the rest of this work),
vs $N_2^2$ proxy points placed along lines of longitude on a distant sphere
of radius $R$.
The latter has been proposed in the context of periodizing
the 3D Laplace equation by Gumerov--Duraiswami \cite{gumerov}.
The radius $R$ used is $R = 3.5$, chosen to optimize proxy point efficiency.
The errors we test are $\eper$ and $\eflux$ defined in \sref{s:res},
for the Neumann scattering from a grating of ``smooth'' obstacles from \fref{f:shapes}(a)
at wavenumber $k=20$.
The period is around $7 \lambda$. 
We show in \fref{f:comp} the convergence results,
and see that
$N_2 \approx 2\pdeg$ is needed to achieve similar errors in the two
representations,
i.e.\ the proxy scheme requires four times the number of unknowns
required by the spherical harmonic scheme.
(Note that efficiency differences due to the different proxy point arrangements
discussed in \cite{gumerov} are small compared to this factor.)
We also checked that the solution potential $u$ at the point
$(0.9,0.9,0.9)$ agreed between the two methods
at their converged parameters with a relative error of $1.2 \times 10^{-11}$.
This is consistent with the errors shown in the figure, and provides
some verification of the correctness of each scheme.
In terms of speed,
for comparable 12-digit errors, the SVD of the $Q$ matrix for
$\pdeg=58$ spherical harmonics takes 217 s, whereas using
$N_2=100$ proxy points takes 1097 s, around five times slower.
Thus we claim that spherical harmonics are by far the better choice for
periodizing Helmholtz problems.

\bibliographystyle{amsplain}
\bibliography{liu}

\providecommand{\bysame}{\leavevmode\hbox to3em{\hrulefill}\thinspace}
\providecommand{\MR}{\relax\ifhmode\unskip\space\fi MR }
\providecommand{\MRhref}[2]{%
  \href{http://www.ams.org/mathscinet-getitem?mr=#1}{#2}
}
\providecommand{\href}[2]{#2}
\begin{thebibliography}{10}

\bibitem{Anand_Corner}
A.~Anand, J.~S. Ovall, and C.~Turc, \emph{Well-conditioned boundary integral
  equations for two-dimensional sound-hard scattering problems in domains with
  corners}, J. Integral Equations Appl. \textbf{24} (2012), no.~3, 321--358.

\bibitem{arenshabil}
T.~Arens, \emph{Scattering by biperiodic layered media: The integral equation
  approach}, Ph.D. thesis, Karlsruhe Institute of Technology, 2010.

\bibitem{arens10}
T.~Arens, K.~Sandfort, S.~Schmitt, and A.~Lechleiter, \emph{Analysing {E}wald's
  method for the evaluation of {G}reen's functions for periodic media}, IMA J.
  Numer. Anal. \textbf{78} (2013), no.~3, 405--431.

\bibitem{atwater}
H.~A. Atwater and A.~Polman, \emph{Plasmonics for improved photovoltaic
  devices}, Nature Materials \textbf{9} (2010), no.~3, 205--213.

\bibitem{pollution}
I.~M. Babuska and S.~A. Sauter, \emph{Is the pollution effect of the {FEM}
  avoidable for the {H}elmholtz equation considering high wave numbers?}, SIAM
  J. Numer. Anal. \textbf{34} (1997), no.~6, 2392--2423.

\bibitem{bao95}
G.~Bao, \emph{Finite element approximation of time harmonic waves in periodic
  structures}, SIAM J. Numer. Anal. \textbf{32} (1995), no.~4, 1155--1169.

\bibitem{mfs}
A.~H. Barnett and T.~Betcke, \emph{Stability and convergence of the {M}ethod of
  {F}undamental {S}olutions for {H}elmholtz problems on analytic domains}, J.
  Comput. Phys. \textbf{227} (2008), no.~14, 7003--7026.

\bibitem{BG_jcp}
A.~H. Barnett and L.~Greengard, \emph{A new integral representation for
  quasi-periodic fields and its application to two-dimensional band structure
  calculations}, J. Comput. Phys. \textbf{229} (2010), 6898--6914.

\bibitem{BG_BIT}
\bysame, \emph{A new integral representation for quasi-periodic scattering
  problems in two dimensions}, BIT Numer. Math. \textbf{51} (2011), no.~1,
  67--90.

\bibitem{qp3dqbx}
A.~H. Barnett, L.~Greengard, and Z.~Gimbutas, \emph{Efficient and robust
  integral equation methods for acoustic scattering from doubly-periodic media
  in three dimensions}, 2015, in preparation.

\bibitem{beriot16}
H~B\'eriot, A~Prinn, and G~Gabard, \emph{Efficient implementation of high-order
  finite elements for {H}elmholtz problems}, Int.\ J.\ Numer.\ Meth.\ Engng
  \textbf{106} (2016), 213--240.

\bibitem{Fridon_Optical}
A.~Bijamov, F.~Shubitidze, P.~M. Oliver, and D.~V. Vezenov, \emph{Optical
  response of magnetic fluorescent microspheres used for force spectroscopy in
  the evanescent field}, Langmuir \textbf{26} (2010), 12003--12011.

\bibitem{bonnetBDS}
A.-S. Bonnet-BenDhia and F.~Starling, \emph{Guided waves by electromagnetic
  gratings and non-uniqueness examples for the diffraction problem}, Math.
  Meth. Appl. Sci. \textbf{17} (1994), 305--338.

\bibitem{gillman3d}
J.~Bremer, A.~Gillman, and P.-G. Martinsson, \emph{A high-order accurate
  accelerated direct solver for acoustic scattering from surfaces}, BIT Numer.
  Math. \textbf{55} (2015), no.~2, 367--397.

\bibitem{bremer3d}
J.~Bremer and Z.~Gimbutas, \emph{A {Nystr\"om} method for weakly singular
  integral operators on surfaces}, J. Comput. Phys. \textbf{231} (2012),
  4885--4903.

\bibitem{brunohaslam09}
O.~P. Bruno and M.~C. Haslam, \emph{Efficient high-order evaluation of
  scattering by periodic surfaces: deep gratings, high frequencies, and
  glancing incidences}, J. Opt. Soc. Am. A \textbf{26} (2009), no.~3, 658--668.

\bibitem{bruno01}
O.~P. Bruno and L.~A. Kunyansky, \emph{Surface scattering in three
  dimensions:}, Proc. R. Soc. Lond. A \textbf{457} (2001), 2921--2934.

\bibitem{brunoqp3d}
O.~P. Bruno, S.~Shipman, C.~Turc, and S.~Venakides, \emph{Efficient evaluation
  of doubly periodic {G}reen functions in {3D} scattering, including {W}ood
  anomaly frequencies}, 2013, preprint, {\tt arXiv:1307.1176v1}.

\bibitem{Chen13}
W.~Chen, J.~Lin, and C.~S. Chen, \emph{the method of fundamental solutions for
  solving exterior axisymmetric {H}elmholtz problems with high wave-number},
  Adv.\ Appl.\ Math.\ Mech. \textbf{5} (2013), no.~4, 477--493.

\bibitem{fmm1}
H.~Cheng, W.~Y. Crutchfield, Z.~Gimbutas, L.~Greengard, F.~Ethridge, J.~Huang,
  V.~Rokhlin, N.~Yarvin, and J.~Zhao, \emph{A wideband fast multipole method
  for the {H}elmholtz equation in three dimensions}, J. Comput. Phys.
  \textbf{216} (2006), 300--325.

\bibitem{mlqp}
M.~H. Cho and A.~H. Barnett, \emph{Robust fast direct integral equation solver
  for quasi-periodic scattering problems with a large number of layers}, Opt.
  Express \textbf{23} (2015), no.~2, 1775--1799.

\bibitem{CK83}
D.~Colton and R.~Kress, \emph{Integral equation methods in scattering theory},
  Wiley, 1983.

\bibitem{coltonkress}
\bysame, \emph{Inverse acoustic and electromagnetic scattering theory}, second
  ed., Applied Mathematical Sciences, vol.~93, Springer-Verlag, Berlin, 1998.

\bibitem{cohl}
J.~T. Conway and H.~S. Cohl, \emph{Exact fourier expansion in cylindrical
  coordinates for the three-dimensional {H}elmholtz {G}reen's function}, Z.
  Angew. Math. Phys. \textbf{61} (2010), 425--442.

\bibitem{acousticmeta}
R.~V. Craster and S.~Guenneau, \emph{Acoustic metamaterials: Negative
  refraction, imaging, lensing and cloaking}, Springer Series in Materials
  Science (166), Springer, 2013.

\bibitem{davis59}
P.~J. Davis, \emph{On the numerical integration of periodic analytic
  functions}, Proceedings of a Symposium on Numerical Approximations (R.~E.
  Langer, ed.), University of Wisconsin Press, 1959.

\bibitem{DCD1}
D.~C. Dobson, \emph{Optimal design of periodic antireflective structures for
  the {H}elmholtz equation}, Euro. J. Appl. Math. \textbf{4} (1993), 321--340.

\bibitem{doicu}
A.~Doicu, Y.~A. Eremin, and T.~Wriedt, \emph{Acoustic and electromagnetic
  scattering analysis using discrete sources}, Academic Press, San Diego, CA,
  2000.

\bibitem{handbkac}
F.~A. Everest and K.~C. Pohlmann, \emph{Mast handbook for acoustics, 5th ed.},
  McGraw-Hill, 2009.

\bibitem{ewald}
P.~P. Ewald, \emph{Die berechnung optischer und elektrostatischer
  gitterpotentiale}, Ann. Phys. \textbf{64} (1921), 253--287.

\bibitem{mfs_review}
G.~Fairweather and A.~Karageorghis, \emph{The method of fundamental solutions
  for elliptic boundary value problems}, Adv. Comput. Math. \textbf{9} (1998),
  no.~1-2, 69--95.

\bibitem{FDTDimp}
K.-Y. Fung and H.~Ju, \emph{Time domain impedance boundary conditions}, Int. J.
  Comput. Fluid D. \textbf{18} (2004), no.~6, 503--511.

\bibitem{qpfds}
A.~Gillman and A.~H. Barnett, \emph{A fast direct solver for quasiperiodic
  scattering problems}, J. Comput.\ Phys. \textbf{248} (2013), 309--322.

\bibitem{Gillman_Barnett_jcp}
\bysame, \emph{A fast direct solver for quasiperiodic scattering problems}, J.
  Comput.\ Phys. \textbf{248} (2013), 309--322.

\bibitem{HFMM3D}
Z.~Gimbutas and L.~Greengard, \emph{{FMMLIB3D}, {F}ortran libraries for fast
  multipole method in three dimensions}, 2012, {\tt
  http://www.cims.nyu.edu/cmcl/fmm3dlib/fmm3dlib.html}.

\bibitem{fmm_maxwell}
\bysame, \emph{Fast multi-particle scattering: A hybrid solver for the
  {M}axwell equations in microstructured materials}, J. Comput. Phys.
  \textbf{232} (2013), 22--32.

\bibitem{fmm_Greengard}
L.~Greengard and V.~Rokhlin, \emph{A fast algorithm for particle simulations},
  Journal of Computational Physics \textbf{135} (1997), no.~2, 280--292.

\bibitem{multisphere}
N.~A. Gumerov and R.~Duraiswami, \emph{Computation of scattering from clusters
  of spheres using the fast multipole method}, J. Acoust. Soc. Am. \textbf{117}
  (2005), no.~4, 1744--1761.

\bibitem{gumerov}
\bysame, \emph{A method to compute periodic sums}, J. Comput. Phys.
  \textbf{272} (2014), 307--326.

\bibitem{FDTD3DBC}
J.~H\"aggblad and B.~Engquist, \emph{Consistent modeling of boundaries in
  acoustic finite-difference time-domain simulations}, J. Acoust. Soc. Am.
  \textbf{132} (2012), no.~3, 1303--1310.

\bibitem{hao3daxi}
S.~Hao, P.-G. Martinsson, and P.~Young, \emph{An efficient and highly accurate
  solver for multi-body acoustic scattering problems involving rotationally
  symmetric scatterers}, Comput. Math. Appl. \textbf{69} (2015), no.~4,
  304--318.

\bibitem{helsing_axi}
J.~Helsing and A.~Karlsson, \emph{An explicit kernel-split panel-based
  {Nystr\"om} scheme for integral equations on axially symmetric surfaces}, J.
  Comput. Phys. \textbf{272} (2014), 686--703.

\bibitem{helsingIEEE15}
\bysame, \emph{Determination of normalized magnetic eigenfields in microwave
  cavities}, IEEE Trans. Microw. Theory Tech. \textbf{63} (2015), 1457--1467.

\bibitem{Hcorner}
A.~Hochman, Y.~Leviatan, and J.~K. White, \emph{On the use of rational-function
  fitting methods for the solution of 2{D} {L}aplace boundary-value problems},
  J. Comput. Phys. \textbf{238} (2013), 337--358.

\bibitem{jobook}
J.~D. Joannopoulos, S.~G. Johnson, R.~D. Meade, and J.~N. Winn, \emph{Photonic
  crystals: Molding the flow of light}, 2nd ed., Princeton Univ. Press,
  Princeton, NJ, 2008.

\bibitem{jordan86}
K.~E. Jordan, G.~R. Richter, and P.~Sheng, \emph{An efficient numerical
  evaluation of the {G}reen's function for the {H}elmholtz operator on periodic
  structures}, J. Comput. Phys. \textbf{63} (1986), 222--235.

\bibitem{jorgenson}
R.~E. Jorgenson and R.~Mittra, \emph{Efficient calculation of the free space
  periodic {G}reen's function}, IEEE Trans. Antennas Propagat. \textbf{38}
  (1990), 633--642.

\bibitem{kangro2d}
U.~Kangro, \emph{Convergence of collocation method with delta functions for
  integral equations of first kind}, Integr. Equ. Oper. Theory \textbf{66}
  (2010), no.~2, 265--282.

\bibitem{kangro3d}
\bysame, \emph{Solution of three-dimensional electromagnetic scattering
  problems by interior source methods}, AIP Conf. Proc. \textbf{1479} (2012),
  2328--2331.

\bibitem{karagaxi}
A.~Karageorghis and G.~Fairweather, \emph{The method of fundamental solutions
  for axisymmetric acoustic scattering and radiation problems}, J. Acoust. Soc.
  Am. \textbf{104} (1998), no.~6, 3212--3218.

\bibitem{Ka90}
M.~Katsurada, \emph{Asymptotic error analysis of the charge simulation method
  in a {J}ordan region with an analytic boundary}, J. Fac. Sci. Univ. Tokyo
  Sect. IA Math. \textbf{37} (1990), no.~3, 635--657. \MR{MR1080874
  (91m:65263)}

\bibitem{Katsurada1}
M.~Katsurada and H.~Okamoto, \emph{A mathematical study of the charge
  simulation method {I}}, J. Fac. Sci. Univ. Tokyo \textbf{35} (1988),
  507--518.

\bibitem{kress95}
R~Kress, \emph{On the numerical solution of a hypersingular integral equation
  in scattering theory}, J. Comput.\ Appl.\ Math. \textbf{61} (1995), 345--360.

\bibitem{LIE}
R.~Kress, \emph{Linear integral equations}, second ed., Appl.\ Math.\ Sci.,
  vol.~82, Springer, 1999.

\bibitem{mfs_Kupradze}
V.~D. Kupradze and M.~A. Aleksidze, \emph{The method of functional equations
  for the approximate solution of certain boundary value problems}, Comput.
  Math. Math. Phys. \textbf{4} (1964), no.~4, 82 -- 126.

\bibitem{Ky96}
A.~G. Kyurkchan, B.~Y. Sternin, and V.~E. Shatalov, \emph{Singularities of
  continuation of wave fields}, Physics - Uspekhi \textbf{12} (1996),
  1221--1242.

\bibitem{junlai}
J.~Lai, M.~Kobayashi, and A.~H. Barnett, \emph{A fast and robust solver for the
  scattering from a layered periodic structure with multi-particle inclusions},
  2014, {\tt arXiv:1412.7466}, in review, {\em J. Comput. Phys.}

\bibitem{lintonrev}
C.~M. Linton, \emph{Lattice sums for the {H}elmholtz equation}, SIAM Review
  \textbf{52} (2010), no.~4, 603--674.

\bibitem{linton07}
C.~M. Linton and I.~Thompson, \emph{Resonant effects in scattering by periodic
  arrays}, Wave Motion \textbf{44} (2007), 165--175.

\bibitem{larrythesis}
Y.~Liu, \emph{The numerical solution of frequency-domain acoustic and
  electromagnetic periodic scattering problems}, 2016, Ph.D. thesis, Dartmouth
  College. Available at
  \verb+https://math.dartmouth.edu/~yliu/thesis_YuxiangLiu.pdf+.

\bibitem{Malcolm14}
A.~Malcolm and D.~P. Nicholls, \emph{Operator expansions and constrained
  quadratic optimization for interface reconstruction: Impenetrable periodic
  acoustic media}, Wave Motion \textbf{51} (2014), 23--40.

\bibitem{CWnystrom}
A.~Meier, T.~Arens, S.~N Chandler-Wilde, and A.~Kirsch, \emph{A {Nystr\"{o}m}
  method for a class of integral equations on the real line with applications
  to scattering by diffraction gratings and rough surfaces}, J. Integral
  Equations Appl. \textbf{12} (2000), 281--321.

\bibitem{scatterometry}
R.~Model, A.~Rathsfeld, H.~Gross, M.~Wurm, and B.~Bodermann, \emph{A
  scatterometry inverse problem in optical mask metrology}, J. Phys.: Conf.
  Ser. \textbf{135} (2008), 012071.

\bibitem{noisecontrol}
M.~M{\"o}ser, \emph{Engineering acoustics: An introduction to noise control},
  Springer-Verlag, 2004.

\bibitem{poroelas}
B.~Nennig, E.~Perrey-Debain, and J.-D. Chazot, \emph{The method of fundamental
  solutions for acoustic wave scattering by a single and a periodic array of
  poroelastic scatterers}, Eng. Anal. Boundary Elements \textbf{35} (2011),
  1019--1028.

\bibitem{nicholas}
M.~J. Nicholas, \emph{A higher order numerical method for {3-D} doubly periodic
  electromagnetic scattering problems}, Commun. Math. Sci. \textbf{6} (2008),
  no.~3, 669--694.

\bibitem{otani08}
Y.~Otani and N.~Nishimura, \emph{A periodic {FMM} for {M}axwell's equations in
  {3D} and its applications to problems related to photonic crystals}, J.
  Comput. Phys. \textbf{227} (2008), 4630--52.

\bibitem{gmres}
Y.~Saad and M.~H. Schultz, \emph{{GMRES}: a generalized minimal residual
  algorithm for solving nonsymmetric linear systems}, SIAM J. Stat. Sci.
  Comput. \textbf{7} (1986), no.~3, 856--869.

\bibitem{shipmanreview}
S.~Shipman, \emph{Resonant scattering by open periodic waveguides}, Progress in
  Computational Physics (PiCP), vol.~1, pp.~7--50, Bentham Science Publishers,
  2010.

\bibitem{Fridon_MAS}
F.~Shubitidze, K.~O'Neill, S.~A. Haider, K.~Sun, and K.~D. Paulsen,
  \emph{Application of the method of auxiliary sources to the wide-band
  electromagnetic induction problem}, IEEE Trans. Geosci. Remote Sensing
  \textbf{40} (2002), 928--942.

\bibitem{taflove}
A.~Taflove, \emph{Computational electrodynamics: The finite-difference
  time-domain method}, Artech House, Norwood, MA, 1995.

\bibitem{PTRtref}
L.~N. Trefethen and J.~A.~C. Weideman, \emph{The exponentially convergent
  {T}rapezoidal rule}, SIAM Review \textbf{56} (2014), no.~3, 385--458.

\bibitem{ying06}
L.~Ying, G.~Biros, and D.~Zorin, \emph{A high-order {3D} boundary integral
  equation solver for elliptic {PDE}s in smooth domains}, J. Comput. Phys.
  \textbf{216} (2006), 247--275.

\bibitem{Young}
P.~Young, S.~Hao, and P.-G. Martinsson, \emph{A high-order {Nystr\"om}
  discretization scheme for boundary integral equations defined on rotationally
  symmetric surfaces}, J. Comput. Phys. \textbf{231} (2012), no.~11,
  4142--4159.

\bibitem{acousticcloak}
S.~Zhang, C.~Xia, and N.~Fang, \emph{Broadband acoustic cloak for ultrasound
  waves}, Phys. Rev. Lett. \textbf{106} (2011), 024301.

\end{thebibliography}

\end{document}